\documentclass[leqno,11pt]{amsart}
\usepackage{amssymb, amsmath}
\usepackage{amssymb,amsfonts}
\usepackage{amsmath,latexsym}
\usepackage{pdfsync}
\usepackage{bbm}

\usepackage{color}
\newtheorem{theorem}{Theorem}[section]
\newtheorem{lemma}[theorem]{Lemma}

\newcommand{\N}{\mathbb{N}}
\newcommand{\Z}{\mathbb{Z}}

\newcommand{\p}{\partial}

\newcommand{\beqq}{\begin{eqnarray}}
\newcommand{\enqq}{\end{eqnarray}}

\newcommand{\enn}{\end{equation}}
\newcommand{\bef}{\begin{proof}}
\newcommand{\enf}{\end{proof}}
\let\al=\alpha
\let\b=\beta

\let\d=\delta
\let\e=\varepsilon

\let\la=\lambda

\let\f=\frac

\let\De=\Delta

\let\wh=\widehat

\let\th=\theta
\let\pa=\partial

\def\cF{{\cal F}}

\def\cM{{\mathcal M}}

\def\cF{{\mathcal F}}

\def\sh{\sqrt{\hbar}}
\def\rv{\rm v}
\def\ga{\gamma}
\def\vf{\varphi}

\def\h{{\rm h}}
\def\v{{\rm v}}

\def\R{\mathbf R}
\def\Z{\mathbf Z}

\newcommand{\andf}{\quad\hbox{and}\quad}
\newcommand{\with}{\quad\hbox{with}\quad}
\def\Supp{\mathop{\rm Supp}\nolimits\ }

\def\longformule#1#2{
\displaylines{ \qquad{#1} \hfill\cr \hfill {#2} \qquad\cr } }
\def\inte#1{
\displaystyle\mathop{#1\kern0pt}^\circ }
\newcommand{\w}[1]{\langle {#1} \rangle}
\def\no{\noindent}
\def\ff{\frak{f}}

\def\D{\langle D_x\rangle}
\def\t{\langle t\rangle}
\def\s{\langle s\rangle}

\def\eqdef{\buildrel\hbox{\footnotesize def}\over =}

\def\eqdefa{\buildrel\hbox{\footnotesize def}\over =}

\newcommand{\beq}{\begin{equation}}
\newcommand{\eeq}{\end{equation}}
\newcommand{\ben}{\begin{eqnarray}}
\newcommand{\een}{\end{eqnarray}}
\newcommand{\beno}{\begin{eqnarray*}}
\newcommand{\eeno}{\end{eqnarray*}}

\newtheorem{thm}{Theorem}[section]
\newtheorem{lem}{Lemma}[section]
\newtheorem{rmk}{Remark}[section]
\newtheorem{col}{Corollary}[section]
\newtheorem{prop}{Proposition}[section]
\renewcommand{\theequation}{\thesection.\arabic{equation}}

\begin{document}
\title[Global small solution of 2-D Prandtl system]{On the global small solution of 2-D Prandtl system with initial data in the optimal Gevrey class}

\author{Chao Wang}
\address{School of Mathematical Science, Peking University, Beijing 100871,
China} \email{wangchao@math.pku.edu.cn}

\author{Yuxi Wang}
\address{School of Mathematics, Sichuan University, Chengdu 610064, China}
\email{wangyuxi@scu.edu.cn}

\author{Ping Zhang}%
\address{\footnote{Corresponding author} Academy of
Mathematics $\&$ Systems Science and  Hua Loo-Keng Key Laboratory of
Mathematics, The Chinese Academy of Sciences, Beijing 100190, China, and School of Mathematical Sciences,
University of Chinese Academy of Sciences, Beijing 100049, China.
} \email{zp@amss.ac.cn}

\date{\today}

\maketitle

\begin{abstract} Motivated by \cite{DG19}, we prove the global existence and large
time behavior  of small solutions to 2-D Prandtl system for data with Gevrey 2 regularity in the $x$ variable
and  Sobolev regularity in the $y$ variable. In particular, we extend the global well-posedness result in \cite{PZ5} for
2-D Prandtl system
with analytic data to data with optimal Gevery regularity in the sense of \cite{Ger1}.
\end{abstract}

\vspace{0.1cm}


\renewcommand{\theequation}{\thesection.\arabic{equation}}
\setcounter{equation}{0}
\section{Introduction}

In this paper, we investigate the global existence and the large time behavior  of solutions to the following
 2-D Prandtl system with small initial data in the  Gevrey class:
\begin{align}\label{eq: Prandtl}
\left\{
\begin{aligned}
&\pa_t u+u\pa_x u+v\pa_y u-\pa_y^2 u+\pa_x P^{\rm E}=0, \quad (t,x,y)\in\R_+\times\R\times\R_+,\\
&\pa_x u+\pa_y v=0,\\
&(u,v)|_{y=0}=0,\quad \lim_{y\to +\infty}u=U^{\rm E}(t,x),\\
&u|_{t=0}=u_0(x,y),
\end{aligned}
\right.
\end{align}
where $(u,v)$ stand for the tangential and normal velocities of the boundary flow, and $(U^{\rm E}(t,x),$ $ P^{\rm E}(t,x))$ designate
the trances on the boundary of the tangential velocity and pressure of the outflow, which satisfies Bernoulli's law
\beno
\pa_t U^{\rm E}+U^{\rm E}\pa_x U^{\rm E}+\pa_x P^{\rm E}=0.
\eeno

 The system \eqref{eq: Prandtl} is the foundation of the boundary layer theory. It was proposed by Prandtl \cite{Pra} in 1904 in order to explain
 the disparity between the boundary conditions verified by  ideal fluid and viscous fluid with small viscosity.
Heuristically, these
 boundary layers are of amplitude $O(1)$ and of  thickness  $O(\sqrt{\nu})$ where there is a transition from the interior flow governed by
 Euler equation to the Navier-Stokes flow with a vanishing viscosity $\nu> 0$.  One may check \cite{E, Olei} and
references therein for more introductions on boundary layer theory. Especially we refer to \cite{Guo} for a comprehensive  recent survey.
\smallskip

Compared with 2-D Navier-Stokes equations,
 there is no
horizontal diffusion in the $u$ equation of \eqref{eq: Prandtl}, and the normal velocity $v$ is determined by
the divergence free condition of the velocity field $(u,v),$ so that
the
nonlinear term $v\p_y u$  almost behaves like $-\p_xu\p_y u.$  Moreover, this term is not skew-symmetric in the Sobolev space
$H^s.$ As a consequence,
 one horizontal derivative is lost in the process of energy estimate.
Therefore the question of whether or not the Prandtl system with
general data is well-posed in Sobolev spaces is still open.
Physically, the difficulty of the above problem is related to the underlying instabilities of the boundary layer,
for instance, the phenomenon of separation (see \cite{DM20,SWZ}). \smallskip

The first well-posedness result of \eqref{eq: Prandtl} goes back to  Oleinik \cite{Olei66}. Under the assumption that
$u_0(x,y)$ in \eqref{eq: Prandtl} is monotonic in the $y$ variable,  she first introduced Crocco transformation
and then proved the local
existence and uniqueness of classical solutions to \eqref{eq: Prandtl}. This result was revisited in \cite{Alex} and  \cite{MW}
by performing energy estimates in weighted Sobolev spaces. While
 with
the additional ``favorable" condition on the pressure, Xin and Zhang
\cite{Xin} obtained the global existence of weak solutions to this
system. We remark that such monotonicity assumption excludes the phenomenon of reverse flow.

In general, G\'ervard-Varet and Dormy \cite{Ger1}  proved the
ill-posedness in Sobloev spaces for the linearized Prandtl system
around non-monotonic shear flows. The nonlinear ill-posedness was
also established in \cite{Ger2, Guo} in the sense of non-Lipschtiz
continuity of the flow.
Yet for the data which is analytic in both $x$ and $y$ variables,
Sammartino and Caflisch \cite{Caf} established the local
well-posedness result of \eqref{eq: Prandtl} (see also \cite{KV13}). The analyticity in $y$
variable was removed by Lombardo, Cannone and Sammartino in
\cite{Can}. The main argument used in  \cite{Can, Caf} is to apply
the abstract Cauchy-Kowalewskaya (CK) theorem.

To relax the analyticity condition for the initial data turns out to be very hard.  In the special
case when $u_0$ has for each value of $x$ a single non-degenerate critical point in $y,$
G\'ervard-Varet and Masmoudi \cite{GM} proved the
 well-posedness of \eqref{eq: Prandtl} for a class of data with
Gevrey class $\f74$ in the $x$ variable. Chen, the second author and Zhang \cite{CWZ18} proved the
well-posedness of the linearized Prandtl equation around a non-monotonic shear flow in Gevery class $2-\theta$ for any $\theta>0.$
Li and Yang \cite{LY20} extended the result in \cite{GM} for data that are small perturbations of a shear flow with a single non-degenerate
critical point, which belongs to the Gevery class 2. Eventually,  Dietert and G\'ervard-Varet \cite{DG19} removed the
structure assumption in \cite{LY20} and proved the local well-posedness of \eqref{eq: Prandtl} for data
with Gevrey 2 regularity in $x$ variable and Sobolev regularity in $y$ variable.
 This result is optimal in the sense that  G\'ervard-Varet and Dormy \cite{Ger1}
constructed a class of solution with the growth like $e^{\sqrt{k}t}$ for the linearized Prandtl system around
a non-monotonic shear flow, where $k$ is the tangential frequency.

 \smallskip

On the other hand, as pointed out by Grenier, Guo, and Nguyen \cite{GGN15, GGN16, GGN17} (see also \cite{DR04}), in order to make progress towards proving or disproving the inviscid limit of the Navier-Stokes equations, one must understand its behavior on a longer time interval than the one which causes the instability used to prove ill-posedness.   Oleinik and Samokhin \cite{Olei} also asked the following question (see open problem {\bf 4} on page
 500 of \cite{Olei}): ``{\rm It has been shown in Chapter 4 that under certain assumptions the system of
 nonstationary two-dimensional boundary layer admits one and only one solution in the domain $D=\left\{ 0<t<T, 0<x<X, 0<y<\infty \right\}$
 either for small enough $T$ and any $X>0$ or small enough $X$ and any $T>0.$ What are the conditions ensuring the existence
 and uniqueness of a solution of the nonstationary Prandtl system in the domain $D$ with arbitrary $X$ and $T?$}"

Indeed the third author and Zhang first addressed
  the long time existence for Prandtl system with small analytic data  in \cite{ZHZ} and an almost global existence result was provided in \cite{IV16}. Lately, Paicu and the third author \cite{PZ5} succeeded in proving the global existence of small analytic
  solutions to \eqref{eq: Prandtl}. Furthermore, the large time behavior of the solutions have been obtained in \cite{PZ5}. Similar
  result was obtained for the MHD boundary layer equations in \cite{LnZ1}.
\smallskip

Motivated by \cite{DG19, PZ5}, we are going to study the global well-posedness of \eqref{eq: Prandtl}
with small initial data which have Gevery 2 regularity in $x$ variable and Sobolev regularity
  in $y$ variable.
For a clear presentation, we shall take $U^{\rm E}=0$ and $P^{\rm E}=C$ in the rest of this paper, so that $(u,v)$ solves
\begin{align}\label{eq: tu}
\left\{
\begin{aligned}
&\pa_t u+u\pa_x u+v\pa_y u-\pa_y^2 u=0,  \quad (t,x,y)\in\R_+\times\R\times\R_+,\\
&\pa_x u+\pa_y v=0,\\
&(u,v)|_{y=0}=0,\quad \lim_{y\to +\infty}u=0,\\
&u|_{t=0}=u_0(x,y).
\end{aligned}
\right.
\end{align}
We remark that as in \cite{PZ5},
we can take $U^{\rm E}(t,x)=\e f(t)$ for $f(t)$ decaying sufficiently fast at infinity.

Whereas due to the divergence free condition of the velocity field $(u,v),$ there exists a potential function $\vf$ such that $u=\pa_y \vf$ and $v=-\pa_x\vf$. Then we deduce from \eqref{eq: tu} (see \cite{PZ5} for the detailed derivation) that
\beno
\left\{
\begin{array}{ll}
\pa_t\vf+u\pa_x\vf+2\int_{y}^\infty (\pa_y u\pa_x\vf)dy'-\pa_y^2\vf=0,\\
\vf|_{y=0}=0,\\
\vf|_{t=0}=\vf_0.
\end{array}
\right.
\eeno

To derive the faster decay estimate of the solutions to \eqref{eq: tu},  which will be needed to control the
Gevrey radius of the solutions, we recall  from \cite{PZ5} the good quantity
\beq \label{S1eq3a}
G\eqdefa u+\f{y}{2\w{t}}\vf \andf g\eqdefa \p_yG= \p_yu+\f{y}{2\w{t}}u+\f\vf{2\w{t}}, \eeq  which
solves
\beq \label{S1eq1}
\left\{
\begin{array}{ll}
\p_tG-\p_y^2G+\w{t}^{-1}G+u\p_xG+v\p_yG\\
\quad-\f1{2\w{t}}v\p_y({y\vf})+\f{y}{\w{t}}\int_y^\infty\left(\p_y u\p_x\vf\right)\,dy'=0,\\
G|_{y=0}=0 \andf \lim_{y\to +\infty}G(t,x,y)=0,\\
G|_{t=0}=G_0\eqdefa u_0+\f{y}2\vf_0.
\end{array}
\right.
\eeq

Let $\w{\xi}\eqdefa \bigl(1+\xi^2\bigr)^\f12,$ we denote
\beq \label{S1eq1a}
e^{\Phi(t,D_x)}\eqdefa e^{\delta(t)\D^\f12} \with \delta(t)=\delta-\lambda\theta(t) \andf f_\Phi\eqdefa e^{\Phi(t,D_x)}f,
\eeq
and the weighted anisotropic Sobolev norm
\beq \label{S3eq-2}
\begin{split}
\|f\|_{H^{s,k}_\Psi}\eqdefa \sum_{0\leq\ell\leq k}\Bigl(\int_0^\infty e^{2\Psi(t,y)}\|\p_y^\ell{f}(\cdot,y)\|_{H^s_\h}^2\,dy\Bigr)^{\f12}
 \with  \Psi(t,y)\eqdefa \frac{y^2}{8\w{t}}.
 \end{split}
\eeq

The main result of this paper states as follows:

\begin{thm}\label{th1.1}
{\sl
 Let  $u_0=\p_y\vf_0$ satisfy $u_0(x,0)=0$ and  $\int_0^\infty u_0\,dy=0.$ Let  $G_0\eqdefa u_0+\f{y}2\vf_0.$
For some sufficiently small but fixed $\eta\in \bigl(0,\f16\bigr),$ we denote
 \beq \label{S9eq1}
 \begin{split}
{\rm E}(t)\eqdefa &
\|\t^{\f{1-\eta}{4}}u_\Phi\|_{H^{\f{11}{2},0}_\Psi}^2+\|\t^{\f{3-\eta}{4}}\pa_yu_\Phi\|_{H^{5,0}_\Psi}^2+\|\t^{\f{5-\eta}{4}}G_\Phi\|_{H^{4,0}_\Psi}^2\\
&+\|\t^{\f{7-\eta}{4}}\pa_yG_\Phi\|_{H^{3,0}_\Psi}^2+\|\t^{\f{9-\eta}{4}}\pa_y^2G\|_{H^{3,0}_\Psi}^2
+\|\t^{\f{11-\eta}{4}}\pa_y^3G\|_{H^{2,0}_\Psi}^2,
\end{split}
\eeq
and
 \beq \label{S9eq2}
 \begin{split}
{\rm D}(t)\eqdefa &\|\t^{\f{1-\eta}{4}}\pa_yu_\Phi\|_{H^{\f{11}{2},0}_\Psi}^2+\|\t^{\f{3-\eta}{4}}\pa_y^2u_\Phi\|_{H^{5,0}_\Psi}^2+\|\t^{\f{5-\eta}{4}}\pa_yG_\Phi\|_{H^{4,0}_\Psi}^2\\
&+\|\t^{\f{7-\eta}{4}}\pa_y^2G_\Phi\|_{H^{3,0}_\Psi}^2+\|\t^{\f{9-\eta}{4}}\pa_y^3G\|_{H^{3,0}_\Psi}^2
+\|\t^{\f{11-\eta}{4}}\pa_y^4G\|_{H^{2,0}_\Psi}^2.
\end{split}
\eeq
We assume moreover that
\beq \label{S9eq3}
\|u_\Phi(0)\|_{H^{\f{25}4,0}_\Psi}^2+{\rm E}(0)\leq \e^2.
\eeq
 Then
  there exists $\e_0, \la_0>0$ so that  for $\varepsilon\leq \e_0$ and $\la\geq \la_0,$
  $\th(t)$ in \eqref{S1eq1a} satisfies $\sup_{t\in [0,\infty)}\th(t)\leq \f\d{4\la},$ and
   the system (\ref{eq: tu}) has a  unique global solution $u$ which satisfies
  \beq \label{S9eq4}
  \rm{E}(t)+c\eta\int_0^t\rm{D}(t')\,dt'\leq C\e^2\quad \forall \ t\in\R.
  \eeq
}
\end{thm}

\begin{rmk}
\begin{itemize}
\item[(1)]  As in \cite{DG19}, there is a loss on the Gevrey radius of the solution $u.$
Moreover, we need $u_\Phi(0)\in H^{\f{25}4,0}_\Psi(\R^2_+)$ in order to guarantee that
$u_\Phi(t)\in H^{\f{11}2,0}_\Psi(\R^2_+).$ These loss of regularities is due to the instabilities described in \cite{DG19}.

\item[(2)] Compared with \cite{PZ5} with analytic data, here our smallness condition can not be
presumed only on the initial data of  the good quantity, $G.$  Moreover, the decay of Gevrey solutions
has a loss of $\f\eta4$ in \eqref{S9eq4}. We do not know if we could get rid of this loss or not.

\item[(3)]  The main idea to prove \eqref{S9eq4} is to combine the method in \cite{DG19}, which
is based on both a tricky change of unknown and appropriate choice of test function, and  the
time weighted energy estimate method
in \cite{PZ5,PZZ2}. Furthermore, as in \cite{PZ5},  we shall use the faster decay estimate of the good quantity,
 $G,$ in \eqref{S1eq3a}
to control the Gevrey radius  of the solutions, which will be one of the crucial step to construct
the global solutions of \eqref{eq: tu} with optimal Gevrey regular data.

\end{itemize}

\end{rmk}

\smallskip

 Let us end this introduction by the notations that we shall use
in this context.\vspace{0.2cm}

For~$a\lesssim b$, we mean that there is a uniform constant $C,$
which may be different on different lines, such that $a\leq Cb$. $(a\ |\
b)_{L^2_+}\eqdefa\int_{\R^2_+}a(x,y) {b}(x,y)\,dx\,dy$  stands for
the $L^2$ inner product of $a,b$ on $\R^2_+$ and
$L^p_+=L^p(\R^2_+)$  with $\R^2_+\eqdefa\R\times\R_+.$ For $X$ a Banach space
and $I$ an interval of $\R,$ we denote by $L^q(I;\,X)$ the set of
measurable functions on $I$ with values in $X,$ such that
$t\longmapsto\|f(t)\|_{X}$ belongs to $L^q(I).$  In particular,  we denote by
$L^p_T(L^q_{\rm h}(L^r_{\rm v}))$ the space $L^p([0,T];
L^q(\R_{x_{}};L^r(\R_{y}^+))).$

\renewcommand{\theequation}{\thesection.\arabic{equation}}
\setcounter{equation}{0}
%%%%%%%%%%%%%%%%%%%%%%%%%%%%%%%%%%%%%%%%%%%%%%
%%%%%%%%%%%%%%%%%%%%%%%%%%%%%%%%%%%%%%%%%%

\section{preliminary}\label{sect2}

\subsection{Littlewood-Paley theory}
For the convenience of the readers,  we shall collect some basic facts on anisotropic  Littlewood-Paley theory in this subsection.
Let us first recall from
\cite{BCD} that \beq
\begin{split}
&
S^{\rm h}_ka=\cF^{-1}(\chi(2^{-k}|\xi|)\widehat{a}) \andf \Delta_k^{\rm h}a=
\left\{
\begin{aligned}
&\cF^{-1}(\varphi(2^{-k}|\xi|)\widehat{a}) \quad \mbox{if}\ \ k\geq 0;\\
&S_0^\h a \qquad\qquad\qquad\ \ \mbox{if}\ \  k=-1,\\
&0 \qquad \quad\qquad\qquad\quad \mbox{if}\ \ k\leq -2,
\end{aligned}
\right.
\end{split} \label{1.3a}\eeq where and in all that follows, $\cF
a$ and $\widehat{a}$ always denote the partial  Fourier transform of
the distribution $a$ with respect to $x$ variable,  that is, $
\widehat{a}(\xi,y)=\cF_{x\to\xi}(a)(\xi,y),$
  and $\chi(\tau),$ ~$\varphi(\tau)$ are
smooth functions such that
 \beno
&&\Supp \varphi \subset \Bigl\{\tau \in \R\,/\  \ \frac34 \leq
|\tau| \leq \frac83 \Bigr\}\andf \  \ \forall
 \tau>0\,,\ \sum_{k\in\Z}\varphi(2^{-k}\tau)=1,\\
&&\Supp \chi \subset \Bigl\{\tau \in \R\,/\  \ \ |\tau|  \leq
\frac43 \Bigr\}\quad \ \ \ \andf \  \ \, \chi(\tau)+ \sum_{k\geq
0}\varphi(2^{-k}\tau)=1.
 \eeno

To deal with the estimate concerning the product of two distributions,
we shall frequently use  Bony's decomposition (see \cite{Bo}) in
the horizontal variable:
 \ben\label{Bony} fg=T^{\rm h}_fg+T^{\rm
h}_{g}f+R^{\rm h}(f,g), \een where \beno
\begin{split}
 T^{\rm h}_fg\eqdefa\sum_kS^{\rm
h}_{k-1}f\Delta_k^{\rm h}g,\quad R^{\rm
h}(f,g)\eqdefa&\sum_k\widetilde{\Delta}_k^{\rm h}f\Delta_{k}^{\rm h}g \with
\widetilde{\Delta}_k^{\rm h}f\eqdefa
\displaystyle\sum_{|k-k'|\le 1}\Delta_{k'}^{\rm h}f. \end{split} \eeno

Finally we  recall the definitions of norms to anisotropic Sobolev space.
For $s\in\R, k\in\N,$ we denote
\begin{align*} \|g\|_{H^s_\h}\eqdefa & \Bigl(\int_{\R}\bigl(1+|\xi|^2\bigr)^s|\widehat{g}(\xi)|^2\,d\xi\Bigr)^{\f12} \andf\\
\|f\|_{H^{s,k}}\eqdef &\sum_{0\leq\ell\leq k}\Bigl(\int_0^\infty\int_{\R}\bigl(1+|\xi|^2\bigr)^s|\p_y^\ell\widehat{f}(\xi,y)|^2\,d\xi\,dy\Bigr)^{\f12}.
\end{align*}

\subsection{Para-product related estimates} In what follows, we shall always assume that $t<T^\ast$ with
$T^\ast$ being determined by
\beq\label{1.8a}
T^\ast\eqdefa \sup\Bigl\{\ t>0,\ \ \th(t) <\f\d{2\la}\ \Bigr\}.
\eeq
So that by
virtue of \eqref{S1eq1a}, for any $t<T^\ast,$ $\delta(t)\geq \f\d2$ and there holds the
following convex inequality
\beq\label{1.8bb} \Phi(t,\xi)\leq
\Phi(t,\xi-\eta)+\Phi(t,\eta)\quad\mbox{for}\quad \forall\ \xi,\eta\in \R,
\eeq
with $\Phi(t,\xi)\eqdefa \delta(t)\w{\xi}^{\f12}$.

Then along the same line to the proof of the classical product laws in \cite{BCD}, which corresponds
to the cases when $\Phi(t,\xi)=0,$ we have

\begin{lemma}\label{lem: T_fg}
{\sl Let $s\in\R$  and
 $(T_f^\h)^*$ be the adjoint operator of $T_f^\h$. If $\sigma>\f12, $ we have
\begin{align*}
&\|(T_f^\h g)_\Phi\|_{H^s_\h}\leq C\|f_\Phi\|_{H_\h^\sigma}\|g_\Phi\|_{H^s_\h},\\
&\|\bigl((T^\h_f)^*g\bigr)_\Phi\|_{H^s_\h}\leq C\|f_\Phi\|_{H_\h^\sigma}\|g_\Phi\|_{H^s_\h}.
\end{align*}
If $s_1+s_2>s+\f12>0,$ there holds
\begin{align*}
\|(R^\h(f,g))_\Phi\|_{H^s_\h}\leq C\|f_\Phi\|_{H_{\h}^{s_1}}\|g_\Phi\|_{H^{s_2}_\h}.
\end{align*}
}
\end{lemma}

\begin{lemma}\label{lem: com1}
{\sl Let $s>0$ and $\sigma>\f32.$ Then one has
\begin{align*}
&\|\bigl((T_a^\h T_b^\h-T_{ab}^\h)f\bigr)_\Phi\|_{H^s_\h}\leq C\|a_\Phi\|_{H^\sigma_\h}\|b_\Phi\|_{H^\sigma_\h}\|f_\Phi\|_{H^{s-1}_\h},\\
&\|\bigl([\D^s; T_a^\h]f\bigr)_{\Phi}\|_{L^2_\h}\leq C\|a_\Phi\|_{H^\sigma_\h}\|f_\Phi\|_{H^{s-1}_\h},\\
&\|\bigl((T_a^\h-(T_a^\h)^*)f\bigr)_\Phi\|_{H^{s}_\h}\leq C\|a_\Phi\|_{H^\sigma_\h}\|f_\Phi\|_{H^{s-1}_\h},
\end{align*}
where $\D^s$ denotes the Fourier multiplier with symbol $\bigl(1+|\xi|^2\bigr)^{\f{s}2}$. We also have
\begin{align*}
\|([T_a^\h;T_b^\h]f)_\Phi\|_{H^s_\h}\leq C\|a_\Phi\|_{H^\sigma_\h}\|b_\Phi\|_{H^\sigma_\h}\|f_\Phi\|_{H^{s-1}_\h}.
\end{align*}
}
\end{lemma}

\begin{lemma}\label{lem: com2}
{\sl Let $\Phi(t,\xi)\eqdefa \d(t)\w{\xi}^{\f12}.$ Let $s\in\R$ and $\sigma>\f32.$
Then  one has
\begin{align}\label{S2eq5}
\|(T^\h_a\pa_xf)_{\Phi}-T_a^\h\pa_xf_{\Phi}\|_{H^s_\h}\leq C\d(t)\|a_\Phi\|_{H^\sigma_\h}\|f_\Phi\|_{H^{s+\f12}_\h}.
\end{align}
}
\end{lemma}
\begin{proof} We first write
\beq \label{S2eq1}
\wh{\bigl(T^\h_a\pa_xf\bigr)_\Phi}-\wh{T^\h_a\pa_xf_\Phi}=\sum_{k\geq 0}\int_{\R}m(t,\xi,\eta)\wh{S_{k-1}^\h a_\Phi}(\xi-\eta)
i\eta \wh{\De_k^\h f_\Phi}(\eta)\,d\eta,
\eeq
where $m(t,\xi,\eta)\eqdefa \bigl(e^{\Phi(t,\xi)}-e^{\Phi(t,\eta)}\bigr)e^{-\Phi(t,\xi-\eta)-\Phi(t,\eta)}$ with $\Phi(t,\xi)= \d(t)\w{\xi}^{\f12}.$
It is easy to observe that
\begin{align*}
e^{\Phi(t,\xi)}-e^{\Phi(t,\eta)}=&\d(t)\int_0^1 e^{\d(t)\left(\al|\xi|^{\f12}+(1-\al)|\eta|^{\f12}\right)}\,d\al\bigl(\w{\xi}^{\f12}-
\w{\eta}^{\f12}\bigr),
\end{align*}
and for any $\al\in [0,1],$
\beq \label{S2eq6}
\al\w{\xi}^{\f12}+(1-\al)\w{\eta}^{\f12}\leq \al\w{\xi-\eta}^{\f12}+\w{\eta}^{\f12}.
\eeq
So that
\beq \label{S2eq2}
\begin{split}
|m(t,\xi,\eta)|\leq &  \d(t)\frac{|\w{\xi}-\w{\eta}|}{\w{\xi}^{\f12}+\w{\eta}^{\frac12}}
\leq  \d(t)\frac{|\xi-\eta|}{\w{\xi}^{\f12}+\w{\eta}^{\frac12}}.
\end{split}
\eeq

In view of \eqref{S2eq2}, we find
\beno
\bigl|\int_{\R}m(t,\xi,\eta)\wh{S_{k-1}^\h a_\Phi}(\xi-\eta)
i\eta \wh{\De_k^\h f_\Phi}(\eta)\,d\eta\bigr|\leq \d(t)\int_{\R}|\xi-\eta||\wh{S_{k-1}^\h a_\Phi}(\xi-\eta)|
|\eta|^{\f12} \wh{\De_k^\h f_\Phi}(\eta)\,d\eta,
\eeno
from which, we infer
\begin{align*}
\bigl\|\int_{\R}m(t,\cdot,\eta)\wh{S_{k-1}^\h a_\Phi}(\cdot-\eta)
i\eta \wh{\De_k^\h f_\Phi}(\eta)\,d\eta\bigr\|_{L^2}\lesssim &\d(t)2^{\f{k}2}\|\wh{S_{k-1}^\h \pa_x a_\Phi}\|_{L^1}\|\De_k^\h f_\Phi\|_{L^2}\\
\lesssim &\d(t)c_k2^{-ks}\|a_\Phi\|_{H^\sigma}\| f_\Phi\|_{\dot H^{s+\f12}},
\end{align*}
where $\left(c_k\right)_{k\in\Z}$ designates a generic element of $\ell^2(\N)$ so that $\sum_{k\geq 0}c_k^2=1.$
Therefore, thanks to the support properties to the  terms in \eqref{S2eq1}, we conclude the proof of
\eqref{S2eq5}.
\end{proof}

\begin{lemma}\label{lem: com3}
{\sl Let $s\in\R,$  $\sigma>\f52$ and $D_x\eqdefa \f1i\pa_x.$ Let $\Phi(t,\xi)\eqdefa \d(t)\w{\xi}^{\f12},$  and $\Lambda(\xi)\eqdefa \xi(1+\xi^2)^{-\f34}.$ Then if $0<\d(t)\leq L,$ one has
\begin{align}\label{S2eq8}
\|(T_a^\h\pa_xf)_{\Phi}-T^\h_a\pa_xf_{\Phi}-\f{\d(t)}2T^\h_{D_x a}\Lambda(D)\pa_xf_\Phi\|_{H^s_\h}\leq C_L\|a_\Phi\|_{H^\sigma_\h}\|f_\Phi\|_{H^{s}_\h}.
\end{align}}
\end{lemma}

\begin{proof}
Similar to \eqref{S2eq1}, we write
\beq\label{S2eq3}
\begin{split}
\cF\Bigl((T_a^\h\pa_xf)_{\Phi}-&T^\h_a\pa_xf_{\Phi}-\f{\d(t)}2T^\h_{D_x a}\Lambda(D)\pa_xf_\Phi\Bigr)(t,\xi)\\
=&\sum_{k\geq 0}\int_{\R}\cM(t,\xi,\eta)\wh{S_{k-1}^\h a_\Phi}(\xi-\eta)i\eta\wh{\De_k^\h f_\Phi}(\eta)\,d\eta\with\\
\cM(t,\xi,\eta)\eqdefa &\bigl(e^{\Phi(t,\xi)}-e^{\Phi(t,\eta)}-\f{\d(t)}2\Lambda(\eta)(\xi-\eta)e^{\Phi(t,\eta)}\bigr)
e^{-\Phi(t,\xi-\eta)-\Phi(t,\eta)}.
\end{split}
\eeq
Notice that $\Phi(t,\xi)=\d(t)\w{\xi}^{\f12},$ we get, by applying Taylor's expansion, that
$$\longformule{
e^{\Phi(t,\xi)}-e^{\Phi(t,\eta)}=\d(t)e^{\d(t)\w{\eta}^{\f12}}\bigl(\w{\xi}^{\f12}-\w{\eta}^{\f12}\bigr)}{{}
+\f{\d^2(t)}2\int_0^1(1-\al)e^{\d(t)\bigl(\al\w{\xi}^{\f12}+(1-\al)\w{\eta}^{\f12}\bigr)}\,d\al\bigl(\w{\xi}^{\f12}-\w{\eta}^{\f12}\bigr)^2,}
$$
and
\beno
\w{\xi}^{\f12}-\w{\eta}^{\f12}=\f\eta2\bigl(1+\eta^2\bigr)^{-\f34}(\xi-\eta)+\f12\int_0^1(1-\al)\Lambda'(\al\xi+(1-\al)\eta)\,d\al(\xi-\eta)^2,
\eeno
from which and \eqref{S2eq6}, we infer
\begin{align*}
|\cM(t,\xi,\eta)|\leq & \f{\d(t)}2\int_0^1(1-\al)|\Lambda'(\al\xi+(1-\al)\eta)|\,d\al(\eta-\xi)^2+\f{\d^2(t)}2\bigl(\w{\xi}^{\f12}-\w{\eta}^{\f12}\bigr)^2\\
\leq &C_L\Bigl(\int_0^1(1-\al)|\Lambda'(\al\xi+(1-\al)\eta)|\,d\al(\eta-\xi)^2+\bigl(\w{\xi}^{\f12}-\w{\eta}^{\f12}\bigr)^2\Bigr).
\end{align*}
Observing that
if $|\xi-\eta|\leq \f{\eta}2,$ we have
$|\eta+\al(\xi-\eta)|\geq |\eta|-|\xi-\eta|\geq \f{|\eta|}2$ so that
\begin{align*}
\int_0^1(1-\al)|\Lambda'(\al\xi+(1-\al)\eta)|\,d\al\leq &C\int_0^1\f1{(1+|\eta+\al(\xi-\eta)|)^{\f32}}\,d\al\leq C\w{\eta}^{-\f32}.
\end{align*}
While if $|\xi-\eta|\geq \f{\eta}2,$ we have
\begin{align*}
\int_0^1(1-\al)|\Lambda'(\al\xi+(1-\al)\eta)|\,d\al\leq &C\int_0^1\f1{(1+|\eta+\al(\xi-\eta)|)^{\f32}}\,d\al\\
\leq & \f1{|\xi-\eta|}\int_0^{|\xi-\eta|}\f1{(1+|\eta+\tau|)^{\f32}}\,d\tau
\leq C|\eta|^{-1}.
\end{align*}
which together with the fact
\begin{align*}
&\bigl|\w{\xi}^{\f12}-\w{\eta}^{\f12}\bigr|\leq \f{\w{\eta}-\w{\xi}}{\w{\xi}^{\f12}+\w{\eta}^{\f12}}\leq C\w{\eta}^{-\f12}|\eta-\xi|,
\end{align*}
 ensures that
\beq \label{S2eq4}
|\cM(t,\xi,\eta)|\leq C_L|\eta|^{-1}(\eta-\xi)^2.
\eeq
Thanks to \eqref{S2eq4}, we obtain
\begin{align*}
\bigl\|\int_{\R}\cM(t,\cdot,\eta)\wh{S_{k-1}^\h a_\Phi}(\cdot-\eta)\wh{\De_k^\h f_\Phi}(\eta)\,d\eta\bigr\|_{L^2_\h}
\leq &C_L
\bigl\|\int_{\R}\bigl|\wh{S_{k-1}^\h \pa_x^2a_\Phi}(\cdot-\eta)\bigr| \bigl|\wh{\De_k^\h f_\Phi}(\eta)\bigr|\,d\eta\bigr\|_{L^2_\h}\\
\leq & C_L \|\wh{\pa_x^2a_\Phi}\|_{L^1_\h}\|\De_k^\h f_\Phi\|_{L^2_\h}\\
\leq &C_Lc_k2^{-ks}\|\pa_x^2a_\Phi\|_{H^\sigma_\h}\|f_\Phi\|_{H^s_\h},
\end{align*}
from which the support properties to terms in \eqref{S2eq3}, we complete the proof of \eqref{S2eq8}.
\end{proof}

\subsection{Two extended lemmas of \cite{PZ5}}
To derive the decay estimates of the solution to \eqref{eq: tu}, let us recall the following Lemmas:

\begin{lemma}\label{lem2.1}
{\sl Let $\Psi(t,y)\eqdefa \frac{y^2}{8\w{t}}$ and $d$ be a nonnegative integer. Let $u$ be a smooth enough function on $\R^d\times\R_+$
which decays to zero sufficiently fast as $y$ approaching to $+\infty.$  Then one has
\begin{subequations} \label{S2eq9}
\begin{gather}
\int_{\R^d\times\R_+} |\partial_y u(X,y)|^2 e^{2\Psi}\,dX\,dy \geq \frac1{2\w{t}}\int_{\R^d\times\R_+} |u(X,y)|^2 e^{2\Psi}\,dX\,dy,
\label{S2eq9a}\\
\int_{\R^d\times\R_+} |\partial_y u(X,y)|^2 e^{2\Psi}\,dX\,dy \geq \frac{\kappa}{2\w{t}}\int_{\R^d\times\R_+} |u(X,y)|^2 e^{2\Psi}\,dX\,dy \nonumber\\
\qquad\qquad\qquad\qquad\qquad\qquad\qquad+\f{\kappa(1-\kappa)}{4}\int_{\R^d\times\R_+} |\f{y}{\t}u(X,y)|^2 e^{2\Psi}\,dX\,dy, \label{S2eq9b}
\end{gather}
\end{subequations}
for any $\kappa>0.$
}
\end{lemma}

\begin{proof} \eqref{S2eq9a} was presented in Lemma 3.3 of \cite{IV16} and Lemma 3.1 of \cite{PZ5}. Let us now modify its proof to prove \eqref{S2eq9b}.
Indeed
we  get, by using integration by parts, that
\beno
\begin{split}
\int_{R_+} u^2(X,y) e^{\frac{y^2}{4\w{t}}} \,dy=&\int_{R_+} (\partial_y y) u^2(X,y) e^{\frac{y^2}{4\w{t}}}\,dy\\
=&-2\int_{R_+} y u(X,y) \partial_y u(X,y) e^{\frac{y^2}{4\w{t}}}\,dy-\frac{1}{2\w{t}}\int_{R_+} y^2 u^2(X,y) e^{\frac{y^2}{4\w{t}}}\,dy.
\end{split}
\eeno
By integrating the above inequality over $\R^d$ with respect to the $X$ variables, we find
\beno
\begin{split}
\int_{\R^d\times\R_+} u^2& e^{\frac{y^2}{4\w{t}}} \,dX\,dy+\frac{1}{2\w{t}}\int_{\R^d\times\R_+} y^2 u^2 e^{\frac{y^2}{4\w{t}}}\,dX\,dy
=-2\int_{\R^d\times\R_+} y u \partial_y u e^{\frac{y^2}{4\w{t}}}\,dX\,dy\\
\leq & \frac{\kappa}{2\w{t}}\int_{\R^d\times\R_+} y^2 u^2 e^{\frac{y^2}{4\w{t}}}\,dX\,dy+\f2\kappa\w{t}\int_{\R^d\times\R_+} (\partial_y u)^2 e^{\frac{y^2}{4\w{t}}}\,dX\,dy.
\end{split}
\eeno
This leads to \eqref{S2eq9b}.
\end{proof}

\begin{lemma}\label{lem2.2}
{\sl Let $G$ be defined by \eqref{S1eq3a}
 and $\Psi(t,y)\eqdefa \frac{y^2}{8\w{t}}.$
Let $\vf$ and $u$ be  smooth enough functions which decay to zero sufficiently fast as $y$ approaching to $\infty$ on $[0,T].$
 Then, for any $
\ga\in (0,1)$ and $t\leq T,$ one has
\begin{subequations} \label{S2eq20}
\begin{gather}
\bigl\|e^{\ga\Psi}\De_k^\h u_\Phi(t)\bigr\|_{L^2_+}\lesssim \bigl\|e^{\Psi}\De_k^\h G_\Phi(t)\bigr\|_{L^2_+}; \label{S2eq20a}\\
\bigl\|e^{\ga\Psi}\De_k^\h\p_yu_\Phi(t)\bigr\|_{L^2_+}\lesssim \bigl\|e^{\Psi}\De_k^\h\p_yG_\Phi(t)\bigr\|_{L^2_+}; \label{S2eq20b}\\
\bigl\|e^{\ga\Psi}\De_k^\h\p_y^2u(t)\bigr\|_{L^2_+}\lesssim \bigl\|e^{\Psi}\De_k^\h\p_y^2G(t)\bigr\|_{L^2_+}; \label{S2eq20c}\\
 \bigl\|e^{\ga\Psi}\De_k^\h\p_y^2u(t)\bigr\|_{L^\infty_{\rm v}(L^2_\h)}\lesssim \bigl\|e^{\Psi}\De_k^\h\p_y^2G(t)\bigr\|_{L^\infty_{\rm v}(L^2_\h)};  \label{S2eq20c1}\\
\w{t}^{-1}\|e^{\ga\Psi}\De_k^\h\p_y(y\vf)_\Phi(t)\|_{L^2_+}+\w{t}^{-\f12}\|e^{\ga\Psi}\De_k^\h\p_y^2(y\vf)_\Phi(t)\|_{L^2_+}\label{S7eq20d}\\
\nonumber
\qquad+\w{t}^{-\f34}\|e^{\ga\Psi}\De_k^\h\p_y(y\vf)_\Phi(t)\|_{L^\infty_{\rm v}(L^2_\h)}+\w{t}^{-\f14}\|e^{\ga\Psi}\De_k^\h\p_y^2(y\vf)_\Phi(t)\|_{L^\infty_{\rm v}(L^2_\h)}\lesssim \|e^\Psi\De_k^\h \p_yG_\Phi(t)\|_{L^2_+};\\
\w{t}^{-\f12}\|e^{\ga\Psi}\De_k^\h\p_y^3(y\vf)(t)\|_{L^2_+}\lesssim \|e^\Psi\De_k^\h \p_y^2G(t)\|_{L^2_+}\label{S7eq20e}.
 \end{gather}
\end{subequations}
}
\end{lemma}

\begin{proof} The estimates \eqref{S2eq20} have been presented in Lemma 3.2 of \cite{PZ5} except \eqref{S2eq20c}, \eqref{S2eq20c1}, \eqref{S7eq20e} and
\beq\label{S2eq13}
\w{t}^{-\f12}\|e^{\ga\Psi}\De_k^\h\p_y^2(y\vf)_\Phi(t)\|_{L^2_+}+
\w{t}^{-\f14}\|e^{\ga\Psi}\De_k^\h\p_y^2(y\vf)_\Phi(t)\|_{L^\infty_{\rm v}(L^2_\h)}\lesssim \|e^\Psi\De_k^\h \p_yG_\Phi(t)\|_{L^2_+}.
\eeq

The following fact will be used frequently: for any $\kappa\in(0,1]$ and function $f$, which
decays to zero sufficiently fast as $y$ approaching to $\infty,$ it holds that
 \begin{align}
 \|e^{\kappa\Psi} f_\Phi\|_{L^2_\v}\leq C\t^\f12\|e^{\kappa\Psi} \pa_y f_\Phi\|_{L^2_\v},\label{est: (f,pa_y f)}\\
 \|e^{\kappa\Psi} f_\Phi\|_{L^\infty_\v}\leq C\t^\f12\|e^{\kappa\Psi} \pa_y f_\Phi\|_{L^\infty_\v}.\label{est: (f,pa_y f)1}
 \end{align}
 Indeed we observe that
 \begin{align*}
 \|e^{\kappa\Psi} f_\Phi\|_{L^2_\v}=\bigl\|e^{\kappa\Psi} \int_y^\infty \pa_yf_\Phi dy'\bigr\|_{L^2_\v}\leq
 \bigl\| \int_y^\infty e^{-\f{\kappa(y-y')^2}{8\t}}e^{\kappa\Psi}\pa_yf_\Phi dy'\bigr\|_{L^2_\v}.
 \end{align*}
 Applying Young's inequality leads to \eqref{est: (f,pa_y f)}. So does \eqref{est: (f,pa_y f)1}.

On the other hand, notice that $u=\p_y\vf$ and $\vf|_{y=0}=0,$ we deduce from \eqref{S1eq3a} that
\beq \label{S1eq2}
\vf(t,x,y)=e^{-\f{y^2}{4\w{t}}}\int_0^ye^{\f{(y')^2}{4\w{t}}}G(t,x,y')\,dy',
\eeq
which implies
\begin{subequations} \label{S1eq34q}
\begin{gather} \label{S1eq3}
u=\pa_y\vf=-\f{y}{2\w{t}}e^{-\f{y^2}{4\w{t}}}\int_0^ye^{\f{(y')^2}{4\w{t}}}G(t,x,y')\,dy'+G,\\
\label{S1eq4}
\pa_yu=-\f{y}{2\w{t}}G+\pa_yG(t,y)
+\Bigl(-\f{1}{2\w{t}}+\f{y^2}{4\w{t}^2}\Bigr)e^{-\f{y^2}{4\w{t}}}\int_0^ye^{\f{(y')^2}{4\w{t}}}G(t,x,y')\,dy',
\end{gather}
\end{subequations}
and
\begin{subequations} \label{S1eq34p}
\begin{gather}
 \label{eq: pa_y^2 u}
\pa_y^2u=-\f{y}{2\t}\pa_y G-\f{1}{2\t}G+\pa_y^2 G+\f{y}{2\t^2}\vf-\f{1}{2\t}\bigl(1-\f{y^2}{2\t}\bigr)\pa_y\vf,\\
 \label{S2eq10}
\p_y^2(y\vf)=-\f{y}{\t}\vf+\bigl(1-\f{y^2}{2\t}\bigr)\pa_y\vf+G+y\pa_yG,\\
 \label{eq: pa_y^2(y vf)}
\p_y^3(y\vf)=-\f{y}{\t}\pa_y\vf-\f{1}{\t}\vf+\bigl(1-\f{y^2}{2\t}\bigr)\pa_y^2\vf-\f{y}{\t}\pa_y \vf+2\pa_yG+y\pa_y^2G.
\end{gather}
\end{subequations}

Let us  prove \eqref{S2eq20c} first. Indeed it follows from \eqref{est: (f,pa_y f)} that
\begin{align*}
\t^{-1}\|e^{\gamma\Psi}y\De_k^\h\pa_y G_\Phi\|_{L^2_{+}}\leq \t^{-1}\|e^{-(1-\gamma)\Psi}y\|_{L^\infty_\v}\|e^\Psi \De_k^\h\pa_y G_\Phi\|_{L^2_{+}}\leq C\|e^\Psi \De_k^\h\pa_y^2 G_\Phi\|_{L^2_{+}}.
\end{align*}
Along the same line, due to $u=\pa_y\vf,$ we deduce from \eqref{S2eq20a} and \eqref{est: (f,pa_y f)} that
\begin{align*}
\t^{-1}\|e^{\gamma\Psi}\De_k^\h G_\Phi\|_{L^2_{+}}\leq&C\t^{-\f12}\|e^{\gamma\Psi}\De_k^\h \pa_yG_\Phi\|_{L^2_{+}}\leq C\|e^{\Psi}\De_k^\h \pa_y^2G_\Phi\|_{L^2_{+}},\\
\t^{-2}\|e^{\gamma\Psi}y\De_k^\h \vf_\Phi\|_{L^2_{+}}\leq& C\t^{-\f32}\|e^{\gamma_1\Psi}\De_k^\h \vf_\Phi\|_{L^2_{+}}\leq C\t^{-1}\|e^{\gamma_1\Psi}\De_k^\h \pa_y\vf_\Phi\|_{L^2_{+}}\\
\leq&C\t^{-1}\|e^{\Psi}\De_k^\h G_\Phi\|_{L^2_{+}}\leq C\|e^{\Psi}\De_k^\h \pa_y^2G_\Phi\|_{L^2_{+}},\end{align*}
and
\begin{align*}
\t^{-1}\|e^{\gamma\Psi}\bigl(1-\f{y^2}{2\t}\bigr)\De_k^\h\pa_y\vf_\Phi\|_{L^2_{+}}\leq
&C\t^{-1}\|e^{\gamma_1\Psi}\De_k^\h \pa_y\vf_\Phi\|_{L^2_{+}}\leq C\|e^{\Psi}\De_k^\h \pa_y^2G_\Phi\|_{L^2_{+}},
\end{align*}
for $\gamma_1\in(\gamma,1).$

By summarizing the above estimates and  using \eqref{eq: pa_y^2 u}, we achieve \eqref{S2eq20c}. Estimate \eqref{S2eq20c1} can be proved similarly by using \eqref{est: (f,pa_y f)1}

On the other hand, for any $\ga_1\in (\ga,1),$ we get, by applying \eqref{est: (f,pa_y f)} and \eqref{S2eq20b}, that
\begin{align*}
\t^{-1}\|e^{\ga\Psi}y\De_k^\h\vf_\Phi(t)\|_{L^2_+}\lesssim &\t^{-\f12}\|y\t^{_\f12}e^{(\ga-\ga_1)\Psi}\|_{L^\infty_\v}\|e^{\ga_1\Psi}\De_k^\h \vf_\Phi(t)\|_{L^2_+}\\
\lesssim &\|e^{\ga_1\Psi}\De_k^\h u_\Phi(t)\|_{L^2_+}\\
\lesssim &\t^{\f12}\|e^{\ga_1\Psi}\De_k^\h \pa_y u_\Phi(t)\|_{L^2_+}\lesssim \t^{\f12}\|e^{\Psi}\De_k^\h \pa_y G_\Phi(t)\|_{L^2_+}.
\end{align*}
Along the same line, we have
\begin{align*}
\bigl\|e^{\ga\Psi}\bigl(1-\f{y^2}{2\t}\bigr)\De_k^\h\pa_y\vf_\Phi(t)\bigr\|_{L^2_+}
\lesssim &\bigl\|\bigl(1-\f{y^2}{2\t}\bigr) e^{(\ga-\ga_1)\Psi}\bigr\|_{L^\infty_\v}\|e^{\ga_1\Psi}\De_k^\h u_\Phi(t)\|_{L^2_+}\\
\lesssim &\t^{\f12}\|e^{\Psi}\De_k^\h \pa_yG_\Phi(t)\|_{L^2_+}.
\end{align*}
Finally, it is easy to observe that
\begin{align*}
\|e^{\ga\Psi}y\De_k^\h\pa_y G_\Phi(t)\|_{L^2_+}\lesssim &\t^{\f12}\bigl\|{y}{\t^{-\f12}} e^{(\ga-1)\Psi}\bigr\|_{L^\infty_\v}
\|e^{\Psi}\De_k^\h \pa_y G_\Phi(t)\|_{L^2_+}\\
\lesssim &\t^{\f12}\|e^{\Psi}\De_k^\h \pa_yG_\Phi(t)\|_{L^2_+}.
\end{align*}

In view of \eqref{S2eq10}, by summarizing the above estimates, we achieve
\beno
\w{t}^{-\f12}\|e^{\ga\Psi}\De_k^\h\p_y^2(y\vf)_\Phi(t)\|_{L^2_+}\lesssim \|e^\Psi\De_k^\h \p_yG_\Phi\|_{L^2_+}.
\eeno
The second term in \eqref{S2eq13} and estimate \eqref{S7eq20e} can be proved similarly, we omit the details here.
This completes the proof of the lemma.
\end{proof}

\setcounter{equation}{0}
\section{Sketch of the proof to Theorem \ref{th1.1}}

In this section, we shall sketch the proof of Theorem \ref{th1.1}. Let us start with the new formulation of the problem.

\subsection{New formulation of the problem}

By applying Bony's decomposition \eqref{Bony} in the horizontal variable, we write
\begin{align*}
&u\pa_x u=T^\h_u\pa_x u+T_{\pa_x u}^\h u+R^\h(u,\pa_x u),\\
&v\pa_y u=T_{\pa_y u}^\h v+T_{v}^\h\pa_y u+R^\h(v,\pa_y u).
\end{align*}
Applying the operator:  $e^{\Phi(t,D_x)},$ to the above equalities yields
\beq \label{S1eq2a}
\begin{split}
&(u\pa_x u)_\Phi=T_u^\h \pa_x u_\Phi+\f{\d(t)}2T_{D_xu}^\h\Lambda(D_x)\pa_x u_\Phi+\bigl(T_{\pa_xu}^\h u+R^\h(u,\pa_x u)\bigr)_\Phi+f_1,\\
&(v\pa_y u)_\Phi=T_{\pa_y u}^\h v_\Phi+T_v^\h\pa_y u_\Phi+\f{\d(t)}2T_{\pa_yD_xu}^\h\Lambda(D_x)v_\Phi+R^\h(v,\pa_y u)_\Phi+f_2,
\end{split} \eeq
where the remainder terms $f_1, f_2$ are defined by
\beq \label{eq: (f_1,f_2)}
\begin{split}
&f_1\eqdefa (T_u^\h\pa_x u)_\Phi-T_u^\h\pa_x u_\Phi-\f{\d(t)}2T^\h_{D_xu}\Lambda(D_x)\pa_x u_\Phi,\\
&f_2\eqdefa (T^\h_{\pa_y u}v)_\Phi-T^\h_{\pa_y u}v_\Phi-\f{\d(t)}2T^\h_{\pa_yD_xu}\Lambda(D_x)v_\Phi+(T^\h_{v}\pa_y u)_\Phi-T^\h_v\pa_yu_\Phi.
\end{split} \eeq
Here and in the sequel, we always denote $D_x\eqdefa\f1{i}\pa_x.$

In view of \eqref{S1eq2a}, we get, by applying the operator $e^{\Phi(t,D_x)}$ to the system  \eqref{eq: tu}, that
\begin{align}\label{eq: tu_Phi}
\left\{
\begin{aligned}
&\pa_t u_\Phi+\la \dot{\theta}\D^\f12 u_\Phi+T^\h_u\pa_x u_\Phi+T^\h_v\pa_y u_\Phi+T^\h_{\pa_y u}v_\Phi\\
&\qquad+\f{\d(t)}2\bigl(T^\h_{D_xu}\Lambda(D_x)\pa_x u_\Phi+T^\h_{\pa_yD_xu}\Lambda(D_x)v_\Phi\bigr)-\pa_y^2 u_\Phi=f,\\
&\pa_x u_\Phi+\pa_y v_\Phi=0,\\
&(u_\Phi,v_\Phi)|_{y=0}=0,\quad \lim_{y\to +\infty}u_\Phi=0,\\
&u_\Phi|_{t=0}=e^{\d\D^\f12}u_0,
\end{aligned}
\right.
\end{align}
where the source term $f$ is given by
\begin{align}\label{eq: f_3}
f=-f_1-f_2-f_3 \with f_3\eqdefa \bigl(T^\h_{\pa_xu} u+R^\h(u,\pa_xu)+R^\h(v,\pa_y u)\bigr)_\Phi.
\end{align}

 To construct the global small solution of \eqref{eq: tu} with Gevery 2 regularity in $x$ variable, we need first to
 modify the definitions of the  auxiliary functions $H,\phi,$ which  are first introduced by Dietert and G\'erard-Varet in \cite{DG19}.
 In order to do it, let $\Psi(t,y)=\f{y^2}{8\t}$ and $\hbar(t)$ be a positive and non-decreasing smooth function, we define  the operator
\begin{align}\label{defcL}
\mathcal{L}\eqdefa\pa_t+\la\dot\th(t)\D^\f12+T^\h_u\pa_x+T^\h_v\pa_y+\f{\d(t)}2T^\h_{D_xu}\Lambda(D)\pa_x-\pa_y^2,
\end{align}
and its adjoint operator in  $L^2(\hbar(t)e^{2\Psi})$ is given by
\beq\label{defcL*}
\begin{split}
\mathcal{L}^*\eqdefa&-\pa_t-\f{\hbar'}{\hbar}+\f{y^2}{4\t^2}+\la\dot\th(t)\D^\f12-\pa_x(T^\h_u)^*\\
&-\bigl(\pa_y+\f{y}{2\t}\bigr)(T^\h_v)^*-
\f{\d(t)}2\Lambda(D)\pa_x(T^\h_{D_xu})^*-\bigl(\pa_y+\f{y}{2\t}\bigr)^2.
\end{split}\eeq

We now introduce the function $H$ via
\begin{align}\label{eq: H}
\left\{
\begin{aligned}
&\mathcal{L}\int_y^\infty H dz=\dot\th(t) \int_y^\infty u_\Phi dz, \\
&\pa_y H|_{y=0}=0,\quad H|_{y\to +\infty}=0,\\
&H|_{t=0}=0,
\end{aligned}
\right.
\end{align}
and $\phi$ by
\begin{align}\label{eq: phi}
\left\{
\begin{aligned}
&\mathcal{L}^*\phi=\dot\th(t) H, \\
&\phi|_{y=0}=0,\quad \phi|_{y\to +\infty}=0,\\
&\phi|_{t=T}=0.
\end{aligned}
\right.
\end{align}
The existence of $H$ and $\phi$ will be guaranteed by the {\it a priori} estimates presented in Proposition \ref{S1prop1}.

We have the following remark concerning the definitions of $H$ and $\phi:$

\begin{itemize}

\item[(1)]
 The reason why we need the function $\hbar(t)$ in  the adjoint operator $\mathcal{L}^\ast$ given by \eqref{defcL*}
is for the purpose of deriving the decay estimates of the solutions $H, \phi$ and $u.$ We notice that $\hbar(t)=1$ in
the corresponding definition in \cite{DG19}.

\item[(2)] In \eqref{eq: H} and \eqref{eq: phi},
we define $H$ and $\phi$ via  $\int_y^\infty H\,dz$ and $\int_y^\infty \phi\,dz$ (instead of $\int_0^y H\,dz$ and $\int_0^y \phi\,dz$
 as in \cite{DG19}) is to make the solutions $H$ and $\phi$ decay faster as $y$ approaching
to $\infty.$

\item[(3)] We have the additional $\dot \th(t)$ before  $\int_y^\infty u_\Phi\,dz$ in \eqref{eq: H} and $H$ in \eqref{eq: phi}
is for the purpose of  controlling the Gevery radius of the solutions $H$ and $\phi$ simultaneously with their norms.
 This idea   goes back to
\cite{Ch04} where Chemin introduced a tool to  make analytical type estimates and
controlling the size of the analytic radius simultaneously to classical Navier-Stokes system. It was used in the context of
anisotropic Navier-Stokes system \cite{CGP} ( see also \cite{LnZ1,PZ5, PZZ2, mz1, mz2,  ZHZ}), which implies  the global
well-posedness of three dimensional Navier-Stokes system with a
class of ``ill prepared data", which is slowly varying in the
vertical variable, namely of the form $\e x_3,$  and the $B^{-1}_{\infty,\infty}(\R^3)$
norm of which blows up as the small parameter goes to zero.

\end{itemize}

Applying  $\pa_y$  to the first equation of \eqref{eq: H} yields
\begin{align}
\dot\th(t) u_\Phi&=\mathcal{L}H-T^\h_{\pa_yu}\pa_x\int_y^\infty Hdz+T^\h_{\pa_y v}H-\f{\d(t)}2T^\h_{\pa_yD_xu}\Lambda(D_x)\pa_x\int_y^\infty Hdz,\label{tu_Phi}
\end{align}
Whereas by
taking $\pa_x$ to the first equation of \eqref{eq: H}  and using the fact $v=\int_y^\infty\pa_xu dz,$ we find
\beq \label{v_Phi}
\begin{split}
\dot\th(t)v_\Phi=&\mathcal{L}\pa_x\int_y^\infty H\,dz+T^\h_{\pa_xu}\pa_x\int_y^\infty H\,dz\\
&-T^\h_{\pa_x v}H+\f{\d(t)}2T^\h_{\pa_xD_xu}\Lambda(D_x)\pa_x\int_y^\infty H\,dz.
\end{split} \eeq

\subsection{Ansatz} We shall use the conjugate method to derive the Weighted Gevrey estimates of $H$ and $\phi$ in Section \ref{Sect4}.
In the process of the estimates,  we shall use the formulas \eqref{tu_Phi} and \eqref{v_Phi}, and
thus the time derivative of $\dot\th(t)$ gets involved in. This makes the evolution equations for $\th(t)$ as in all the
previous references like in \cite{PZ5}, which is defined via
 \begin{equation}\label{open}
 \quad\left\{\begin{array}{l}
\displaystyle \dot{\th}(t)=\w{t}^\f14\|e^\Psi \p_yG_\Phi(t)\|_{B^{\f{1}2,0}},\\
\displaystyle \th|_{t=0}=0.
\end{array}\right.
\end{equation}
to be impossible here.

Fortunately, we observe that what we really used in \cite{PZ5} is the time decay estimate of $\|e^\Psi \p_yG_\Phi(t)\|_{B^{\f{1}2,0}}$. Motivated by this observation,
for some $\beta>1$ which will be fixed in Subsection \ref{Sub3.4}, we define
 \beq \label{S3eq-1}
\left\{
\begin{aligned}
&\dot\th(t)=\e^{\f12}\t^{-\beta},\\
&\th(0)=0.\end{aligned}
\right.\eeq
Let $G$ and $T^\ast$ be determined  respectively by \eqref{S1eq3a} and \eqref{1.8a}.
 Let $\e>0,$ $\ga_0\in \bigl(1,\f54\bigr)$  and a large enough constant $C,$ we define
\beq \label{assume: 1}
\begin{split}
T^\star\eqdefa \sup\Bigl\{\ t<T^\ast, \ &\|G_\Phi(t)\|_{H^{4,0}_\Psi}+\t^\f12\|\pa_yG_\Phi(t)\|_{H^{3,0}_\Psi}\\
&+\t\|\pa_y^2G(t)\|_{H^{3,0}_\Psi}+\t^{\f32}\|\pa_y^3G(t)\|_{H^{2,0}_\Psi}\leq C\e\t^{-\ga_0} \ \Bigr\}.
\end{split}
\eeq

\subsection{The key {\it a priori} estimates} In this subsection, we shall present the {\it a priori}
estimates used in the proof of Theorem \ref{th1.1}.

In Section \ref{Sect4}, we shall prove the following proposition concerning the {\it a priori} estimates of $H$ and $\phi.$

\begin{prop}\label{S1prop1}
{\sl Let $H$ and $\phi$ be  smooth enough solutions of \eqref{eq: H} and \eqref{eq: phi} respectively, which decay to zero
sufficiently fast as $y$ approaching to $\infty.$ Then there exist $c_0,\e_1\in(0,1)$ and $\la_1>0$ so that for
$\e\leq\e_1,$ $\eta\geq \e^{\f32},$ and $\la \geq \la_1,$ and  for any $t<T\leq T^\star,$ there hold
\beq\label{S3eq2a}
\begin{split}
\|\w{t}^{\f{1-\eta}4}&\phi(t)\|_{H^{\f{29}{4},0}_\Psi}^2+\la
\int_t^T\dot\th(t')
\|\w{t'}^{\f{1-\eta}4}\phi(t')\|_{H^{\f{15}{2},0}_\Psi}^2\,dt'\\
&+c_0\eta\int_t^T\|\w{t'}^{\f{1-\eta}4} \pa_y\phi(t')\|_{H^{\f{29}{4},0}_\Psi}^2\,dt'
  \leq \int_t^T\dot\th(t')\|\w{t'}^{\f{1-\eta}4} H(t')\|_{H^{7,0}_\Psi}^2\,dt',
\end{split}
\eeq
and
\beq \label{S3eq10}
\begin{split}
\|\w{T}^{\f{1-\eta}4}& H(T)\|_{H^{\f{27}{4},0}_\Psi}^2+\la\int_0^T\dot\th(t)\|\w{t}^{\f{1-\eta}4}  H(t)\|_{H^{7,0}_\Psi}^2\,dt\\
&+c_0\eta\int_0^T\Bigl(\|\w{t}^{\f{1-\eta}4} \p_{y}H(t)\|_{H^{\f{27}{4},0}_\Psi}^2+\|\w{t}^{-\f{3+\eta}4} {y}H(t)\|_{H^{\f{27}{4},0}_\Psi}^2\Bigr)\,dt
\\
\leq& \|u_\Phi(0)\|_{H^{\f{25}{4}}_\Psi}^2
+\e^\f32\int_0^T\Bigl(\t^{-\gamma_0-\f14}\|\w{t}^{\f{1-\eta}4}u_\Phi\|_{H^{6,0}_\Psi}^2+\|\w{t}^{\f{1-\eta}4}\pa_y u_\Phi\|_{H^{\f{11}{2},0}_\Psi}^2 \Bigr)\,dt.
\end{split}
 \eeq}
 \end{prop}

 With Proposition \ref{S1prop1}, motivated by \cite{DG19} and \cite{PZ5}, we shall derive
  the  {\it a priori} decaying estimates of the solution $u$ to \eqref{eq: tu} in Subsection \ref{Sect5}.
  This will be the most crucial ingredient used in the proof of Theorem \ref{th1.1}.
  In order to close the estimate in \eqref{S4eq2} below, we shall derive the Gevery estimate of $\pa_yu$ in Subsection \ref{Sect6}.
  The main result states as follows:

 \begin{prop}\label{S1prop2}
{\sl  Let $u$ be smooth enough solution of \eqref{eq: tu} which decay to zero
sufficiently fast as $y$ approaching to $\infty.$ Then there exist $\e_2\in(0,1)$ and $\la_2>0$ so that for
$\e\leq\e_2,$ $\eta\geq \e^{\f32},$ $\la \geq \la_2$ and  for any $T\leq T^\star,$ there hold
\beq\label{S4eq2}
\begin{split}\|&\w{T}^{\f{1-\eta}4} u_\Phi(T)\|_{H^{\f{11}{2},0}_\Psi}^2
+\la\int_0^T\dot\th(t)\|\w{t}^{\f{1-\eta}4} u_\Phi(t)\|_{H^{\f{23}{4},0}_\Psi}^2\,dt\\
&+\f34\eta\int_0^T \|\w{t}^{\f{1-\eta}4} \pa_yu_\Phi(t)\|_{H^{\f{11}{2},0}_\Psi}^2\,dt
\leq \| e^{\d\D^\f12} u_0\|_{H^{\f{11}{2},0}_\Psi}^2+C\eta^{-1}\e\|\w{T}^{\f{1-\eta}4} H(T)\|_{H^{\f{27}{4},0}_\Psi}^2\\
&+\eta\e^\f14\int_0^T\|\w{t}^{\f{3-\eta}4}\pa_y^2 u_\Phi\|_{H^{5,0}_\Psi}^2\,dt+\eta\e^\f14\int_0^T\|\w{t}^{-\f{3+\eta}4}{y} H\|_{H^{\f{27}{4},0}_\Psi}^2\,dt\\
&+C\bigl(1+\la\e^\f12+\e^\f14\eta^{-1}\bigr)\int_0^T \dot\th(t)\Bigl(\|\w{t}^{\f{1-\eta}4} H\|_{H^{7,0}_\Psi}^2+\| \w{t}^{\f{3-\eta}4}\pa_yu_\Phi\|_{H^{\f{21}{4},0}_\Psi}^2\Bigr)\,dt,
\end{split} \eeq
and
\beq \label{S5eq6}
\begin{split}
&\|\t^{\f{3-\eta}{4}}\pa_yu_\Phi(T)\|_{H^{5,0}_\Psi}^2+\la
\int_0^T\dot\th(t)\|\t^{\f{3-\eta}{4}}\pa_yu_\Phi(t)\|_{H^{\f{21}{4},0}_\Psi}^2\,dt\\
&\qquad+\bigl(1-\f{\eta}{8}\bigr)\int_0^T\|\t^{\f{3-\eta}{4}}\pa_y^2u_\Phi(t)\|_{H^{5,0}_\Psi}^2\,dt\\
&\leq \|e^{\d\D^\f12}\pa_yu_0\|_{H^{5,0}_\Psi}^2+C\eta^{-1}\| e^{\D^\f12} u_0\|_{H^{\f{11}{2},0}_\Psi}^2+C\eta^{-2}\e\|\w{T}^{\f{1-\eta}4} H(T)\|_{H^{\f{27}{4},0}_\Psi}^2\\
&+\e^\f14\int_0^T\|\w{t}^{-\f{3+\eta}4}{y} H\|_{H^{\f{27}{4},0}_\Psi}^2\,dt+C\eta^{-1}\bigl(1+\la\e^\f12+\e^\f14\eta^{-1}
\bigr)
\int_0^T \dot\th(t)\|\w{t}^{\f{1-\eta}4} H\|_{H^{7,0}_\Psi}^2\,dt.
\end{split}
\eeq
}
\end{prop}

In Subsection \ref{Sect7.1}, we shall derive the Gevery estimate of $G_\Phi.$ While the Gevery estimate of $\pa_yG_\Phi$
will be presented
 in Subsection \ref{Sect7.2}. These estimates are crucial to control the Gevery radius to the solution of \eqref{eq: tu}.
 More precisely,

 \begin{prop}\label{S1prop3}
 {\sl Let $G$ be determined by \eqref{S1eq3a}. Then if $\la\geq 2C\e^{\f12}(1+\eta^{-1}\e)$ and $\b\leq\ga_0-\f1{12},$ one has
\beq \label{S6eq4}
\begin{split}
\|\w{T}^{\f{5-\eta}{4}}&G_\Phi(T)\|_{H^{3,0}_\Psi}^2+
\la\int_0^T\dot\th(t)
\|\t^{\f{5-\eta}{4}} G_\Phi(t)\|_{H^{\f{13}{4},0}_\Psi}^2\,dt\\
&+\f{3}{4}\eta\int_0^T\|\t^{\f{5-\eta}{4}}\pa_yG_\Phi(t)\|_{H^{3,0}_\Psi}^2\,dt\\
\leq &\|G_\Phi(0)\|_{H^{4,0}_\Psi}^2+C\e^\f12\int_0^T\dot\th(t)\|\t^{\f{1-\eta}{4}}u_\Phi\|_{H^{\f{23}{4},0}_\Psi}^2\,dt,
\end{split}\eeq
and
\beq \label{S7eq4}
\begin{split}
\|\langle & T\rangle^\f{7-\eta}{4}\pa_yG_\Phi(T)\|_{H^{3,0}_\Psi}^2
+\la\int_0^T\dot\th(t)
\|\t^\f{7-\eta}{4} \pa_yG_\Phi(t)\|_{H^{\f{13}{4},0}_\Psi}^2\,dt\\
&+\bigl(1-\f\eta 4\bigr)\int_0^T\|\t^\f{7-\eta}{4}\pa_y^2G_\Phi(t)\|_{H^{3,0}_\Psi}^2\,dt\\
\leq&C\eta^{-1}(\|G_\Phi(0)\|_{H^{4,0}_\Psi}^2+\|\pa_yG_\Phi(0)\|_{H^{3,0}_\Psi}^2)
+C\eta^{-1}\e^\f12\int_0^T\dot\th(t)\|\t^{\f{1-\eta}{4}}u_\Phi\|_{H^{\f{23}{4},0}_\Psi}^2\,dt.
\end{split} \eeq
}
\end{prop}

In Section \ref{Sect8}, we shall derive the Sobolev estimates of $\pa_y^2G$ and $\pa_y^3G.$

\begin{prop}\label{S1prop4}
 {\sl Let $G$ be determined by \eqref{S1eq3a}. Then if $\la\geq 2C\e^{\f12}(1+\eta^{-1}\e)$ and $\b\leq 2\ga_0-\f1{2},$ one has
\beq\label{S11eq7}
\begin{split}
\|\w{T}^{\f{9-\eta}{4}}&\pa_y^2G(T)\|_{H^{3,0}_\Psi}^2
+(1-\f\eta 4)\int_0^T\|\t^{\f{9-\eta}{4}}\pa_y^3G(t)\|_{H^{3,0}_\Psi}^2\,dt\\
\leq&C\eta^{-1}\e^\f12\int_0^T\dot\th(t)\|\t^{\f{1-\eta}{4}}u_\Phi\|_{H^{\f{23}{4},0}_\Psi}^2\,dt\\
&+C\eta^{-1}\Bigl(\|G_\Phi(0)\|_{H^{4,0}_\Psi}^2+\|\pa_yG_\Phi(0)\|_{H^{3,0}_\Psi}^2+\|\pa_y^2G(0)\|_{H^{3,0}_\Psi}^2\Bigr),
\end{split}\eeq
and if moreover $\eta\geq \e^2,$
\beq\label{S8eq7pq}
\begin{split}
\|&\t^{\f{11-\eta}{4}}\pa_y^3G(T)\|_{H^{2,0}_\Psi}^2
+\bigl(1-\f\eta 4\bigr)\int_0^T\|\t^{\f{11-\eta}{4}}\pa_y^4G(t)\|_{H^{2,0}_\Psi}^2\,dt\\
&\quad\leq C\eta^{-1}\e^\f12\int_0^T\dot\th(t)\|\t^{\f{1-\eta}{4}}u_\Phi\|_{H^{\f{23}{4},0}_\Psi}^2\,dt\\
&\qquad+C\eta^{-1}\Bigl(\|G_\Phi(0)\|_{H^{4,0}_\Psi}^2+\|\pa_yG_\Phi(0)\|_{H^{3,0}_\Psi}^2+\|\pa_y^2G(0)\|_{H^{3,0}_\Psi}^2+\|\pa_y^3G(0)\|_{H^{2,0}_\Psi}^2\Bigr).
\end{split}\eeq
}
\end{prop}

\subsection{Proof of Theorem \ref{th1.1}} \label{Sub3.4}
Now we are in a position to complete the proof of  Theorem \ref{th1.1}.

\begin{proof}[Proof of Theorem \ref{th1.1}]
Given initial data satisfying \eqref{S9eq3}, we deduce from \cite{DG19} that \eqref{eq: tu}
has a unique solution on $[0, \frak{T}].$ We are going to prove that $\frak{T}=\infty$ if $\e$ in \eqref{S9eq3} is
sufficiently small.
Otherwise, if $ \frak{T}<\infty,$ we take  small enough $\eta\in \left(0,\f16\right),$ which satisfies $\eta\geq C\e^{\f14},$
and take $\ga_0=\f54-\eta$ in \eqref{assume: 1}. Precisely, let $T^\ast$ be determined by \eqref{1.8a},
we denote
\beq\label{S9eq5}
\begin{split}
T_\eta^\star\eqdefa \sup\Bigl\{ t<T^\ast,\quad
&\|G_\Phi(t)\|_{H^{4,0}_\Psi}+\t^\f12\|\pa_yG_\Phi(t)\|_{H^{3,0}_\Psi}\\
&+\t\|\pa_y^2G(t)\|_{H^{3,0}_\Psi}+\t^\f32\|\pa_y^3G(t)\|_{H^{2,0}_\Psi}\leq C_1\e\t^{-\f54+\eta}\ \Bigr\},
\end{split}
\eeq
where the constant $C_1$ will be determined later on.

Let ${\rm E}(t)$ and ${\rm D}(t)$ be determined respectively by \eqref{S9eq1} and \eqref{S9eq2}.  Let $H$ and $\phi$ be  smooth enough solutions of \eqref{eq: H} and \eqref{eq: phi} respectively, which decay to zero
sufficiently fast as $y$ approaching to $\infty.$
For $t\leq T^\star_\eta,$ we introduce
\begin{subequations} \label{energy}
\begin{gather}
\label{S9eq1a}
\frak{E}(t)\eqdefa \|\t^{\f{1-\eta}{4}}\phi\|_{H^{\f{29}{4},0}_\Psi}^2+\|\t^{\f{1-\eta}{4}}H\|_{H^{\f{27}{4},0}_\Psi}^2+{\rm E}(t),\\
 \label{S9eq2a}
\frak{D}(t)\eqdefa \|\t^{\f{1-\eta}{4}}\pa_y\phi\|_{H^{\f{29}{4},0}_\Psi}^2+\|\t^{\f{1-\eta}{4}}\pa_yH\|_{H^{\f{27}{4},0}_\Psi}^2
+\|\t^{-\f{3+\eta}{4}}{y} H\|_{H^{\f{27}{4},0}_\Psi}^2+{\rm D}(t),\\
\label{S9eq2b}
\frak{G}(t)\eqdefa \|\t^{\f{1-\eta}{4}}\phi\|_{H^{\f{15}{2},0}_\Psi}^2+\|\t^{\f{1-\eta}{4}}H\|_{H^{7,0}_\Psi}^2+\|\t^{\f{1-\eta}{4}}u_\Phi\|_{H^{\f{23}{4},0}_\Psi}^2
\\
\qquad\qquad\qquad\qquad+\|\t^{\f{3-\eta}{4}}\pa_yu_\Phi\|_{H^{\f{21}{4},0}_\Psi}^2+\|\t^{\f{5-\eta}{4}}G_\Phi\|_{H^{\f{17}{4},0}_\Psi}^2
+\|\t^{\f{7-\eta}{4}}\pa_yG_\Phi\|_{H^{\f{13}{4},0}_\Psi}^2.
\nonumber
\end{gather}
\end{subequations}

Then taking $\ga_0=\f54-\eta$ and $\hbar(t)=\w{t}^{\f{1-\eta}2}$ in \eqref{S4eq3c} leads to
\beq\label{S9eq2c}
\begin{split}
\e\int_0^T\t^{-\gamma_0-\f14}&\|\t^{\f{1-\eta}4} u_\Phi\|_{H^{6,0}_\Psi}^2\,dt
\leq  \eta\|\w{T}^{\f{1-\eta}4} u_\Phi(T)\|_{H^{\f{11}{2},0}_\Psi}^2\\
&+C\eta^{-1}\e\|\w{T}^{\f{1-\eta}4} H(T)\|_{H^{\f{27}{4},0}_\Psi}^2+\eta\e^\f 14\int_0^T\frak{D}(t)\,dt\\
&+C\bigl(1+\la\e^\f12+\e^\f14\eta^{-1}\bigr)\int_0^T \dot\th(t)\frak{G}(t)\,dt.
\end{split}\eeq
By inserting the above estimate into \eqref{S3eq10}, and summing up the estimates from \eqref{S3eq2a} to \eqref{S8eq7pq}, we conclude that
that for $\beta\leq \f76-\eta,$ $\e\leq \min(\e_1,\e_2),$ and $\la\geq\max\left(\la_1, \la_2\right)$
\beq\label{S9eq2d}
\begin{split}
\frak{E}(T)&+\la\int_0^T\dot\th(t)\frak{G}(t)\,dt+c_0\eta\int_0^T\frak{D}(t)\,dt
\leq C_\eta\bigl(\|u_\Phi(0)\|_{H^{\f{25}4,0}_\Psi}^2+ {\rm E}(0)\bigr)\\
&+C\left(1+\la\e^\f12+\eta^{-1}\bigl(1+\e^\f12\eta^{-1}\bigr)\right)\int_0^T \dot\th(t)\frak{G}(t)\,dt
+\bigl(\e^\f14+\e^\f12\eta^{-1}\bigr)\int_0^T\frak{D}(t)\,dt.
\end{split} \eeq
By taking $\e\leq \min\bigl(\e_1,\e_2, \e_3)$ with $\e_3\eqdefa \min\bigl(\bigl(\f{c_0}4\bigr)^2\eta^4, \f1{(2C)^2}\bigr),$
and $\la\geq \la_0\eqdefa \max\left(\la_1, \la_2, \la_3\right)$ with $\la_3\eqdefa 2C\bigl(1+\eta^{-1}(1+\e^\f12\eta^{-1})\bigr),$
we deduce \eqref{S9eq3} and \eqref{S9eq2d} that
\beq\label{S9eq6}
\frak{E}(T)+\f{c_0\eta}2\int_0^TD(t)dt\leq C_\eta\bigl(\|u_\Phi(0)\|_{H^{\f{25}4,0}_\Psi}^2+ {\rm E}(0)\bigr)\leq C_\eta\e^2\leq \f{C_1^2}{4}\e^2,
\quad \forall \ T\leq T^\star_\eta,
\eeq
by taking $C_1\eqdefa 2\sqrt{C_\eta}$ in \eqref{S9eq5}.

On the other hand, by choosing $\beta=\f76-\eta>1,$ which satisfies all the assumptions from Propositions \ref{S1prop1}-\ref{S1prop4},
 we deduce from \eqref{S3eq-1} that
\beq \label{S9eq7}
\th(t)=\int_0^t\dot\th(s)ds= \e^\f12\int_0^t \s^{-\beta}ds\leq C\e^\f12\leq \f{\d}{4\la},\quad \forall \ t\leq T^\star_\eta,
\eeq
for $\e\leq \e_4.$

Therefore for $\e\leq \e_0\eqdefa \min\bigl(\e_1,\e_2, \e_3, \e_4)$ and $\la\geq \la_0,$
\eqref{S9eq6} and \eqref{S9eq7} contradicts with \eqref{1.8a} and \eqref{S9eq5}. So that
we get, by applying a standard continuous argument, that $T^*=T^\star_\eta=\infty. $
This completes the proof of Theorem \ref{th1.1}.
\end{proof}

\setcounter{equation}{0}

\section{The Gevrey Estimates of $H$ and $\phi.$}\label{Sect4}

 Let $(u,v)$ be a smooth enough solution of \eqref{eq: tu} on $[0,T^\ast]$ with $T^\ast$ being determined by
 \eqref{1.8a}. Let $H$ and $\phi$
 be determined respectively by \eqref{eq: H} and \eqref{eq: phi} on $[0,T^\ast].$ The goal of this
 section is to present the {\it a priori} estimates of $H$ and $\phi.$

Indeed thanks to \eqref{assume: 1}, for $t\leq T^\star,$ we deduce from Lemma \ref{lem2.2} that for any $\ga\in (0,1)$
\beq \label{assume: 2}
\|u_\Phi(t)\|_{H^{4,0}_{\ga\Psi}}+\t^\f12\|\pa_yu_\Phi(t)\|_{H^{3,0}_{\ga\Psi}}
+\t\|\pa_y^2u(t)\|_{H^{3,0}_{\ga\Psi}}\leq C\e\t^{-\ga_0}.
\eeq

Furthermore, we observe that

\begin{lem}\label{lem: phi}
{\sl Let $(u,v)$ be a smooth enough solution of \eqref{eq: tu} on $[0,T^\star].$
Then  for any $\ga\in (0,1)$ and $t\leq T^\star,$ we have
\begin{subequations}
\begin{gather}\label{S3eq8}
\|u_\Phi(t)\|_{L_{\v,\ga\Psi}^\infty(H^{3}_\h)}\leq C\e\w{t}^{-\left(\ga_0+\f14\right)} \andf
\|v_\Phi(t)\|_{L^\infty_{\v,\ga\Psi}(H^{3}_\h)}\leq C\e\w{t}^{-\left(\ga_0-\f14\right)},\\
\label{S3eq8b}
\|\pa_y u(t)\|_{L_{\v,\ga\Psi}^\infty(H^{3}_\h)}\leq C\e\w{t}^{-\left(\ga_0+\f34\right)}\andf
\|\pa_y^2 u(t)\|_{L_{\v,\ga\Psi}^\infty(H^{2}_\h)}\leq C\e\w{t}^{-\left(\ga_0+\f54\right)},
 \end{gather}
\end{subequations}
with $\|f\|_{L^\infty_{\v,\Psi}}\eqdefa \|e^{\Psi}f\|_{L^\infty_\v}$.}
\end{lem}
\begin{proof} It follows from \eqref{assume: 2} and $u=-\int_y^\infty\p_yu\,dz$ that
\begin{align*}
\|u_\Phi(t)\|_{L_{\v,\ga\Psi}^\infty(H^{3}_\h)}\leq &\sup_{y>0}\Bigl\{e^{\ga\Psi}\Bigl(\int_y^\infty e^{-2\ga\Psi}\,dz\Bigr)^{\f12}\Bigr\}
\Bigl(\int_0^\infty e^{2\ga\Psi}\|\p_yu_\Phi(t,\cdot,y)\|_{H^{3}_\h}^2\,dy\Bigr)^{\f12}\\
\leq &C\t^{\f14}\|\pa_yu_\Phi(t)\|_{H^{3,0}_{\ga\Psi}}\leq
  C\e\w{t}^{-\left(\ga_0+\f14\right)}.
\end{align*}
Similarly as $\p_xu+\p_yv=0,$ we deduce from \eqref{assume: 2}  that
\begin{align*}
\|v_\Phi\|_{L^\infty_{\v,\ga\Psi}(H^{3}_\h)}=&\bigl\|\int_y^\infty\pa_xu_\Phi dz\bigr\|_{L^\infty_{\v,\gamma\Psi}(H^{3}_\h)}\\
\leq &\sup_{y>0}\Bigl\{e^{\ga\Psi}\Bigl(\int_y^\infty e^{-2\ga\Psi}\,dz\Bigr)^{\f12}\Bigr\}\Bigl(\int_0^\infty e^{2\ga\Psi}\|u_\Phi(t,\cdot,y)\|_{H^{4}_\h}\,dy\Bigr)^{\f12}\\
\leq &C\t^{\f14}\|u_\Phi(t)\|_{H^{4,0}_{\ga\Psi}}\leq
  C\e\w{t}^{-\left(\ga_0-\f14\right)}.
\end{align*}
This leads to \eqref{S3eq8}.

The first inequality in \eqref{S3eq8b} can be derived along the same line. For the second one, thanks to Sobolev embedding theorem,
for any $\gamma_1\in(\gamma,1),$ we deduce from \eqref{S2eq20c1}  and \eqref{assume: 1} that
\begin{align*}
\|\pa_y^2 u(t)\|_{L_{\v,\ga\Psi}^\infty(H^{2}_\h)}\leq& C\|\pa_y^2 G(t)\|_{L_{\v,\gamma_1\Psi}^\infty(H^{2}_\h)}\leq C\|\pa_y^2 G(t)\|_{H^{2,0}_{\Psi}}^\f12\Big(\|\pa_y^3 G(t)\|_{H^{2,0}_{\Psi}}^\f12+\|\f{y}{\t}\pa_y^2 G(t)\|_{H^{2,0}_{\gamma_1\Psi}}^\f12\Big)\\
\leq&C\|\pa_y^2 G(t)\|_{H^{2,0}_{\Psi}}^\f12\Big(\|\pa_y^3 G(t)\|_{H^{2,0}_{\Psi}}^\f12+\t^{-\f12}\|\pa_y^2 G(t)\|_{H^{2,0}_{\Psi}}^\f12\Big)\leq C\e\t^{-\gamma_0-\f54}.
\end{align*}
This finishes the proof of Lemma \ref{lem: phi}.
\end{proof}

\subsection{The Estimate of $\phi$}
We shall present the {\it a priori} estimate of $\phi$ in this subsection.
 The main result states as follows:

\begin{prop}\label{lem: phi}
{\sl Let $\phi$ be a smooth enough solution of \eqref{eq: phi} on $[0,T^\star],$ which decays to zero sufficiently fast
as $y$ approaching to $+\infty.$ Let $\hbar(t)$ be a non-negative and non-decreasing function on $[0,T]$ with $T\leq T^\star.$ Then
if $\b\leq\ga_0+\f14,$  for any  sufficiently small $\eta>0$ and $t<T,$ one has
\beq\label{S3eq2}
\begin{split}
\|&\sqrt{\hbar}\phi(t)\|_{H^{\f{29}{4},0}_\Psi}^2
-\int_t^T\|\sqrt{\hbar'}\phi(t')\|_{H^{\f{29}{4},0}_\Psi}^2\,dt'+ (1-\eta) \int_t^T\|\sh\pa_y\phi(t')\|_{H^{\f{29}{4},0}_\Psi}^2\,dt'\\
&+2\left(\la-C(\eta^{-1}\e^{\f32}+1)\right)\int_t^T\dot\th(t')
\|\sh\phi(t')\|_{H^{\f{15}{2},0}_\Psi}^2\,dt'\leq  \int_t^T\dot\th(t')\|\sh H(t')\|_{H^{7,0}_\Psi}^2\,dt'.
\end{split}
\eeq}
\end{prop}

\begin{proof}
By taking $H^{\f{29}{4},0}_\Psi$ inner product of \eqref{eq: phi} with $\hbar\phi,$ we find
\begin{align*}
\hbar\langle\mathcal{L}^*\phi,\phi  \rangle_{H^{\f{29}{4},0}_\Psi}=\dot\th\hbar\langle H,\phi  \rangle_{H^{\f{29}{4},0}_\Psi}.
\end{align*}
Here and in all that follows, we always denote
\beq \label{S3eq2ad}
\langle a, b  \rangle_{H^{\sigma,0}_\Psi} \eqdefa {\rm Re}\int_0^\infty\int_{\R}e^{2\Psi}\D^\sigma a \bigr| \overline{\D^\sigma b}\,dx\,dy.
\eeq

It is easy to observe that
\begin{align*}
\int_t^T|\dot\th\hbar\langle H,\phi  \rangle_{H^{\f{29}{4},0}_\Psi}|\,dt'\leq \f12\int_t^T\dot\th(t')\|\sh\phi\|_{H^{\f{15}{2},0}_\Psi}^2\,dt'+\f12\int_t^T\dot\th(t')\|\sh H\|_{H^{7,0}_\Psi}^2\,dt'.
\end{align*}
While in view of \eqref{defcL*}, we  get, by using integrating by parts, that
\beq\label{S3eq3}
\begin{split}
\hbar&\langle\mathcal{L}^*\phi,\phi  \rangle_{H^{\f{29}{4},0}_\Psi}=-\f12\f{d}{dt}\|\sh\phi(t)\|_{H^{\f{29}{4},0}_\Psi}^2
-\f12\hbar'(t)\|\phi(t)\|_{H^{\f{29}{4},0}_\Psi}^2\\
&+\f12\int_{\R^2_+}\Bigl(\pa_t(e^{2\Psi})+\f{y^2}{2\t^2}e^{2\Psi}\Bigr)|\sh\D^{\f{29}{4}}\phi |^2dxdy+\la\dot\th(t)\|\sh\phi \|_{H^{\f{15}{2},0}_\Psi}^2\\
&+\|\sh\pa_y\phi\|_{H^{\f{29}{4},0}_\Psi}^2+\f{\hbar(t)}{2\t}\int_{\R^2_+}y e^{2\Psi} \D^{\f{29}{4}}\phi
  \bigl| \D^{\f{29}{4}}\pa_y\phi\,dx\,dy
-\hbar \bigl\langle \pa_x(T^\h_u)^*\phi  ,\phi  \bigr\rangle_{H^{\f{29}{4},0}_\Psi}\\
&-\hbar\bigl\langle \bigl(\pa_y+\f{y}{2\t}\bigr)(T^\h_v)^*\phi, \phi  \bigr\rangle_{H^{\f{29}{4},0}_\Psi}-\f{\d(t)\hbar(t)}2
\bigl\langle \Lambda(D)\pa_x(T^\h_{D_xu})^*\phi, \phi  \bigr\rangle_{H^{\f{29}{4},0}_\Psi}.
\end{split}\eeq
A direct computation gives
\begin{align*}
\f12\Bigl(\pa_t(e^{2\Psi})+\f{y^2}{2\t^2}e^{2\Psi}\Bigr)=\f{y^2}{8\t^2}e^{2\Psi},
\end{align*}
and applying H\"older inequality yields
\begin{align*}
\f{\hbar}{2\t}\bigl|\int_{\R^2_+}y e^{2\Psi}&\D^{\f{29}{4}}\phi \bigr| \D^{\f{29}{4}}\pa_y\phi~  dxdy\bigr|\\
&\leq \f1{2 }\|\sh\pa_y\phi\|_{H^{\f{29}{4},0}_\Psi}^2+\f{1 }{8\t^2}\int_{\R^2_+}|\sh y^2 e^{\Psi}\D^{\f{29}{4}}\phi|^2  dxdy.
\end{align*}
As a result, it comes out
\begin{align*}
&\f12\int_{\R^2_+} \Bigl(\pa_t(e^{2\Psi})+\f{y^2}{2\t^2}e^{2\Psi}\Bigr)|\sh\D^{\f{29}{4}}\phi |^2dxdy\\
&\qquad+\f{\hbar}{2\t}\int_{\R^2_+} y e^{2\Psi}\D^{\f{29}{4}}\phi \bigr| \D^{\f{29}{4}}\pa_y\phi~  dxdy \geq  -\f1{2 }\|\sh \pa_y\phi\|_{H^{\f{29}{4},0}_\Psi}^2.
\end{align*}

Next let us turn to the estimate of the remaining terms in \eqref{S3eq3}.\smallskip

\no\underline{$\bullet$ The estimate of $\hbar\bigl\langle \pa_x(T^\h_u)^*\phi  ,\phi  \bigr\rangle_{H^{\f{29}{4},0}_\Psi}.$}
By using a standard commutator's argument, we write
\beq \label{S3eq3a}
\begin{split}
\bigl\langle \pa_x&(T^\h_u)^*\phi, \phi  \bigr\rangle_{H^{\f{29}{4},0}_\Psi}
=-\f12\Bigl(\bigl\langle [\D^{\f{29}{4}}; (T_u^\h)^*]\phi ,\D^{\f{29}{4}}\pa_x\phi  \bigr\rangle_{L^2_\Psi}\\
&+ \bigl\langle \D^{\f{29}{4}}\phi ,\D^{\f{29}{4}}\pa_x\bigl((T_u^\h)^*-T_u^\h\bigr)\phi  \bigr\rangle_{L^2_\Psi}+\bigl\langle \D^{\f{29}{4}}\phi , [T_u^\h; \D^{\f{29}{4}}\pa_x]\phi  \bigr\rangle_{L^2_\Psi}\Bigr),
\end{split}\eeq
which implies
\begin{align*}
\bigl|\bigl\langle \pa_x (T^\h_u)^*\phi, \phi  \bigr\rangle_{H^{\f{29}{4},0}_\Psi}\bigr|
\leq C&\|\phi\|_{H^{\f{29}{4},0}_\Psi}\Bigl(\|[\D^{\f{29}{4}}; (T_u^\h)^*]\phi\|_{H^{1,0}_\Psi}\\
&+\|((T_u^\h)^*-T_u^\h)\phi\|_{H^{\f{33}{4},0}_\Psi}
+\|[T_u^\h; \D^{\f{29}{4}}\pa_x]\phi \|_{L^2_\Psi}\Bigr).
\end{align*}
It follows from Lemma \ref{lem: com1}  that
\begin{align*}
\|[\D^{\f{29}{4}}; (T_u^\h)^*]\phi\|_{H^{1,0}_\Psi} &\leq C\|u\|_{L_{\rv}^\infty(H^{\f32+}_\h)}\|\phi\|_{H^{\f{29}{4},0}_\Psi},\\
\|((T_u^\h)^*-T_u^\h)\phi\|_{H^{\f{33}{4},0}_\Psi}& \leq C\|u\|_{L_{\rv}^\infty(H^{\f32+}_\h)}^2\|\phi\|_{H^{\f{29}{4},0}_\Psi},\\
\|[T_u^\h; \D^{\f{29}{4}}\pa_x]\phi \|_{L^2_\Psi}&\leq C\|u\|_{L_{\rv}^\infty(H^{\f32+}_\h)}\|\phi\|_{H^{\f{29}{4},0}_\Psi},
\end{align*}
which  imply that
\begin{align}\label{S3eq3aq}
\bigl|\bigl\langle \pa_x(T^\h_u)^*\phi  ,\phi  \bigr\rangle_{H^{\f{29}{4},0}_\Psi}\bigr|\leq C\bigl(\|u\|_{L_\v^\infty(H^{\f32+}_\h)}+\|u\|_{L_\v^\infty(H^{\f32+}_\h)}^2\bigr)\|\phi\|_{H^{\f{29}{4},0}_\Psi}^2.
\end{align}
Here and in all that follows, we always denote $\sigma_+$ to be a constant which is slightly bigger than $\sigma.$

Then thanks to \eqref{S3eq-1} and \eqref{S3eq8}, we obtain
\begin{align*}
\hbar\bigl|\bigl\langle \pa_xT^*_u\phi  ,\phi  \bigr\rangle_{H^{\f{29}{4},0}_\Psi}\bigr|\leq&  C\bigl(\e^\f12\t^{\beta-\gamma_0-\f14}+\e^{\f32}\t^{\beta-2\gamma_0-\f12}\bigr)\dot\th(t)\|\sh\phi\|_{H^{\f{29}{4},0}_\Psi}^2.
\end{align*}

\no\underline{$\bullet$ The estimate of $\hbar\langle (\pa_y+\f{y}{2\t})(T^\h_v)^*\phi  ,\phi  \rangle_{H^{\f{29}{4},0}_\Psi}.$}
By using integration by parts and boundary condition $\phi|_{y=0}=0,\quad \phi|_{y\to +\infty}=0,$  and using  \eqref{S3eq8}, we find
\begin{align*}
\hbar\bigl|\bigl\langle  \bigl(\pa_y+\f{y}{2\t}\bigr)(T^\h_v)^*\phi  ,\phi  \bigr\rangle_{H^{\f{29}{4},0}_\Psi}\bigr|
=&\hbar\bigl|\bigl\langle (T^\h_v)^*\phi  ,\pa_y\phi  \bigr\rangle_{H^{\f{29}{4},0}_\Psi}\bigr|\\
\leq &C\| v\|_{L^\infty_\v(H^{\f12+}_\h)}\|\sh\phi\|_{H^{\f{29}{4},0}_\Psi}\|\sh\pa_y\phi\|_{H^{\f{29}4,0}_\Psi}\\
\leq&C\e\t^{-\ga_0+\f14}\|\sh\phi\|_{H^{\f{15}{2},0}_\Psi}\|\sh\pa_y\phi\|_{H^{\f{29}4,0}_\Psi}.
\end{align*}
Applying Young's inequality gives
\begin{align*}
\hbar\bigl|\bigl\langle  \bigl(\pa_y+\f{y}{2\t}\bigr)(T^\h_v)^*\phi  ,\phi  \bigr\rangle_{H^{\f{29}{4},0}_\Psi}\bigr|
\leq &C\eta^{-1}\e^{\f32}\t^{\beta-2\gamma_0+\f12}\dot\th(t)\|\sh\phi\|_{H^{\f{15}{2},0}_\Psi}^2+\f{\eta }2\|\sh\pa_y\phi\|_{H^{\f{29}4,0}_\Psi}^2,
\end{align*}

\no\underline{$\bullet$ The estimate of $\f{\hbar\d(t)}2\bigl\langle \Lambda(D)\pa_x(T^\h_{D_xu})^*\phi  ,\phi  \bigr\rangle_{H^{\f{29}{4},0}_\Psi}.$}
 It follows from Lemma \ref{lem: T_fg}, \eqref{S3eq-1} and \eqref{S3eq8}  that
\begin{align*}
\f{\d(t)\hbar(t)}2\bigl|\bigl\langle \Lambda(D)\pa_x(T^\h_{D_xu})^*\phi    ,\phi  \bigr\rangle_{H^{\f{29}{4},0}_\Psi}\bigr|\leq& C\|
u\|_{L^\infty_\v(H^{\f32+}_\h)}\|\sh\phi\|_{H^{\f{15}2,0}_\Psi}^2\\
\leq & C\e^\f12\t^{\beta-\gamma_0-\f14}\dot\th(t)\|\sh\phi\|_{H^{\f{15}2,0}_\Psi}^2.
\end{align*}

Notice that $\phi|_{t=T}=0,$
by inserting the above estimates into  \eqref{S3eq3} over $[t,T]$ and then integrating  the resulting
inequality, we arrive at
\begin{align*}
\f12\|&\sh\phi(t)\|_{H^{\f{29}{4},0}_\Psi}^2
-\f12\int_t^T\|\sqrt{\hbar'}\phi(t')\|_{H^{\f{29}{4},0}_\Psi}^2\,dt'\\
&+\la\int_t^T\dot\th(t')\|\sh\phi(t')\|_{H^{\f{15}{2},0}_\Psi}^2\,dt'+ \f{1-\eta}2\int_t^T\|\sh\pa_y\phi(t')\|_{H^{\f{29}{4},0}_\Psi}^2\,dt'\\
\leq &\f12\int_t^T\dot\th(t')\|\sh H(t')\|_{H^{7,0}_\Psi}^2\,dt'
+C(\eta^{-1}\e^{\f32}+1)\int_t^T \dot\th(t')\|\sh\phi(t')\|_{H^{\f{15}{2},0}_\Psi}^2\,dt',
\end{align*}
if $\b$ in \eqref{S3eq-1} satisfying  $\beta\leq \gamma_0+\f14.$ This leads to \eqref{S3eq2}, and we complete
the proof of Proposition \ref{lem: phi}.
\end{proof}

In particular, if we take  $\hbar(t)=\w{t}^{\f{1-\eta}2}$ in \eqref{S3eq2}, we obtain

\begin{col}
\label{S3col1}
{\sl Under the assumptions of Proposition \ref{lem: phi}, there exist  $\e_0$  and $\bar{\la}_1$ so that
for $\e\leq\e_0,$ $\eta\geq 2\e^{\f32}$ and $\la\geq \bar{\la}_1,$  \eqref{S3eq2a} holds.
}
\end{col}

\begin{proof} Indeed it follows from Lemma \ref{lem2.1} that
\beno
\f12{\w{t}^{-\f{1+\eta}2}}\|\phi(t)\|_{H^{\f{29}{4},0}_\Psi}^2 \leq \|\w{t}^{\f{1-\eta}4} \pa_y\phi(t)\|_{H^{\f{29}{4},0}_\Psi}^2,\eeno
so that
\beq\label{S3eq2p}
-\f{1-\eta}2\int_t^T\|\w{t'}^{-\f{1+\eta}4} \phi(t')\|_{H^{\f{29}{4},0}_\Psi}^2\,dt'+ (1-\eta)\int_t^T\|\w{t'}^{\f{1-\eta}4}\pa_y\phi(t')\|_{H^{\f{29}{4},0}_\Psi}^2\,dt'\geq 0.
\eeq
Then by taking   $\hbar(t)=\w{t}^{\f{1-\eta}2}$ in \eqref{S3eq2} and using the above inequality,
 we achieve
 \beq\label{S3eq2as}
\begin{split}
\|\w{t}^{\f{1-\eta}4}\phi(t)\|_{H^{\f{29}{4},0}_\Psi}^2+2\left(\la-C(\eta^{-1}\e^{\f32}+1)\right)
&\int_t^T\dot\th(t')
\|\w{t'}^{\f{1-\eta}4}\phi(t')\|_{H^{\f{15}{2},0}_\Psi}^2\,dt'\\
&\qquad
  \leq \int_t^T\dot\th(t')\|\w{t'}^{\f{1-\eta}4} H(t')\|_{H^{7,0}_\Psi}^2\,dt'.
\end{split}
\eeq
Then for $\e$ small enough and $\eta\geq 2\e^{\f32},$ we have $\eta^{-1}\e^{\f32}+1\leq \f32$ Then
taking $\la\geq \bar{\la}_1\eqdefa 3C$ in \eqref{S3eq2as} leads to \eqref{S3eq2a}.
\end{proof}

\subsection{The estimate of $H$}
In this subsection,  we shall present the {\it a priori} estimate of $H.$  Before proceeding, we  first derive
 the estimate of source term $f$ in \eqref{eq: tu_Phi}.

\begin{lem}\label{lem: f}
{\sl Let $f$ be given by \eqref{eq: f_3}. Then for  any $s>0,$ there holds
\beq\label{S3eq4}
\|f\|_{H^{s,0}_\Psi}\leq C\t^\f14\Bigl(\|\pa_y G_\Phi\|_{H^{\f52+,0}_{\Psi}}\|u_\Phi\|_{H^{s,0}_\Psi}+
\|G_\Phi\|_{H^{\f52+,0}_{\Psi}}\|\pa_y u_\Phi\|_{H^{s-\f12,0}_\Psi}\Bigr).
\eeq}
\end{lem}

\begin{proof}
Recall from  \eqref{eq: f_3} that $f=-f_1-f_2-f_3$ with $f_i~(i=1,2,3)$ being given respectively by \eqref{eq: (f_1,f_2)} and \eqref{eq: f_3}.
We first observe from the proof of Lemma \ref{lem: phi} that for any $\ga\in (0,1)$ and $\sigma\in\R$
\beq\label{S3eq4a}
\begin{split}
\|e^{\ga\Psi}u(t)\|_{L_\v^\infty(H^{\sigma}_\h)}\leq
 C\t^\f14\|\pa_y G(t)\|_{H^{\sigma,0}_{\Psi}} \andf \|e^{\ga\Psi} v\|_{L^\infty_\v(H^{\sigma}_\h)}\leq C\t^\f14 \|G(t)\|_{H^{\sigma+1,0}_\Psi}.
\end{split}\eeq
Thanks to \eqref{S3eq4a}, we
deduce from  Lemma \ref{lem: com3}   that
\begin{align*}
\|f_1\|_{H^{s,0}_\Psi}\leq C\|u_\Phi\|_{L^\infty_\v(H^{\f52+}_\h)}\|u_\Phi\|_{H^{s,0}_\Psi}\leq& C\t^{\f14}\|\pa_y G_\Phi\|_{H^{\f52+,0}_\Psi}\|u_\Phi\|_{H^{s,0}_\Psi}.
\end{align*}
Similarly due to $v=\int_y^\infty \pa_xu\,dz,$ it follows from Lemma \ref{lem: com2}, Lemma \ref{lem: com3} and Lemma \ref{lem2.2} that
\begin{align*}
\|f_2\|_{H^{s,0}_\Psi}\leq& C\bigl(\|\pa_y u_\Phi\|_{L^2_{\v,\f34\Psi}(H^{\f52+}_\h)}\|e^{\f{\Psi}4}v_\Phi\|_{L^\infty_\v(H^{s-1}_\h)}+\|v_\Phi\|_{L^\infty_\v(H^{\f32+}_\h)}\|\pa_y u_\Phi\|_{H^{s-\f12,0}_\Psi}\bigr)\\
\leq&C\t^\f14\bigl(\|\pa_y u_\Phi\|_{H^{\f52+,0}_{\f34\Psi}}\|u_\Phi\|_{H^{s,0}_\Psi}+\|u_\Phi\|_{H^{\f52+,0}_{\f34\Psi}}\|\pa_y u_\Phi\|_{H^{s-\f12,0}_\Psi}\bigr)\\
\leq&C\t^\f14\bigl(\|\pa_y G_\Phi\|_{H^{\f52+,0}_{\Psi}}\|u_\Phi\|_{H^{s,0}_\Psi}+\|G_\Phi\|_{H^{\f52+,0}_{\Psi}}\|\pa_y u_\Phi\|_{H^{s-\f12,0}_\Psi}\bigr).
\end{align*}
Finally by applying Lemma \ref{lem: T_fg} and Lemma \ref{lem2.2}, we arrive at
\begin{align*}
\|f_3\|_{H^{s,0}_\Psi}\leq& C\bigl(\|\pa_xu_\Phi\|_{L^\infty_\v(H^{\f12+}_\h)}\|u_\Phi\|_{H^{s,0}_\Psi}
+\|v_\Phi\|_{L^\infty_\v(H^{1+}_\h)}\|\pa_y u_\Phi\|_{H^{s-\f12,0}_\Psi}\bigr)\\
\leq&C\t^\f14\bigl(\|\pa_y G_\Phi\|_{H^{\f52+,0}_{\Psi}}\|u_\Phi\|_{H^{s,0}_\Psi}+\|G_\Phi\|_{H^{\f52+,0}_{\Psi}}\|\pa_y u_\Phi\|_{H^{s-\f12,0}_\Psi}\bigr).
\end{align*}
By summarizing  the above estimates, we conclude the proof of \eqref{S3eq4}.
\end{proof}

Next let us turn to the estimates of $H.$

\begin{prop}\label{lem: H}
{\sl Let $H$ be a smooth enough solution of \eqref{eq: H} on $[0,T^\star]$ which decays to zero
 as $y$ approaching $\infty.$ Then if  $\beta\leq\min\{\gamma_0+\f14,\f12(\gamma_0+\f54)\}$  and  $\la\e^\f12\leq C,$ for any
 $T\in [0,T^\star],$ we have
 \beq\label{S3eq4b}
\begin{split}
\|\sh &H(T)\|_{H^{\f{27}{4},0}_\Psi}^2+2\bigl(\la(1-C\e^{\f12})-C(1+\eta^{-1}\e)\e^{\f12}\bigr)\int_0^T\dot\th(t)\|\sh H(t)\|_{H^{7,0}_\Psi}^2\,dt\\
&-\int_0^T\|\sqrt{\hbar'} H(t)\|_{H^{\f{27}{4},0}_\Psi}^2\,dt+\bigl(1-\f{\eta}{4}\bigr)\int_0^T\|\sh \pa_yH(t)\|_{H^{\f{27}{4},0}_\Psi}^2\,dt\\
\leq&C\bigl(1+\la\e^{\f12}\bigr)\int_0^T \dot\th(t)\|\sh \phi(t)\|_{H^{\f{15}{2},0}_\Psi}^2\,dt+\|\sh(0)u_\Phi(0)\|_{H^{\f{25}{4}}_\Psi}^2+\|\sh(0)\phi(0)\|_{H^{\f{29}{4}}_\Psi}^2\\
&+C\e^{\f32}\int_0^T\Bigl(\t^{-\gamma_0-\f14}\|\sh u_\Phi\|_{H^{6,0}_\Psi}^2+\|\sh\pa_y u_\Phi\|_{H^{\f{11}{2},0}_\Psi}^2\Bigr)\, dt.
\end{split} \eeq }
\end{prop}

\begin{proof}
In view of \eqref{eq: tu_Phi} and \eqref{defcL}, we write
\begin{align*}
\mathcal{L}u_\Phi+T_{\pa_y u}^\h v_\Phi+\f{\d(t)}2T^\h_{\pa_yD_xu}\Lambda(D_x)v_\Phi=f.
\end{align*}
By  taking $H^{\f{27}{4},0}_\Psi$ inner product of the above equation with $\hbar(t)\phi$ and integrating
the resulting equality over $[0,T],$ we find
\begin{align*}
\int_0^T\hbar(t)\bigl\langle\mathcal{L}u_\Phi+T^\h_{\pa_y u}v_\Phi+\f{\d(t)}2T^\h_{\pa_yD_xu}\Lambda(D_x)v_\Phi,\phi\bigr\rangle_{H^{\f{27}{4},0}_\Psi}\,dt
=\int_0^T\hbar(t)\langle f ,\phi\rangle_{H^{\f{27}{4},0}_\Psi}\,dt.
\end{align*}
We first get, by applying Lemma \ref{lem: f} with $s=6,$ that
\begin{align*}
\int_0^T&\hbar(t)\bigl|\langle f,\phi\rangle_{H^{\f{27}{4},0}_\Psi}\bigr|\,dt\leq  \int_0^T\|\sh f\|_{H^{6,0}_\Psi}\|\sh\phi\|_{H^{\f{15}{2},0}_\Psi}\,dt\\
\leq & C\int_0^T  \t^\f14\bigl(\|\pa_y G_\Phi\|_{H^{\f52+,0}_{\Psi}}\|\sh u_\Phi\|_{H^{6,0}_\Psi}+\|G_\Phi\|_{H^{\f52+,0}_{\Psi}}\|\sh\pa_y u_\Phi\|_{H^{\f{11}{2},0}_\Psi}\bigr)  \|\sh\phi\|_{H^{\f{15}{2},0}_\Psi}\,dt,
\end{align*}
from which,  \eqref{assume: 1} and \eqref{S3eq-1}, we infer
\begin{align*}
\int_0^T&\hbar(t)\bigl|\langle f,\phi\rangle_{H^{\f{27}{4},0}_\Psi}\bigr|\,dt \\
\leq& C\int_0^T\dot\th(t)\|\sh\phi\|_{H^{\f{15}{2},0}_\Psi}^2dt+\e^2\int_0^T\f{\t^\f12}{\dot\th(t)}\t^{-2\gamma_0}\bigl(\t^{-1}\|\sh u_\Phi\|_{H^{6,0}_\Psi}^2+\|\sh\pa_y u_\Phi\|_{H^{\f{11}{2},0}_\Psi}^2\bigr)\,dt\\
\leq& C\int_0^T\dot\th(t)\|\sh\phi\|_{H^{\f{15}{2},0}_\Psi}^2dt+\e^{\f32}\int_0^T\t^{\beta-2\gamma_0+\f12}\bigl(\t^{-1}\|\sh u_\Phi\|_{H^{6,0}_\Psi}^2+\|\sh\pa_y u_\Phi\|_{H^{\f{11}{2},0}_\Psi}^2\bigr)\, dt.
\end{align*}
Due to
$\beta\leq \gamma_0+\f14$ and $\ga_0>1,$ we achieve
\beq \label{S3eq4c}
\begin{split}
\int_0^T\hbar(t)\bigl|\langle f,\phi\rangle_{H^{\f{27}{4},0}_\Psi}\bigr|\,dt
\leq & C\int_0^T\dot\th(t)\|\sh\phi\|_{H^{\f{15}{2},0}_\Psi}^2dt\\
&+\e^{\f32} \int_0^T\bigl(\t^{-\gamma_0-\f14}\|\sh u_\Phi\|_{H^{6,0}_\Psi}^2+\|\sh\pa_y u_\Phi\|_{H^{\f{11}{2},0}_\Psi}^2\bigr)\, dt.
\end{split}
\eeq

On the other hand, thanks to \eqref{eq: phi}, we get, by using integration by parts, that
\begin{align*}
\int_0^T\hbar\langle \mathcal{L}u_\Phi,\phi\rangle_{H^{\f{27}{4},0}_\Psi}\,dt
=&\int_0^T\hbar\langle u_\Phi,\mathcal{L}^*\phi\rangle_{H^{\f{27}{4},0}_\Psi}\,dt-\hbar(0)\langle u_\Phi(0),\phi(0)\rangle_{H^{\f{27}{4},0}_\Psi}\\
=&\int_0^T\hbar\langle \dot
\th u_\Phi,H\rangle_{H^{\f{27}{4},0}_\Psi}\,dt-\hbar(0)\langle u_\Phi(0),\phi(0)\rangle_{H^{\f{27}{4},0}_\Psi},
\end{align*}
Substituting \eqref{tu_Phi} into the above equality yields
\beq\label{eq: LHS-1}
\begin{split}
\int_0^T\hbar\langle \mathcal{L}u_\Phi,\phi\rangle_{H^{\f{27}{4},0}_\Psi}\,dt=&\int_0^T\hbar\Bigl(\langle \mathcal{L}H, H\rangle_{H^{\f{27}{4},0}_\Psi}-\bigl\langle T^\h_{\pa_yu}\pa_x\int_y^\infty Hdz,H\bigr\rangle_{H^{\f{27}{4},0}_\Psi}\\
&\quad-\f{\d(t)}2\bigl\langle T^\h_{\pa_yD_xu}\Lambda(D_x)\pa_x\int_y^\infty Hdz,H\bigr\rangle_{H^{\f{27}{4},0}_\Psi}\\
&\quad+\langle T_{\pa_yv} H ,H\rangle_{H^{\f{27}{4},0}_\Psi}\Bigr)\,dt-\hbar(0)\langle u_\Phi(0),\phi(0)\rangle_{H^{\f{27}{4},0}_\Psi}.
\end{split}\eeq

Along the same line, by applying \eqref{v_Phi} and then \eqref{eq: phi}, we find
\beq\label{eq: LHS-2}
\begin{split}
&\int_0^T\hbar\langle T^\h_{\pa_y u}v_\Phi,\phi\rangle_{H^{\f{27}{4},0}_\Psi}\,dt
=\int_0^T\f{\hbar}{\dot\th}\Bigl(\bigl\langle T^\h_{\pa_y u}\mathcal{L}\pa_x\int_y^\infty Hdz,\phi\bigr\rangle_{H^{\f{27}{4},0}_\Psi}
\\
&\quad-\langle T^\h_{\pa_y u}T^\h_{\pa_x v}H,\phi\rangle_{H^{\f{27}{4},0}_\Psi}+\bigl\langle T^\h_{\pa_y u}T^\h_{\pa_xu}\pa_x\int_y^\infty Hdz,\phi\bigr\rangle_{H^{\f{27}{4},0}_\Psi}\\
&\quad+\f{\d(t)}{2}\bigl\langle T^\h_{\pa_y u}T^\h_{\pa_xD_xu}\Lambda(D_x)\pa_x\int_y^\infty Hdz,\phi\bigr\rangle_{H^{\f{27}{4},0}_\Psi}\Bigr)\,dt\\
&=\int_0^T\hbar\Bigl(\bigl\langle T^\h_{\pa_y u}\pa_x\int_y^\infty H\,dz,H\bigr\rangle_{H^{\f{27}{4},0}_\Psi}
-\pa_t\left(\f{1}{\dot\th}\right)\bigl\langle T^\h_{\pa_y u}\pa_x\int_y^\infty Hdz,\phi\bigr\rangle_{H^{\f{27}{4},0}_\Psi}\Bigr)\,dt\\
&\quad+\int_0^T\f{\hbar}{\dot\th}\Bigl(\bigl\langle [T^\h_{\pa_y u}; \mathcal{L}]\pa_x\int_y^\infty Hdz,\phi\bigr\rangle_{H^{\f{27}{4},0}_\Psi}
+\langle T^\h_{\pa_y u}T^\h_{\pa_xu}\pa_x\int_y^\infty Hdz,\phi\rangle_{H^{\f{27}{4},0}_\Psi}\\
&\quad-\langle T^\h_{\pa_y u}T^\h_{\pa_x v}H,\phi\rangle_{H^{\f{27}{4},0}_\Psi}
+\f{\d(t)}{2}\bigl\langle T^\h_{\pa_y u}T^\h_{\pa_xD_xu}\Lambda(D_x)\pa_x\int_y^\infty Hdz,\phi\bigr\rangle_{H^{\f{27}{4},0}_\Psi}\Bigr)\,dt,
\end{split}\eeq
and
\beq\label{eq: LHS-3}
\begin{split}
\f12\int_0^T&\hbar(t){\d(t)}\bigl\langle T^\h_{\pa_yD_xu}\Lambda(D_x)v_\Phi,\phi\bigr\rangle_{H^{\f{27}{4},0}}\,dt\\
=&\f12\int_0^T{\d(t)}\hbar(t)\bigl\langle T^\h_{\pa_yD_xu}\Lambda(D_x)\pa_x\int_y^\infty H\,dz,H\bigr\rangle_{H^{\f{27}{4},0}_\Psi}\,dt\\
&-\int_0^T\pa_t\left(\f{\d(t)}{2\dot\th}\right)\hbar(t)\bigl\langle T^\h_{\pa_yD_xu}\Lambda(D_x)\pa_x\int_y^\infty Hdz,\phi\bigr\rangle_{H^{\f{27}{4},0}_\Psi}\,dt\\
&+\int_0^T\f{\hbar(t)\d(t)}{2\dot\th}\Bigl(\bigl\langle [T^\h_{\pa_yD_xu}\Lambda(D_x); \mathcal{L}]\pa_x\int_y^\infty Hdz,\phi\bigr\rangle_{H^{\f{27}{4},0}_\Psi}\\
&\quad+\bigl\langle T^\h_{\pa_yD_xu}\Lambda(D_x)T^\h_{\pa_xu}\pa_x\int_y^\infty Hdz,\phi\bigr\rangle_{H^{\f{27}{4},0}_\Psi}
-\bigl\langle T^\h_{\pa_yD_xu}\Lambda(D_x)T^\h_{\pa_x v}H,\phi\bigr\rangle_{H^{\f{27}{4},0}_\Psi}\\
&\qquad\qquad\qquad\qquad+\f{\d(t)}{2}\bigl\langle T^\h_{\pa_yD_xu}\Lambda(D_x)T^\h_{\pa_xD_xu}\Lambda(D_x)\pa_x\int_y^\infty Hdz,\phi\bigr\rangle_{H^{\f{27}{4},0}_\Psi}\Bigr)\,dt.
\end{split}\eeq
The most trouble terms above are  the second and third terms in \eqref{eq: LHS-1} which are fortunately  canceled respectively by the first term in
 \eqref{eq: LHS-2} and \eqref{eq: LHS-3}.
By summarizing \eqref{eq: LHS-1}, \eqref{eq: LHS-2} and  \eqref{eq: LHS-3},  we arrive at
\beq\label{S3eq5}
\begin{split}
\int_0^T&\hbar\bigl\langle\mathcal{ L}u_\Phi+T^\h_{\pa_y u}v_\Phi+\f{\d(t)}2T^\h_{\pa_yD_xu}\Lambda(D_x)v_\Phi,\phi\bigr\rangle_{H^{\f{27}{4},0}_\Psi}\,dt\\
=&\int_0^T\hbar\langle \mathcal{L}H, H\rangle_{H^{\f{27}{4},0}_\Psi}\,dt+
\int_0^T\hbar\langle T^\h_{\pa_yv} H ,H\rangle_{H^{\f{27}{4},0}_\Psi}\,dt
-\hbar(0)\langle u_\Phi(0),\phi(0)\rangle_{H^{\f{27}{4},0}_\Psi}\\
&-\int_0^T\hbar\Bigl(\pa_t\left(\f{1}{\dot\th}\right)
\bigl\langle T^\h_{\pa_y u}\pa_x\int_y^\infty Hdz,\phi\bigr\rangle_{H^{\f{27}{4},0}_\Psi}\\
&\qquad\qquad\qquad+\pa_t\left(\f{\d(t)}{2\dot\th(t)}\right)\bigl\langle T^\h_{\pa_yD_xu}\Lambda(D_x)\pa_x\int_y^\infty Hdz,\phi\bigr\rangle_{H^{\f{27}{4},0}_\Psi}\Bigr)\,dt,
\end{split}\eeq
\begin{align*}
&+\int_0^T\f{\hbar}{\dot\th}\Bigl(\bigl\langle [T^\h_{\pa_y u};\mathcal{L}]\pa_x\int_y^\infty Hdz,\phi\bigr\rangle_{H^{\f{27}{4},0}_\Psi}
+\bigl\langle T^\h_{\pa_y u}T^\h_{\pa_xu}\pa_x\int_y^\infty Hdz,\phi\bigr\rangle_{H^{\f{27}{4},0}_\Psi}\\
&-\bigl\langle T^\h_{\pa_y u}T^\h_{\pa_x v}H,\phi\bigr\rangle_{H^{\f{27}{4},0}_\Psi}\Bigr)\,dt+
\int_0^T\f{\hbar\d(t)}{2\dot\th}\Bigl(\bigl\langle T^\h_{\pa_y u}T^\h_{\pa_xD_xu}\Lambda(D_x)\pa_x\int_y^\infty Hdz,\phi\bigr\rangle_{H^{\f{27}{4},0}_\Psi}\\
&
+\bigl\langle [T^\h_{\pa_yD_xu}\Lambda(D_x);\mathcal{L}]\pa_x\int_y^\infty Hdz,\phi\bigr\rangle_{H^{\f{27}{4},0}_\Psi}\\
&+\bigl\langle T^\h_{\pa_yD_xu}\Lambda(D_x)T^\h_{\pa_xu}\pa_x\int_y^\infty Hdz,\phi\bigr\rangle_{H^{\f{27}{4},0}_\Psi}
-\bigl\langle T^\h_{\pa_yD_xu}\Lambda(D_x)T^\h_{\pa_x v}H,\phi\bigr\rangle_{H^{\f{27}{4},0}_\Psi}\Bigr)\,dt\\
&+\int_0^T\f{\d(t)^2}{4\dot\th}\hbar\bigl\langle T^\h_{\pa_yD_xu}\Lambda(D_x)T^\h_{\pa_xD_xu}\Lambda(D_x)\pa_x\int_y^\infty H\,dz,\phi\bigr\rangle_{H^{\f{27}{4},0}_\Psi}\,dt\\
\eqdefa&E_1+\cdots+E_{13}.
\end{align*}

By virtue of \eqref{S3eq4c}, to prove \eqref{S3eq4b}, we are left with the estimates of $E_i~(i=1,\cdots,13),$
which we present below.

\begin{lem}\label{S3lem1}
{\sl Let $\beta\leq \gamma_0+\f14.$ Then for any small positive constant $\eta,$ one has
\beq\label{S3eq7}
\begin{split}
E_1\geq& \f12\|\sh H(T)\|_{H^{\f{27}{4},0}_\Psi}^2+\bigl(\la-C(1+\eta^{-1}\e)\e^{\f12}\bigr)\int_0^T\dot\th(t)\|\sh H(t)\|_{H^{7,0}_\Psi}^2\,dt\\
&-\f12\int_0^T\|\sqrt{\hbar'} H(t)\|_{H^{\f{27}{4},0}_\Psi}^2\,dt
+\bigl(\f{1}2-\f{\eta}8\bigr)\int_0^T\|\sh \pa_yH(t)\|_{H^{\f{27}{4},0}_\Psi}^2\,dt .
\end{split}\eeq
}
\end{lem}

\begin{lem}\label{S3lem2}
{\sl Let $\beta\leq \min\bigl\{\ga_0+\f14, \f12\left(\gamma_0+\f54\right)\bigr\}.$  Then one has
\begin{subequations} \label{S3eq7a}
\begin{gather} \label{S3eq7b}
\sum_{i=4}^6|E_i|+|E_{10}|\leq C\bigl(1+\la\e^{\f12}\bigr)\int_0^T\dot\th(t)\bigl(\|\sh H(t')\|_{H^{7,0}_\Psi}^2+
\|\sh \phi(t)\|_{H^{\f{15}{2},0}_\Psi}^2\bigr)\,dt,\\
\label{S3eq7c}
\sum_{i=7}^9|E_i|+\sum_{i=11}^{13}|E_{i}|\leq C\e\int_0^T\dot\th(t)\bigl(\|\sh H(t)\|_{H^{7,0}_\Psi}^2+
\|\sh \phi(t)\|_{H^{\f{15}{2},0}_\Psi}^2\bigr)\,dt.
\end{gather}\end{subequations}
}
\end{lem}

Let us admit the above lemmas for the time being, and continue our proof of Proposition \ref{lem: H}.

Thanks to Lemmas \ref{S3lem1} and \ref{S3lem2}, it remains to handle  the estimates of $E_2$ and $E_3$ in \eqref{S3eq5}.
Indeed we get, by applying Lemma \ref{lem: T_fg},
 \eqref{S3eq8}  and \eqref{S3eq-1}, that
\begin{align*}
| E_2|\leq &\int_0^T\|\pa_y v\|_{L^\infty_\v(H^{\f12+}_\h)}\|\sh H\|_{H^{\f{27}{4},0}_\Psi}^2\,dt\\
\leq &C\e\int_0^T\w{t}^{-\left(\ga_0+\f14\right)}\|\sh H(t)\|_{H^{\f{27}{4},0}_\Psi}^2\,dt\\
\leq& C\e^\f12 \int_0^T\t^{\beta-\gamma_0-\f14}\dot\th(t)\|\sh H(t)\|_{H^{\f{27}{4},0}_\Psi}^2\,dt,
\end{align*}
which together the fact: $\beta\leq \gamma_0+\f14,$ ensures that
\beq \label{S3eq7d}
| E_2| \leq C\e^\f12\int_0^T\dot\th(t)\|\sh H(t)\|_{H^{\f{27}{4},0}_\Psi}^2\,dt.\eeq
Finally
by using H\"older inequality, we have
\begin{align} \label{S3eq7e}
|E_3|\leq& \hbar(0)\|u_\Phi(0)\|_{H^{\f{25}{4}}_\Psi}\|\phi(0)\|_{H^{\f{29}{4}}_\Psi}\leq \f12\|\sh(0)u_\Phi(0)\|_{H^{\f{25}{4}}_\Psi}^2+\f12\|\sh(0)\phi(0)\|_{H^{\f{29}{4}}_\Psi}^2.
\end{align}

 Then under the assumption that $\beta\leq\min\bigl\{\gamma_0+\f14,\f12(\gamma_0+\f54)\bigr\},$
by substituting the estimates \eqref{S3eq7}-\eqref{S3eq7e} into \eqref{S3eq5} and using \eqref{S3eq4c},
we arrive at
\begin{align*}
\f12&\|\sh H(T)\|_{H^{\f{27}{4},0}_\Psi}^2+\bigl(\la-C(1+\eta^{-1}\e)\e^{\f12}\bigr)\int_0^T\dot\th(t)\|\sh H(t)\|_{H^{7,0}_\Psi}^2\,dt\\
&-\f12\int_0^T\|\sqrt{\hbar'} H(t)\|_{H^{\f{27}{4},0}_\Psi}^2\,dt+\bigl(\f{1}2-\f{\eta}{8}\bigr)\int_0^T\|\sh \pa_yH(t)\|_{H^{\f{27}{4},0}_\Psi}^2\,dt\\
\leq&C\bigl(1+\la\e^{\f12}\bigr)\int_0^T \dot\th(t)\Big(\|\sh H(t)\|_{H^{7,0}_\Psi}^2+\|\sh \phi(t)\|_{H^{\f{15}{2},0}_\Psi}^2\Big)\,dt+\f12\|\sh(0)u_\Phi(0)\|_{H^{\f{25}{4}}_\Psi}^2\\
&+\f12\|\sh(0)\phi(0)\|_{H^{\f{29}{4}}_\Psi}^2+C\e^{\f32}\int_0^T\Bigl(\t^{-\gamma_0-\f14}\|\sh u_\Phi\|_{H^{6,0}_\Psi}^2+\|\sh\pa_y u_\Phi\|_{H^{\f{11}{2},0}_\Psi}^2\Bigr)\, dt,
\end{align*} from which, we deduce \eqref{S3eq4b}. This  ends the proof of Proposition \ref{lem: H}.
 \end{proof}

Let us now present the proof of Lemma \ref{S3lem1}. The proof Lemma \ref{S3lem2} involves tedious calculation,
and we postpone it in the Appendix \ref{appa}.

\begin{proof}[Proof of Lemma \ref{S3lem1}]
Notice that $2\pa_t\Psi+(\pa_y\Psi)^2\leq 0,$ we get, by using integrating by parts, that
\beq \label{S3eq7f}
\bigl\langle (\pa_tH-\pa_y^2H) | e^{2\Psi} H \bigr\rangle_{L^2_+}\geq \f12\f{d}{dt}\|H(t)\|_{H^{\f{27}{4},0}_\Psi}^2+\f12\|\pa_yH\|_{H^{\f{27}{4},0}_\Psi}^2.
\eeq
Then in view of \eqref{defcL}, we find
\begin{align*}
E_1\geq\int_0^T\Bigl(&\f12\f{d}{dt}\|\sh H(t)\|_{H^{\f{27}{4},0}_\Psi}^2-\f{\hbar'}{2\hbar}\|\sqrt{\hbar} H(t)\|_{H^{\f{27}{4},0}_\Psi}^2\\
&+\la\dot\th\|\sh H\|_{H^{7,0}_\Psi}^2+\f12\|\sh \pa_yH\|_{H^{\f{27}{4},0}_\Psi}^2-\langle T^\h_u\pa_xH ,H\rangle_{H^{\f{27}{4},0}_\Psi}\\
&-\hbar\langle T_v^\h\pa_yH ,H\rangle_{H^{\f{27}{4},0}_\Psi}-\f{\d(t)}2\hbar\langle T^\h_{D_xu}\Lambda(D_x)\pa_xH ,H\rangle_{H^{\f{27}{4},0}_\Psi}\Bigr) \,dt.
\end{align*}
It follows from a similar derivation of \eqref{S3eq3aq}
 and  \eqref{S3eq8}, \eqref{S3eq-1} that
\begin{align*}
\hbar\bigl|\langle T^\h_u\pa_xH ,H\rangle_{H^{\f{27}{4},0}_\Psi}\bigr|\leq& C\bigl(\|u\|_{L_\v^\infty(H^{\f32+}_\h)}+\|u\|_{L_\v^\infty(H^{\f32+}_\h)}^2\bigr)\|\sh H\|_{H^{\f{27}{4},0}_\Psi}^2\\
\leq& C\bigl(\e^\f12\t^{\beta-\gamma_0-\f14}+\e^{\f32}\t^{\beta-2\gamma_0-\f12}\bigr)\dot\th(t)\|\sh H\|_{H^{\f{27}{4},0}_\Psi}^2.\end{align*}
Applying Lemma \ref{lem: T_fg} gives
\begin{align*}
\hbar \bigl|\langle T^\h_v\pa_yH ,H\rangle_{H^{\f{27}{4},0}_\Psi}\bigr|\leq& C\|v\|_{L^\infty_\v(H^{\f12+}_\h)}\|\sh H\|_{H^{7,0}_\Psi}
\|\sh\pa_yH\|_{H^{\f{13}{2},0}_\Psi}\\
\leq &C\eta^{-1}\e^{\f32}\t^{\beta-2\gamma_0+\f12}\dot\th(t)\|\sh H\|_{H^{7,0}_\Psi}^2+\f\eta8\|\sh\pa_yH\|_{H^{\f{13}{2},0}_\Psi}^2.
\end{align*}
While it follows from  Lemma \ref{lem: T_fg} and \eqref{S3eq8} that
\begin{align*}
\f{\d(t)}2\hbar\bigl|\langle T^\h_{D_xu}\Lambda(D_x)\pa_xH ,H\rangle_{H^{\f{27}{4},0}_\Psi}\bigr|
\leq &C\d\|u\|_{L_\v^\infty(H^{\f32+}_\h)}\|\sh H\|_{H^{7,0}_\Psi}^2\\
\leq & C\e^\f12\t^{\beta-\gamma_0-\f14}\dot\th(t)\|\sh H\|_{H^{7,0}_\Psi}^2.
\end{align*}
By summarizing the above estimates and  using the fact that $\beta\leq \gamma_0+\f14$, we complete
the proof of  \eqref{S3eq7}.
\end{proof}

Taking  $\hbar(t)=\w{t}^{\f{1-\eta}2}$ in \eqref{S3eq4b} gives rise to

\begin{col}\label{S3col2}
{\sl Under the assumption of Proposition \ref{lem: H}, there exist $c_0,\e_1\in(0,1)$ and $\la_1>0$ so that for
$\e\leq\e_1,$ $\eta\geq \e^{\f32},$ and $\la \geq \la_1,$ there holds \eqref{S3eq10}.
}
\end{col}

\begin{proof} By taking $\hbar(t)=\w{t}^{\f{1-\eta}2}$ in \eqref{S3eq4b} and using \eqref{S3eq2p}, we obtain for $\la\e^\f12\leq C$
\beq \label{S3eq20}
\begin{split}
\|\w{T}^{\f{1-\eta}4}& H(T)\|_{H^{\f{27}{4},0}_\Psi}^2+2\bigl(\la-C(1+\eta^{-1}\e)\e^{\f12}\bigr)\int_0^T\dot\th(t)\|\w{t}^{\f{1-\eta}4}  H(t)\|_{H^{7,0}_\Psi}^2\,dt\\
&+\f{3}{4}\eta\int_0^t\|\w{t}^{\f{1-\eta}4} \pa_yH(t')\|_{H^{\f{27}{4},0}_\Psi}^2\,dt'\\
\leq & \|u_\Phi(0)\|_{H^{\f{25}{4}}_\Psi}^2+\|\phi(0)\|_{H^{\f{29}{4}}_\Psi}^2
+C\int_0^T
\dot\th(t)
\|\w{t}^{\f{1-\eta}4}  \phi(t)\|_{H^{\f{15}{2},0}_\Psi}^2\,dt\\
&+C\e^{\f32}\int_0^T\Bigl(\t^{-\gamma_0-\f14}\|\sh u_\Phi\|_{H^{6,0}_\Psi}^2+\|\sh\pa_y u_\Phi\|_{H^{\f{11}{2},0}_\Psi}^2\Bigr)\,dt.
\end{split}\eeq
Observing from Lemma \ref{lem2.1} that
\begin{align*}
\f{3}{4}\eta\int_0^t\|\w{t}^{\f{1-\eta}4} \pa_yH(t')\|_{H^{\f{27}{4},0}_\Psi}^2\,dt'\geq c_0\eta\int_0^t\|\w{t}^{-\f{3+\eta}4}{y} H(t)\|_{H^{\f{27}{4},0}_\Psi}^2\,dt.
\end{align*}
And for $\la\geq \max\bigl(C, \bar{\la}_1\bigr),$ we deduce from  Corollary \ref{S3col1} that
\begin{align*}
\|\phi(0)\|_{H^{\f{29}{4}}_\Psi}^2
+C\int_0^T
\dot\th(t)
\|\w{t}^{\f{1-\eta}4}  \phi(t)\|_{H^{\f{15}{2},0}_\Psi}^2\,dt\leq \int_0^T\dot\th(t)\|\w{t}^{\f{1-\eta}4}  H(t)\|_{H^{7,0}_\Psi}^2\,dt.
\end{align*}
By inserting the above two inequalities into \eqref{S3eq20} and taking $\e\leq\e_1,$ $\eta\geq \e^{\f32},$ and $\la \geq \la_1\eqdefa \max\bigl(2(1+C), \bar{\la}_1\bigr),$
 we obtain \eqref{S3eq10}.
\end{proof}

\subsection{The proof of Proposition \ref{S1prop1}}
By combining Corollaries \ref{S3col1} with \ref{S3col2}, we conclude the proof of Proposition \ref{S1prop1}.

\setcounter{equation}{0}
\section{The Gevrey estimates of $u_\Phi$ and $\p_yu_\Phi$}  \label{Sect5A}

\subsection{The Gevrey  estimate of $u_\Phi$}\label{Sect5}

Motivated by \cite{DG19} concerning the local well-posedness of Prandtl system with initial data in the optimal
Gevery regularity and also \cite{PZ5} concerning the decay of the global analytical solutions to \eqref{eq: Prandtl},  we shall  present the {\it a priori} time-decay estimates for $u_\Phi$ in this subsection. The main result states as follows:

\begin{prop}\label{pro: tu}
{\sl  Let $\hbar(t)$ be a non-negative and non-decreasing function on $[0,T]$ with $T\leq T^\star.$ Then   if
$\beta\leq\min\bigl\{\f{1}{2}(\gamma_0+\f54),\f23 \gamma_0+\f12, \gamma_0+\f14,\f43\gamma_0\bigr\}$, for any $\eta\in (0,\eta_1)$
with $\eta_1$ being sufficiently small, there holds
\beq\label{S4eq1}
\begin{split}
(1&-\eta\bigr)\|\sh u_\Phi(T)\|_{H^{\f{11}{2},0}_\Psi}^2+2\bigl(\la(1-C\e^\f12)-C(1+\e^\f14\eta^{-1})\bigr)\int_0^T\dot\th(t)\|\sh u_\Phi(t)\|_{H^{\f{23}{4},0}_\Psi}^2\,dt\\
&-\int_0^T\|\sqrt{\hbar'} u_\Phi(t)\|_{H^{\f{11}{2},0}_\Psi}^2\,dt+\bigl(1-\f\eta 4\bigr)\int_0^T \|\sh\pa_yu_\Phi(t)\|_{H^{\f{11}{2},0}_\Psi}^2\,dt\\
\leq & C\|\sh(0) u_\Phi(0)\|_{H^{\f{11}{2},0}_\Psi}^2+C\eta^{-1}\e\|\sh H(T)\|_{H^{\f{27}{4},0}_\Psi}^2\\
&+\eta\e^\f14\int_0^T\|\sh\t^{\f12}\pa_y^2 u_\Phi\|_{H^{5,0}_\Psi}^2\,dt+\eta\e^\f14\int_0^T\|\f{y}{\t}\sh H\|_{H^{\f{27}{4},0}_\Psi}^2\,dt\\
&+C\bigl(1+\la\e^\f12+\e^\f14\eta^{-1}\bigr)\int_0^T \dot\th(t)\Bigl(\|\sh H\|_{H^{7,0}_\Psi}^2+\| \sh \t^\f12\pa_yu_\Phi\|_{H^{\f{21}{4},0}_\Psi}^2\Bigr)\,dt.
\end{split}\eeq}
\end{prop}

\begin{proof}  We start the proof of Proposition \ref{pro: tu} by the following lemma:

\begin{lem}\label{lem tu_1}
{\sl Let $\hbar(t)$ be a non-negative and non-decreasing function on $[0,T]$ with $T\leq T^\star.$ Then  if
$\b\leq\ga_0+\f14,$ for any  sufficiently small $\zeta>0$, one has
\beq \label{S4eq3}
\begin{split}
\|\sh& u_\Phi(T)\|_{H^{\f{11}{2},0}_\Psi}^2+2(\la-C\bigl(1+\zeta^{-1}\e)\e^{\f12}\bigr)\int_0^T\dot\th(t)\|\sh u_\Phi(t)\|_{H^{\f{23}{4},0}_\Psi}^2\,dt\\
&-\int_0^T\|\sqrt{\hbar'} u_\Phi(t)\|_{H^{\f{11}{2},0}_\Psi}^2\,dt+\bigl(1-\f\zeta 4\bigr)\int_0^T \|\sh\pa_yu_\Phi(t)\|_{H^{\f{11}{2},0}_\Psi}^2\,dt\\
\leq & \|\sh(0) u_\Phi(0)\|_{H^{\f{11}{2},0}_\Psi}^2+C\e\int_{0}^T\w{t}^{-\left(\ga_0+\f14\right)}\|\sh u_\Phi\|_{H^{6,0}_\Psi}^2\,dt.
\end{split}\eeq}
\end{lem}

Let us postpone the proof of Lemma \ref{lem tu_1} till we finish the proof of Proposition \ref{pro: tu}.

With \eqref{S4eq3}, to prove \eqref{S4eq1}, it remains to handle the estimate of $\e\int_{0}^T\w{t}^{-\left(\ga_0+\f14\right)}\|\sh u_\Phi\|_{H^{6,0}_\Psi}^2\,dt.$
As a matter of fact,  in view of \eqref{tu_Phi},
 we write
 \beq\label{S4eq3d}
 \begin{split}
\e\int_0^T&\w{t}^{-\f14-\gamma_0}\|\sh u_\Phi\|_{H^{6,0}_\Psi}^2\,dt=\e\int_0^T\w{t}^{-\f14-\gamma_0}\hbar
\langle u_\Phi, u_\Phi\rangle_{H^{6,0}_\Psi}dt\\
 =&\int_0^T\f{\e\t^{-\f14-\gamma_0}\hbar}{\dot\th}\Bigl(\langle\mathcal{L}H,u_\Phi\rangle_{H^{6,0}_\Psi}
 -\bigl\langle T_{\pa_yu}^\h\pa_x\int_y^\infty H\,dz,u_\Phi\bigr\rangle_{H^{6,0}_\Psi}\\
 &+\langle T^\h_{\pa_y v}H,u_\Phi\rangle_{H^{6,0}_\Psi}
 -\f{\d(t)}2\bigl\langle T^\h_{\pa_yD_xu}\Lambda(D_x)\pa_x\int_y^\infty H\,dz,u_\Phi\bigr\rangle_{H^{6,0}_\Psi}\Bigr)dt\\
 \eqdefa&B_1+\cdots+B_4.
 \end{split}\eeq

 The estimate of $B_1$ relies on the following lemma, the proof of which will be postponed in Appendix \ref{appb}.

 \begin{lem}\label{S4lem2}
{\sl Let $\beta\leq\min\bigl\{\f{1}{2}\bigl(\gamma_0+\f54\bigr),\f23 \gamma_0+\f12, \gamma_0+\f14,\f43\gamma_0\bigr\}$, then one has
\beq\label{S4eq3e}
\begin{split}
|B_1|\leq&  \zeta\|\sh u_\Phi(T)\|_{H^{\f{11}{2},0}_\Psi}^2+C\zeta^{-1}\e\|\sh H(T)\|_{H^{\f{27}{4},0}_\Psi}^2+\zeta\e\int_0^T\t^{-\gamma_0-\f14}\|u_\Phi\|_{H^{6,0}_\Psi}^2\,dt\\
&+\zeta\e^\f 54\int_0^T\|\sh\pa_y u_\Phi\|_{H^{\f{11}2,0}_\Psi}^2 \,dt+\zeta\e^\f 14\int_0^T\|\sh\t^{\f12}\pa_y^2 u_\Phi\|_{H^{5,0}_\Psi}^2\,dt\\
&+\zeta\e^\f 14\int_0^T\|\f{y}{\t}\sh H\|_{H^{\f{27}{4},0}_\Psi}^2\,dt
+C\bigl(1+\la\e^\f12+\e^\f14\zeta^{-1}\bigr)\times\\
&\qquad \times\int_0^T \dot\th(t)\Bigl(\|\sh u_\Phi\|_{H^{\f{23}{4},0}_\Psi}^2+\|\sh H\|_{H^{7,0}_\Psi}^2+\| \sh \t^\f12\pa_yu_\Phi\|_{H^{\f{21}{4},0}_\Psi}^2\Bigr)\,dt.
\end{split}\eeq
}
\end{lem}

 We now turn to the estimates of the remaining
terms in \eqref{S4eq3d}. \smallskip

\no$\bullet$\underline{The estimate of $B_2.$}  We  observe that for any $s\in\R,$
\beq \label{est: int_y^infty H}
\begin{split}
\|e^{\Psi}\int_y^\infty Hdz\|_{L^\infty_\v(H^s_\h)}\leq& \|\int_y^\infty e^{-\f{(y-z)^2}{8\t}}e^{\f{z^2}{8\t}}\|H(\cdot,z)\|_{H^s_\h} \,dz\|_{L^\infty_\v}\\
\leq&\|e^{-\f{y^2}{8\t}}\|_{L^2_\v}\|H\|_{H^{s,0}_{\Psi}}\leq C\t^\f14\|H\|_{H^{s,0}_{\Psi}}.
\end{split} \eeq
Then we  get, by  applying first Lemma \ref{lem: T_fg} and  \eqref{est: int_y^infty H} and then
\eqref{assume: 2}, that
\begin{align*}
|B_2|\leq& C\e^\f12\int_0^T\t^{\beta-\gamma_0}\|\pa_y u\|_{L^2_\v(H^{\f12+}_\h)}\|\sh \pa_x H\|_{H^{6,0}_\Psi}\|\sh u_\Phi\|_{H^{6,0}_\Psi}\,dt\\
\leq& C\e^\f32\int_0^T\t^{\beta-2\gamma_0-\f12}\| \sh H\|_{H^{7,0}_\Psi}\|\sh u_\Phi\|_{H^{6,0}_\Psi}\,dt.
\end{align*}
For any $\zeta>0,$ by applying Young's inequality, we obtain
\begin{align*}
|B_2|
\leq&C\e^\f32\zeta^{-1}\int_0^T\dot\th(t)\|\sh H\|_{H^{7,0}_\Psi}^2\,dt+ \f\zeta3\e\int_0^T\t^{-\gamma_0-\f14}\|\sh u_\Phi\|_{H^{6,0}_\Psi}^2\,dt,
\end{align*}
if $\beta\leq\gamma_0+\f14.$

\no$\bullet$\underline{The estimate of $B_3.$} Similar to the estimate of $B_2,$ thanks to \eqref{S3eq8},  we deduce
\begin{align*}
|B_3|\leq&C\e^\f12\int_0^T\t^{\beta-\gamma_0-\f14}\|\pa_y v\|_{L^\infty_\v(H^{\f12+}_\h)}\| \sh H\|_{H^{6,0}_\Psi}\|\sh u_\Phi\|_{H^{6,0}_\Psi}\,dt\\
\leq&C\e^\f32\int_0^T\t^{\beta-2\gamma_0-\f12}\|\sh H\|_{H^{7,0}_\Psi}\|\sh u_\Phi\|_{H^{6,0}_\Psi}\,dt\\
\leq&C\e^\f32\zeta^{-1}\int_0^T\dot\th(t)\|\sh H\|_{H^{7,0}_\Psi}^2\,dt+ \f\zeta3\e\int_0^T\t^{-\gamma_0-\f14}\|\sh u_\Phi\|_{H^{6,0}_\Psi}^2\,dt,
\end{align*}
if $\beta\leq\gamma_0+\f14.$

\no$\bullet$\underline{The estimate of $B_4.$} Similar to the estimate of  $B_2,$ we have
\begin{align*}
|B_4|
\leq&C\e^\f32\zeta^{-1}\int_0^T\dot\th(t)\|\sh H\|_{H^{7,0}_\Psi}^2\,dt+ \f\zeta3\e\int_0^T\t^{-\gamma_0-\f14}\|\sh u_\Phi\|_{H^{6,0}_\Psi}^2\,dt,
\end{align*}
if$\beta\leq\gamma_0+\f14.$

By summarizing the estimates of $B_2, B_3, B_4$, we achieve
\begin{align}\label{S4eq5}
|B_2|+|B_3|+|B_4|\leq&C\e^\f32\zeta^{-1}\int_0^T\dot\th(t)\|\sh H\|_{H^{7,0}_\Psi}^2dt+ \zeta\e\int_0^T\t^{-\gamma_0-\f14}\|\sh u_\Phi\|_{H^{6,0}_\Psi}^2dt,
\end{align}
if $\beta\leq\gamma_0+\f14.$

By inserting the estimates \eqref{S4eq3e} and \eqref{S4eq5} into \eqref{S4eq3d} and taking $\zeta$ to be sufficiently small, we arrive at
\beq\label{S4eq3c}
\begin{split}
\e\int_0^T&\t^{-\gamma_0-\f14}\|\sh u_\Phi\|_{H^{6,0}_\Psi}^2\,dt
\leq  2\zeta\|\sh u_\Phi(T)\|_{H^{\f{11}{2},0}_\Psi}^2+C\zeta^{-1}\e\|\sh H(T)\|_{H^{\f{27}{4},0}_\Psi}^2\\
&\qquad+2\zeta\e^\f 54\int_0^T\|\sh\pa_y u_\Phi\|_{H^{\f{11}2,0}_\Psi}^2 \,dt+2\zeta\e^\f 14\int_0^T\|\sh\t^{\f12}\pa_y^2 u_\Phi\|_{H^{5,0}_\Psi}^2\,dt\\
&\qquad+2\zeta\e^\f 14\int_0^T\|\f{y}{\t}\sh H\|_{H^{\f{27}{4},0}_\Psi}^2\,dt+C\bigl(1+\la\e^\f12+\e^\f14\zeta^{-1}\bigr)\times\\
&\qquad\times\int_0^T \dot\th(t)\Bigl(\|\sh u_\Phi\|_{H^{\f{23}{4},0}_\Psi}^2+\|\sh H\|_{H^{7,0}_\Psi}^2+\| \sh \t^\f12\pa_yu_\Phi\|_{H^{\f{21}{4},0}_\Psi}^2\Bigr)\,dt.
\end{split}\eeq

By substituting \eqref{S4eq3c} into \eqref{S4eq3} and taking $\zeta$ to be sufficiently small, we achieve \eqref{S4eq1}
with $\eta=\f\zeta2.$ This completes the proof of Proposition \ref{pro: tu}.
\end{proof}

Proposition \ref{pro: tu} has been proved provided with  the proof of Lemma \ref{lem tu_1}, which we present now.

\begin{proof}[Proof of Lemma \ref{lem tu_1}] Thanks to \eqref{S3eq7f},  we get, by
taking $H^{\f{11}{2},0}_\Psi$ inner product of \eqref{eq: tu_Phi} with $\hbar u_\Phi$ and
using  integrating by parts, that
\beq \label{S4eq3a}
\begin{split}
\f12\f{d}{dt}&\|\sh u_\Phi(t)\|_{H^{\f{11}{2},0}_\Psi}^2-\f12\|\sqrt{\hbar'} u_\Phi(t)\|_{H^{\f{11}{2},0}_\Psi}^2+\la\dot\th(t)\|\sh u_\Phi\|_{
H^{\f{23}{4},0}_\Psi}^2+\f12\|\sh\pa_y u_\Phi\|_{H^{\f{11}{2},0}_\Psi}^2\\
\leq&\hbar\Bigl(|\langle T^\h_{u}\pa_xu_\Phi, u_\Phi\rangle_{H^{\f{11}{2},0}_\Psi}|+
|\langle T^\h_{v}\pa_yu_\Phi, u_\Phi\rangle_{H^{\f{11}{2},0}_\Psi}|
+\f{\d(t)}2|\bigl\langle T^\h_{D_xu}\Lambda(D_x)\pa_x u_\Phi, u_\Phi\bigr\rangle_{H^{\f{11}{2},0}_\Psi}|\\
&\qquad+|\langle T^\h_{\pa_y u}v_\Phi, u_\Phi\rangle_{H^{\f{11}{2},0}_\Psi}|
+\f{\d(t)}2|\langle T^\h_{\pa_yD_xu}\Lambda(D_x)v_\Phi, u_\Phi\rangle_{H^{\f{11}{2},0}_\Psi}|
+|\langle f, u_\Phi\rangle_{H^{\f{11}{2},0}_\Psi}|\Bigr)\\
\eqdefa &A_1+\cdots+A_6.
\end{split}\eeq

We first deduce from  a similar argument of \eqref{S3eq3a} that
\begin{align*}
A_1\leq\f\hbar2\Bigl(&|\bigl\langle\D^{\f{11}{2}} u_\Phi,  [T^\h_{u}; \D^{\f{11}2}]\pa_xu_\Phi\bigr\rangle_{L^2_\Psi}|+
|\bigl\langle \D^{\f{11}{2}} u_\Phi, ((T_u^\h)^*-T^\h_u)\D^{\f{11}{2}}\pa_x u_\Phi\bigr\rangle_{L^2_\Psi}|\\
&+|\langle [\D^{\f{11}{2}},T_{u}]\pa_x u_\Phi, \D^{\f{11}{2}}u_\Phi\rangle_{L^2_\Psi}|\Bigr),
\end{align*}
from which,  Lemma \ref{lem: com1} and \eqref{S3eq8}, we infer
\begin{align*}
A_1\leq& C\bigl(\|u\|_{L_\v^\infty(H^{\f32+}_\h)}+\|u\|_{L_\v^\infty(H^{\f32+}_\h)}^2\bigr)\|\sh u_\Phi\|_{H^{\f{11}{2},0}_\Psi}^2\\
\leq& C\e\w{t}^{-\left(\ga_0+\f14\right)}\|\sh u_\Phi\|_{H^{\f{11}{2},0}_\Psi}^2
\leq C\e^\f12\dot\th(t)\|\sh u_\Phi\|_{H^{\f{11}{2},0}_\Psi}^2,
\end{align*}
if $\beta\leq \gamma_0+\f14.$

While for any $\zeta>0,$ it follows from  Lemma \ref{lem: T_fg} and \eqref{S3eq8} that if  $\beta\leq \gamma_0+\f14,$
\begin{align*}
A_2\leq&C\|v\|_{L^\infty_\v(H^{\f12+}_\h)}\|\sh \pa_y u_\Phi\|_{H^{\f{11}{2},0}_\Psi}\|\sh u_\Phi\|_{H^{\f{11}{2},0}_\Psi}\\
\leq&C\zeta^{-1}\e^2\w{t}^{-2\ga_0+\f12}\|\sh u_\Phi\|_{H^{\f{11}{2},0}_\Psi}^2+\f\zeta{32}\|\sh\pa_y u_\Phi\|_{H^{\f{11}{2},0}_\Psi}^2,\\
\leq&C\zeta^{-1}\e^{\f32}\dot\th\|\sh u_\Phi\|_{H^{\f{11}{2},0}_\Psi}^2+\f\zeta{32}\|\sh\pa_y u_\Phi\|_{H^{\f{11}{2},0}_\Psi}^2
\end{align*}
and
\begin{align*}
A_3\leq&C\|
u\|_{L^\infty_\v(H^{\f32+}_\h)}\|\sh u_\Phi\|_{H^{\f{23}4,0}_\Psi}^2
\leq  C\e^\f12\t^{\beta-\gamma_0-\f14}\dot\th(t)\|\sh u_\Phi\|_{H^{\f{23}4,0}_\Psi}^2.
\end{align*}

On the other hand, due to $v=\int_y^\infty \pa_xu\,dz,$ we observe from \eqref{est: int_y^infty H} that
\beq\label{S4eq3b}
\|v_\Phi\|_{L^\infty_{\v,\ga\Psi}(H^{s}_\h)}\leq C\t^{\f14}\|u_\Phi\|_{H^{s+1,0}_{\ga\Psi}}\quad\forall \ s\in\R \andf \ga\in [0,1].
\eeq
So that we deduce from  Lemma \ref{lem: T_fg}  and \eqref{assume: 2} that
\begin{align*}
A_4\leq&C\|\pa_y u\|_{L^2_{\v}(H^{\f12+}_\h)}\|\sh v_\Phi\|_{L^\infty_{\v,\Psi}(H^{5}_x)}\|\sh u_\Phi\|_{H^{6,0}_\Psi}\\
\leq & C\e\w{t}^{-\left(\ga_0+\f14\right)}\|\sh u_\Phi\|_{H^{6,0}_\Psi}^2,
\end{align*}
and
\begin{align*}
A_5\leq&C\|D_x\pa_y u\|_{L^2_{\v}(H^{\f12+}_\h)}\|\sh v_\Phi\|_{L^\infty_{\v,\Psi}(H^{\f{19}{4}}_\h)}\|\sh u_\Phi\|_{H^{\f{23}{4},0}_\Psi}\\
\leq & C\e^\f12\t^{\beta-\gamma_0-\f14}\dot\th(t)\|\sh u_\Phi\|_{H^{\f{23}4,0}_\Psi}^2.
\end{align*}

Finally we deduce from Lemma \ref{lem: f} that
\begin{align}\label{est: f}
\nonumber
\|f\|_{H^{\f{11}{2},0}_\Psi}
\leq&C\t^\f14\bigl(\|\pa_y G_\Phi\|_{H^{\f52+,0}_\Psi}\|u_\Phi\|_{H^{\f{11}{2},0}_\Psi}+\|G_\Phi\|_{H^{\f52+,0}_\Psi}\|\pa_y u_\Phi\|_{H^{\f{11}{2},0}_\Psi}\bigr),\\
\leq&C\e\Bigl(\t^{-\gamma_0-\f14}\|u_\Phi\|_{H^{\f{11}{2},0}_\Psi}+\t^{-\gamma_0+\f14}\|\pa_y u_\Phi\|_{H^{\f{11}{2},0}_\Psi}\Bigr),
\end{align}
which implies that for any $\zeta>0,$
\begin{align*}
A_6\leq &\hbar\|f\|_{H^{\f{11}{2},0}_\Psi}\|u_\Phi\|_{H^{\f{11}{2},0}_\Psi}\\
\leq & C\e^\f12\bigl(1+\e\zeta^{-1}\bigr))\w{t}^{\beta-\gamma_0-\f14}\dot\th(t)\|\sh u_\Phi\|_{H^{\f{11}{2},0}_\Psi}^2+\f\zeta{32}\|\sh\pa_y u_\Phi\|_{H^{\f{11}{2},0}_\Psi}^2.
\end{align*}

By substituting the above estimates into \eqref{S4eq3a} and integrating the resulting inequality over $[0,T]$ for any
$T<T^\star,$ we achieve
\begin{align*}
\f12&\|\sh u_\Phi(T)\|_{H^{\f{11}{2},0}_\Psi}^2-\f12\int_0^T\bigl\|\sqrt{\hbar'} u_\Phi(t)\|_{H^{\f{11}{2},0}_\Psi}^2
+\la\int_0^T\dot\th(t)\|\sh u_\Phi(t)\|_{H^{\f{23}{4},0}_\Psi}^2\,dt\\
&+\bigl(\f12-\f\zeta 8\bigr)\int_0^T \|\sh\pa_y u_\Phi(t')\|_{H^{\f{11}{2},0}_\Psi}^2\,dt'\\
\leq&\|\sh(0) u_\Phi(0)\|_{H^{\f{11}{2},0}_\Psi}^2+C\e\int_{0}^T\w{t}^{-\left(\ga_0+\f14\right)}\|\sh u_\Phi\|_{H^{6,0}_\Psi}^2\,dt\\
&+C\bigl(\zeta^{-1}\e+1\bigr)\e^{\f12}\int_0^T\dot\th(t)\|\sh u_\Phi(t)\|_{H^{\f{23}{4},0}_\Psi}^2 \, dt,
\end{align*}
if  $\beta\leq \gamma_0+\f14.$  This leads to \eqref{S4eq3}. We thus complete the proof of Lemma \ref{lem tu_1}.
\end{proof}

\subsection{The estimate of $\pa_yu_\Phi$}\label{Sect6}

The purpose of this section is to derive the Gevery estimate of $\pa_yu.$  We first
get, by taking $\pa_y$ to the $u$ equation of \eqref{eq: tu}, that
\begin{align}\label{eq: pa_y tu}
\pa_t\pa_yu+u\pa_x\pa_yu+v\pa_y^2u-\pa_y^3u=0.
\end{align}
Moreover, by virtue of the $u$ equation of \eqref{eq: tu} and the  boundary conditions
that $u|_{y=0}=v|_{y=0}=,0$ we obtain
 \beq \label{S5eq1} \pa_y^2u|_{y=0}=0. \eeq

 While by applying  Bony's decomposition to $u\pa_x\pa_yu$ and $v\pa_y^2u$ in \eqref{eq: pa_y tu} in the horizontal
 variable, we write
\begin{align*}
u\pa_x\pa_yu=&T^\h_u\pa_x\pa_yu+T^\h_{\pa_x\pa_yu}u+R^\h(u,\pa_x\pa_yu),\\
v\pa_y^2u=&T^\h_{\pa_y^2u}v+T^\h_{v}\pa_y^2u+R^\h(v,\pa_y^2u).
\end{align*}
Then inserting the above equalities  into \eqref{eq: pa_y tu} and applying the operator  $e^{\Phi(t,D_x)}$ to resulting equality  gives rise to
\beq\label{eq: om_Phi}
\begin{split}
\pa_t\pa_yu_\Phi+T^\h_u\pa_x\pa_yu_\Phi+&T^\h_{\pa_y^2u}v_\Phi-\pa_y^3u_\Phi=\frak{f},\with\\
\frak{f}=&-\Bigl(T^\h_{\pa_x\pa_yu}u+R^\h(u,\pa_x\pa_yu)+T^\h_{v}\pa_y^2u+R^\h(v,\pa_y^2u)\Bigr)_\Phi\\
&-\bigl((T^\h_u\pa_x\pa_yu)_\Phi-T^\h_u\pa_x\pa_yu_\Phi\bigr)-\bigl((T^\h_{\pa_y^2u}v)_\Phi-T^\h_{\pa_y^2u}v_\Phi\bigr).
\end{split}
\eeq

The main result states as follows:

\begin{prop}\label{pro: pa_y tu}
{\sl Let $u$ be a smooth enough solution of \eqref{eq: tu} which decays to zero sufficiently fast as $y$
  approaching to $\infty.$ Let $\hbar(t)$ be a non-negative and non-decreasing function on $[0,T]$ with $T\leq T^\star.$ Then   if  $\beta\leq \gamma_0+\f14,$
one has
\beq \label{S5eq2}
\begin{split}
\|\sh&\pa_yu_\Phi(T)\|_{H^{5,0}_\Psi}^2+2(\la-C\e^\f12\bigl(1+\eta^{-1}\e)\bigr)\int_0^T\dot\th(t)\|\sh\pa_yu_\Phi(t)\|_{H^{\f{21}{4},0}_\Psi}^2\,dt\\
&-\int_0^T\|\sqrt{\hbar'}\pa_yu_\Phi\|_{H^{5,0}_\Psi}^2\,dt+(1-\f{\eta}{16})\int_0^T\|\sh\pa_y^2u_\Phi(t)\|_{H^{5,0}_\Psi}^2dt\\
\leq& \|\sh(0)e^{\d\D^\f12}\pa_yu_0\|_{H^{5,0}_\Psi}^2+C\e^\f12(1+\eta^{-1}\e\bigr)\int_0^T\dot\th(t)\|\sh\t^{-\f12}u_\Phi(t)\|_{H^{\f{23}{4},0}_\Psi}^2\,dt.
\end{split} \eeq
}
\end{prop}

\begin{proof} Again thanks to \eqref{S3eq7f},
by taking  $H^{5,0}_\Psi$  inner product of \eqref{eq: om_Phi} with $\pa_yu_\Phi,$ we find
\beq \label{est: pa_y tu}
\begin{split}
\f12\f{d}{dt}&\|\sh\pa_yu_\Phi\|_{H^{5,0}_\Psi}^2-\f12\|\sqrt{\hbar'}\pa_yu_\Phi\|_{H^{5,0}_\Psi}^2+\la\dot\th(t)\|\sh\pa_yu_\Phi\|_{H^{\f{21}{4},0}_\Psi}^2+\f12\|\sh\pa_y^2u_\Phi\|_{H^{5,0}_\Psi}^2\\
\leq&|\langle \sh T^\h_u\pa_x\pa_yu_\Phi,\sh\pa_yu_\Phi\rangle_{H^{5,0}_\Psi}|+|\langle T^\h_{\pa_y^2u} v_\Phi,\sh\pa_yu_\Phi\rangle_{H^{5,0}_\Psi}|+|\langle \frak{f},\sh\pa_yu_\Phi\rangle_{H^{5,0}_\Psi}|\\
&\eqdefa D_1+D_2+D_3.
\end{split}\eeq

Let us now handle term by term in \eqref{est: pa_y tu}. We first get, by a similar derivation of \eqref{S3eq3aq}, that
\begin{align*}
D_1\leq&C\bigl(\|u\|_{L_\v^\infty(H^{\f32+}_\h)}+\|u\|_{L_\v^\infty(H^{\f32+}_\h)}^2\bigr)\|\sh\pa_yu_\Phi\|_{H^{5,0}_\Psi}^2,\end{align*}
which together with \eqref{S3eq-1}, \eqref{S3eq8} and $\beta\leq \gamma_0+\f14$ implies that
\beq \label{S5eq3}
\begin{split}
D_1
\leq&C\bigl(\e^\f12\t^{\beta-\gamma_0-\f14}+\e^\f32\t^{\beta-2\gamma_0-\f12}\bigr)\dot\th(t)\|\sh\pa_yu_\Phi\|_{H^{5,0}_\Psi}^2\\
\leq&C\e^\f12\dot\th(t)\|\sh\pa_yu_\Phi\|_{H^{5,0}_\Psi}^2.
\end{split} \eeq

 Whereas it follows from  Lemma \ref{lem: T_fg}, \eqref{assume: 2} and \eqref{S4eq3b} that
 \beq \label{S5eq4}
 \begin{split}
 D_2\leq&C\|\pa_y^2 u\|_{H^{\f12+,0}}\|\sh v_\Phi\|_{L^\infty_{\v,\Psi}(H^{\f{19}{4}}_x)}\|\sh\pa_yu_\Phi\|_{H^{\f{21}{4},0}_\Psi}\\
 \leq& C\e^\f12\t^{\beta-\gamma_0-\f14}\dot\th(t)\Bigl(\|\sh \t^{-\f12}u_\Phi\|_{H^{\f{23}{4},0}_\Psi}^2+\|\sh\pa_yu_\Phi\|_{H^{\f{21}{4},0}_\Psi}^2\Bigr)\\
 \leq&C\e^\f12\dot\th(t)\Bigl(\|\sh \t^{-\f12}u_\Phi\|_{H^{\f{23}{4},0}_\Psi}^2+\|\sh\pa_yu_\Phi\|_{H^{\f{21}{4},0}_\Psi}^2\Bigr),
 \end{split} \eeq
if $\beta\leq \gamma_0+\f14.$

To estimate $D_3,$ we first get, by applying Lemma \ref{lem: T_fg} and \eqref{S3eq8b}, that
\begin{align*}
\|\sh(T^\h_{\pa_x\pa_yu}u)_\Phi\|_{H^{\f{19}{4},0}_\Psi}\leq& C\t^\f12\|\pa_x\pa_yu_\Phi\|_{L^\infty_\v(H^{\f12+}_\h)}\|\sh\t^{-\f12}u_\Phi\|_{H^{\f{19}{4},0}_\Psi}\\
\leq& C\e^\f12\t^{\beta-\gamma_0-\f14}\dot\th(t)\|\sh\t^{-\f12}u_\Phi\|_{H^{\f{19}{4},0}_\Psi},\end{align*}
and
\begin{align*}
\|\sh(T^\h_{v}\pa_y^2u)_\Phi\|_{H^{\f{19}{4},0}_\Psi}\leq& C\|v_\Phi\|_{L^\infty_\v(H^{\f12+}_\h)}\|\sh\pa_y^2u_\Phi\|_{H^{\f{19}{4},0}_\Psi}\\
\leq& C\e^\f34\t^{\f12\beta-\gamma_0+\f14}\dot\th^\f12(t)\|\sh\pa_y^2u_\Phi\|_{H^{\f{19}{4},0}_\Psi}.
\end{align*}
Along the same line, we have
\begin{align*}
\|\sh(R^\h(u,\pa_x\pa_yu))_\Phi\|_{H^{\f{19}{4},0}_\Psi}\leq& C\e^\f12\t^{\beta-\gamma_0-\f14}\dot\th(t)\|\sh\t^{-\f12}u_\Phi\|_{H^{\f{19}{4},0}_\Psi},\\
\|\sh(R^\h(v,\pa_y^2u))_\Phi\|_{H^{\f{19}{4},0}_\Psi}\leq& C\e^\f34\t^{\f12\beta-\gamma_0+\f14}\dot\th(t)^\f12\|\sh\pa_y^2u_\Phi\|_{H^{\f{19}{4},0}_\Psi}.
\end{align*}

While we  deduce from Lemma \ref{lem: com1} and \eqref{S3eq8} that
\begin{align*}
\|\sh(T^\h_u\pa_x\pa_yu)_\Phi-\sh T^\h_u\pa_x\pa_yu_\Phi)\|_{H^{\f{19}{4},0}_\Psi}\leq &C\|u_\Phi\|_{L^\infty_\v(H^{\f32+}_\h)}\|\sh\pa_yu_\Phi\|_{H^{\f{21}{4},0}_\Psi}\\
\leq& C\e^\f12\t^{\beta-\gamma_0-\f14}\dot\th(t)\|\sh\pa_yu_\Phi\|_{H^{\f{21}{4},0}_\Psi}.\end{align*}
Whereas it follows from  Lemma \ref{lem: com1}, \eqref{assume: 2} and \eqref{S4eq3b} that
\begin{align*}
\|\sh((T^\h_{\pa_y^2u}v)_\Phi-T^\h_{\pa_y^2u}v_\Phi)\|_{H^{\f{19}{4},0}_\Psi}\leq&C\|\pa_y^2 u_\Phi\|_{H^{\f32+,0}}\|v_\Phi\|_{L^\infty_{\v,\Psi}(H^{\f{17}{4}}_\h)}\\
\leq& C\e^\f12\t^{\beta-\gamma_0-\f14}\dot\th(t)\|\sh\t^{-\f12}u_\Phi\|_{H^{\f{21}{4}}_\Psi}.
\end{align*}

Therefore, if $\beta\leq \gamma_0+\f14,$ in view of the definition of $\ff$ given by \eqref{eq: om_Phi}, we achieve
\begin{align*}
\|\sh \frak{f}\|_{H^{\f{19}{4},0}_\Psi}\leq&C\e^\f12\dot\th(t)\Bigl(\|\sh\t^{-\f12}u_\Phi\|_{H^{\f{23}{4},0}_\Psi}
+\|\sh\pa_yu_\Phi\|_{H^{\f{21}{4},0}_\Psi}\Bigr)+\e^\f34\dot\th^\f12(t)\|\sh\pa_y^2u_\Phi\|_{H^{5,0}_\Psi},
\end{align*}
which implies that for any $\eta>0$
\beq \label{S5eq5}
\begin{split}
D_3\leq &\|\sh \frak{f}\|_{H^{\f{19}{4},0}_\Psi}\|\sh\pa_yu_\Phi\|_{H^{\f{21}{4},0}_\Psi}\\
\leq&C\e^\f12(1+\eta^{-1}\e\bigr)\dot\th(t)\Bigl(\|\sh\t^{-\f12}u_\Phi\|_{H^{\f{23}{4},0}_\Psi}^2
+\|\sh\pa_yu_\Phi\|_{H^{\f{21}{4},0}_\Psi}^2\Bigr)+\f{\eta}{32}\|\sh\pa_y^2u_\Phi\|_{H^{5,0}_\Psi}^2.
\end{split}
\eeq

By inserting the estimates \eqref{S5eq3}, \eqref{S5eq4} and \eqref{S5eq5} into \eqref{est: pa_y tu}, we obtain
\eqref{S5eq2}. This completes the proof of Proposition \ref{pro: pa_y tu}.
\end{proof}

\subsection{The proof of Proposition \ref{S1prop2}}
Now we are in a position to complete the proof of Proposition \ref{S1prop2}.

\begin{proof}[Proof of Proposition \ref{S1prop2}]
Taking  $\hbar(t)=\w{t}^{\f{1-\eta}2}$ in \eqref{S4eq1}, \eqref{S4eq2} can be proved along the same line
 to that  of Corollary \ref{S3col2}.

While
taking $\hbar=\t^{\f{3-\eta}{2}}$ in \eqref{S5eq2} leads to
\beq \label{S5eq7}
\begin{split}
&\|\t^{\f{3-\eta}{4}}\pa_yu_\Phi(T)\|_{H^{5,0}_\Psi}^2+2\bigl(\la-C\e^\f12(1+\e\eta^{-1})\bigr)
\int_0^T\dot\th(t)\|\t^{\f{3-\eta}{4}}\pa_yu_\Phi(t)\|_{H^{\f{21}{4},0}_\Psi}^2\,dt\\
&\qquad+\bigl(1-\f{\eta}{16}\bigr)\int_0^T\|\t^{\f{3-\eta}{4}}\pa_y^2u_\Phi(t)\|_{H^{5,0}_\Psi}^2\,dt\\
\leq& \|e^{\d\D^\f12}\pa_yu_0\|_{H^{5,0}_\Psi}^2+C\e^\f12\bigl(1+\e\eta^{-1}\bigr)\int_0^T\dot\th(t)\|\t^{\f{1-\eta}{4}}u_\Phi(t)\|_{H^{\f{23}{4},0}_\Psi}^2\,dt\\
&\qquad+\f{3-\eta}2\int_0^T\|\t^{\f{1-\eta}{4}}\pa_yu_\Phi\|_{H^{5,0}_\Psi}^2\,dt.
\end{split} \eeq
Yet it follows from   \eqref{S4eq2} that
\begin{align*}
\la&\int_0^T\dot\th(t)\|\t^{\f{1-\eta}{4}}u_\Phi(t)\|_{H^{\f{23}{4},0}_\Psi}^2\,dt
+\int_0^T\|\t^{\f{1-\eta}{4}}\pa_yu_\Phi\|_{H^{5,0}_\Psi}^2\,dt\\
\leq&C\eta^{-1}\| e^{\d\D^\f12} u_0\|_{H^{\f{11}{2},0}_\Psi}^2+C\eta^{-2}\e\|\w{T}^{\f{1-\eta}4} H(T)\|_{H^{\f{27}{4},0}_\Psi}^2\\
&+\e^\f14\int_0^T\|\w{t}^{\f{3-\eta}4}\pa_y^2 u_\Phi\|_{H^{5,0}_\Psi}^2\,dt+\e^\f14\int_0^T\|\w{t}^{-\f{3+\eta}4}{y} H\|_{H^{\f{27}{4},0}_\Psi}^2\,dt\\
&+C\eta^{-1}\bigl(1+\la\e^\f12+\e^\f14\eta^{-1}\bigr)\int_0^T \dot\th(t)\Bigl(\|\w{t}^{\f{1-\eta}4} H\|_{H^{7,0}_\Psi}^2+\| \w{t}^{\f{3-\eta}4}\pa_yu_\Phi\|_{H^{\f{21}{4},0}_\Psi}^2\Bigr)\,dt.
\end{align*}
So that by taking
\beno
 \la\geq \la_2\eqdefa C\bigl(\eta^{-1}(1+\e^\f14\eta^{-1})+ 2\e^{\f12}(1+\e\eta^{-1}\bigr), \andf \e\leq \e_2\quad\mbox{so that}\ \ 
  C\eta^{-1}\e^\f12\leq \f14,\eeno
we conclude the proof of \eqref{S5eq6} by inserting the above estimate into \eqref{S5eq7}. This
finished the proof of the proof of Proposition \ref{S1prop2}.
\end{proof}

\setcounter{equation}{0}
\section{The Gevery estimates of $G$ and $\pa_yG$} \label{Sect7}

\subsection{The Gevery estimates of $G$} \label{Sect7.1}
The purpose of this Subsection is to present the decay-in-time estimate for the good quantity, $G,$
given by \eqref{S1eq3a}. As in the estimate of $u$ and $\pa_yu,$ we  get, by first applying Bony's decomposition
for the horizontal variable to $u\p_xG, v\pa_y G, v\pa_y(y\varphi)$ and $\pa_y u\pa_x \varphi,$   and then
 applying the operator, $e^{\Phi(t,D_x)},$ to the equation \eqref{S1eq1}, that  
 \beq \label{eq: G_Phi}
\begin{split}
\pa_t & G_\Phi+\la\dot\th(t)\D^\f12 G_\Phi-\pa_y^2 G_\Phi+\w{t}^{-1}G_\Phi+T^\h_{u}\pa_x G_\Phi\\
&+T^\h_{\pa_y G}v_\Phi-\f1{2\w{t}}T^\h_{\pa_y(y\varphi)}v_\Phi+\f{y}{\w{t}}\int_y^\infty T^\h_{\pa_y u}v_\Phi\,dy'=F,
\end{split} \eeq
where
\beq \label{eq: F}
\begin{split}
-F\eqdefa&(T^\h_{u}\pa_x G)_\Phi-T^\h_{u}\pa_x G_\Phi+(T^\h_{\pa_xG}u)_\Phi+(R^\h(u,\pa_x G))_\Phi\\
&+(T^\h_{\pa_y G}v)_\Phi-T^\h_{\pa_y G}v_\Phi+(T^\h_{v}\pa_y G)_\Phi+(R^\h(v,\pa_y G))_\Phi,\\
&-\f1{2\w{t}}\bigl((T^\h_{\pa_y(y\varphi)}v)_\Phi-T^\h_{\pa_y(y\varphi)}v_\Phi\bigr)
-\f1{2\w{t}}\bigl(T^\h_v\pa_y(y\varphi)+R^\h(\pa_y(y\varphi),v)\bigr)_\Phi\\
&+\f{y}{\w{t}}\int_y^\infty\bigl((T^\h_{\pa_y u}v)_\Phi-T^\h_{\pa_y u}v_\Phi\bigr)\,dy'
+\f{y}{\w{t}}\int_y^\infty\bigl(T^\h_v\pa_yu+R^\h(v,\pa_y u)\bigr)_\Phi\,dy'.
\end{split}
\eeq

Let us first present  the estimate for $F.$

\begin{lem}\label{lem: F}
{\sl It holds that
\begin{align}\label{est: F}
\|F\|_{H^{\f{15}{4},0}_\Psi}\leq&C\e\Bigl(\t^{-\gamma_0-\f14}\| G_\Phi\|_{H^{\f{17}{4},0}_{\Psi}}+\t^{-\gamma_0+\f14}\|\pa_y G_\Phi\|_{H^{\f{15}4,0}_\Psi}\Bigr).
\end{align}}
\end{lem}

\begin{proof}
Applying  Lemma \ref{lem: com2} and \eqref{S3eq8} gives
\begin{align*}
\|(T^\h_{u}\pa_x G)_\Phi-T^\h_{u}\pa_x G_\Phi\|_{H^{\f{15}{4},0}_\Psi}\leq & C\|u_\Phi\|_{L^\infty_\v(H^{\f32+}_\h)}
\|G_\Phi\|_{H^{\f{17}{4},0}_\Psi}
\leq C\e\t^{-\gamma_0-\f14}\|G_\Phi\|_{H^{\f{17}{4},0}_\Psi}.
\end{align*}
While it follows from \eqref{S4eq3b} and Lemma \ref{lem2.2} that
\beq\label{S6eq1}
\|v_\Phi\|_{L^\infty_{\v,\f{\Psi}2}(H^{s}_\h)}\leq C\t^{\f14}\|u_\Phi\|_{H^{s+1,0}_{\f{\Psi}2}} \leq C\t^{\f14}\|G_\Phi\|_{H^{s+1,0}_{\Psi}}\quad\forall \ s\in\R,
\eeq
which together with Lemma \ref{lem: com2}  and \eqref{assume: 1} ensures that
\begin{align*}
\|((T^\h_{\pa_y G}v)_\Phi-T^\h_{\pa_y G}v_\Phi)\|_{H^{\f{15}{4},0}_\Psi}\leq& C\|\pa_y G_\Phi\|_{H^{\f32+,0}_{\f{\Psi}2}}\| v_\Phi\|_{L^\infty_{\v,\f{\Psi}2}(H^{\f{13}{4}}_\h)}
\leq C\e\t^{-\gamma_0-\f14}\|G_\Phi\|_{H^{\f{17}{4},0}_\Psi}.
\end{align*}
Notice from \eqref{S7eq20d} and \eqref{assume: 1} that
\beq \label{S6eq2}
\w{t}^{-1}\|\pa_y (y\varphi)_\Phi\|_{H^{\f52+,0}_{\f{\Psi}2}}\leq C\|\pa_y G_\Phi\|_{H^{\f52+,0}_{\Psi}}\leq C\e\t^{-\ga_0-\f12},
\eeq
so that applying Lemma \ref{lem: com2} and \eqref{S6eq1} yields
\begin{align*}
\w{t}^{-1}\|\bigl((T^\h_{\pa_y(y\varphi)}v)_\Phi-T^\h_{\pa_y(y\varphi)}v_\Phi\bigr)\|_{H^{\f{15}{4},0}_\Psi}
\leq& C\w{t}^{-1}\|\pa_y (y\varphi)_\Phi\|_{H^{\f32+,0}_{\f{\Psi}2}}\|v_\Phi\|_{L^\infty_{\v,\f{\Psi}2}(H^{\f{13}{4}}_\h)}\\
\leq& C\e\t^{-\gamma_0-\f14} \| G_\Phi\|_{H^{\f{17}{4},0}_\Psi}.
\end{align*}
Whereas applying Lemma \ref{lem: com2}, \eqref{assume: 2} and \eqref{S6eq1} gives
\begin{align*}
\bigl\|\int_y^\infty \bigl((T^\h_{\pa_y u}v)_\Phi-T^\h_{\pa_y u}&v_\Phi\bigr)\,dy'\bigr\|_{L^\infty_\v(H^{\f{15}{4},0}_\h)}\leq
\bigl\|\int_y^\infty \|\pa_y u_\Phi\|_{H^{\f32+}_\h}\|v_\Phi\|_{H^{\f{13}{4}}_\h}\,dy'\bigr\|_{L^\infty_\v}\\
\leq&C\|\pa_y u_\Phi\|_{L^2_{y,\f34\Psi}(H^{\f12+}_\h)}\|e^{\f34\Psi}v_\Phi\|_{L^\infty_\v(H^{\f{13}{4}}_\h)}\Bigl(\int_y^\infty e^{-3\Psi}dy'\Bigr)^\f12\\
\leq&C\e\t^{-\ga_0} e^{-\f 32\Psi}\|G_\Phi\|_{H^{\f{17}{4},0}_{\Psi}},
\end{align*}
from which, we infer
\begin{align*}
\bigl\|{y}{\w{t}}^{-1}\int_y^\infty\bigl((T^\h_{\pa_y u}v)_\Phi-T^\h_{\pa_y u}v_\Phi\bigr)\,dy'\bigr\|_{H^{\f{15}{4},0}_\Psi}
\leq C\e\t^{-\gamma_0-\f14}\|G_\Phi\|_{H^{\f{17}{4},0}_{\Psi}}.
\end{align*}

On the other hand, by applying
Lemma \ref{lem: T_fg},  \eqref{assume: 1} and \eqref{S3eq8}, we find
\begin{align*}
\|(T^\h_{\pa_xG}u)_\Phi+(R^\h(u,\pa_x G))_\Phi\|_{H^{\f{15}{4},0}_\Psi}\leq& C\|\pa_xG_\Phi\|_{L^\infty_{\v,\f12\Psi}(H^{\f12+}_\h)}\| u_\Phi\|_{H^{\f{15}{4},0}_{\f12\Psi}}\\
\leq&C\e\t^{-\gamma_0-\f14} \| G_\Phi\|_{H^{\f{15}{4},0}_\Psi},\\
%%%%%%%%%%%%%%%%%%%%
\|((T^\h_{v}\pa_y G)_\Phi+(R^\h(v,\pa_y G))_\Phi)\|_{H^{\f{15}{4},0}_\Psi}\leq& C\|v_\Phi\|_{L^\infty_\v(H^{\f12+}_\h)}\|\pa_y G_\Phi\|_{H^{\f{15}4,0}_\Psi}\\
\leq&C\e\t^{-\gamma_0+\f14}\| \pa_y G_\Phi\|_{H^{\f{15}4,0}_\Psi},\\
%%%%%%%%%%%%%%%%
%%%%%%%%%%%%%
\bigl\|\w{t}^{-1}\bigl(T^\h_v\pa_y(y\varphi)+R^\h(\pa_y(y\varphi),v)\bigr)_\Phi\bigr\|_{H^{\f{15}{4},0}_\Psi}\leq& C\|v_\Phi\|_{L^\infty_{\v,\f12\Psi}(H^{\f12+}_x)}\|\w{t}^{-1}\pa_y (y\varphi)_\Phi\|_{H^{\f{15}4,0}_{\f12\Psi}}\\
\leq&C\e\t^{-\gamma_0+\f14}\|\pa_y G_\Phi\|_{H^{\f{15}4,0}_\Psi}.
\end{align*}
and
\begin{align*}
\bigl\|\int_y^\infty \bigl(T^\h_v\pa_yu+R^\h(v,\pa_y u)\bigr)_\Phi&\,dy'\bigr\|_{L^\infty_\v(H^{\f{15}{4},0}_\h)}
\leq \|\int_y^\infty \|\pa_y u_\Phi\|_{H^{\f{15}4}_\h}\|v_\Phi\|_{H^{\f12+}_\h}\,dy'\|_{L^\infty_\v}\\
\leq&C\|\pa_y u_\Phi\|_{L^2_{\v,\f34\Psi}(H^{\f{15}4}_\h)}\|e^{\f34\Psi}v_\Phi\|_{L^\infty_\v(H^{\f12+}_\h)}\Bigl(\int_y^\infty e^{-3\Psi}\,dy'\Bigr)^\f12\\
\leq& C\e\t^{-\ga_0+\f12} e^{-\f 32\Psi}\|\pa_y G_\Phi\|_{H^{\f{15}4,0}_\Psi},
\end{align*}
from which, we infer
\begin{align*}
\bigl\|{y}{\w{t}}^{-1}\int_y^\infty\bigl(T^\h_v\pa_yu+R^\h(v,\pa_y u)_\Phi\bigr)\,dy'\bigr\|_{H^{\f{15}{4},0}_\Psi}
\leq C\e\t^{-\gamma_0+\f14}\| \pa_y G_\Phi\|_{H^{\f{15}4,0}_\Psi}.
\end{align*}

In view of \eqref{eq: F}, we conclude the proof of \eqref{est: F} by summarizing the above estimates.
\end{proof}

\begin{prop}\label{pro: G}
{\sl Let $\hbar(t)$ be a non-negative and non-decreasing function on $[0,T]$ with $T\leq T^\star.$ Let $G$ be determined by
\eqref{S1eq3a}. Then   if
$\b\leq\ga_0-\f1{12},$ for any $\eta>0,$ we have
\beq\label{est: G}
\begin{split}
\f12&\|\sqrt{\hbar}G_\Phi(T)\|_{H^{4,0}_\Psi}^2
-\f12\int_0^T\|\sqrt{\hbar'}G_\Phi(t)\|_{H^{4,0}_\Psi}^2\,dt+\int_0^T\w{t}^{-1}\|\sh G_\Phi(t)\|_{H^{4,0}_\Psi}^2\,dt\\
&+\bigl(\la-C(1+\eta^{-1}\e)\e^{\f12}\bigr)\int_0^T\dot\th(t)
\|\sh G_\Phi(t)\|_{H^{\f{17}{4},0}_\Psi}^2\,dt\\
&+\bigl(\f12-\f\eta 8\bigr)\int_0^T\|\sh\pa_yG_\Phi(t)\|_{H^{4,0}_\Psi}^2\,dt\\
\leq& \f12\|\sqrt{\hbar}(0)G_\Phi(0)\|_{H^{4,0}_\Psi}^2+C\e^\f12\int_0^T\dot\th(t)(\|\sh \t^{-1}u_\Phi\|_{H^{\f{23}{4},0}_\Psi}^2\,dt.
\end{split}
\eeq}
\end{prop}

\begin{proof} Thanks to \eqref{S3eq7f},
we first get, by taking $H^{4,0}_\Psi$ inner product of \eqref{eq: G_Phi} with $\hbar G_\Phi,$ that
\beq\label{S6eq3}
\begin{split}
\f12&\f{d}{dt}\|\sh G_\Phi(t)\|_{H^{4,0}_\Psi}^2-\f12\|\sqrt{\hbar'} G_\Phi(t)\|_{H^{4,0}_\Psi}^2+\w{t}^{-1}\|\sh G_\Phi(t)\|_{H^{4,0}_\Psi}^2\\
&+\la\dot\th(t)\|\sh G_\Phi\|_{
H^{\f{17}{4},0}_\Psi}^2+\f12\|\sh\pa_y G_\Phi\|_{H^{4,0}_\Psi}^2\\
\leq& \hbar|\langle T^\h_{u}\pa_x G_\Phi, G_\Phi \rangle_{H^{4,0}_\Psi}|+\hbar|\langle T^\h_{\pa_y G}v_\Phi, G_\Phi \rangle_{H^{4,0}_\Psi}|+\f\hbar{2\w{t}}|\langle T^\h_{\pa_y(y\varphi)}v_\Phi, G_\Phi \rangle_{H^{4,0}_\Psi}|\\
&+\hbar\big|\langle \f{y}{\w{t}}\int_y^\infty T_{\pa_y u}v_\Phi\,dy', G_\Phi \rangle_{H^{4,0}_\Psi}\bigr|+\hbar|\langle F, G_\Phi \rangle_{H^{4,0}_\Psi}|
\eqdefa  I_1+\cdots+I_5.
\end{split}\eeq

Next, let us handle the  estimates of $I_i,~i=1,\cdots,5,$ term by term.\smallskip

\no\underline{$\bullet$ The estimate $I_1.$} It follows from a similar derivation of \eqref{S3eq3aq} that
\begin{align*}
I_1\leq&C\bigl(\|u\|_{L_\v^\infty(H^{\f32+}_\h)}+\|u\|_{L_\v^\infty(H^{\f32+}_\h)}^2\bigr)\|\sh G_\Phi\|_{H^{4,0}_\Psi}^2\\
\leq& C\e^\f12\dot\th(t)\|\sh G_\Phi\|_{H^{4,0}_\Psi}^2,
\end{align*}
if $\beta\leq\gamma_0+\f14.$

\no\underline{$\bullet$ The estimate $I_2.$}
By applying Lemma \ref{lem: T_fg}  and \eqref{S6eq1}, we find
\begin{align*}
I_2\leq&C\sh\|\pa_y G\|_{L^2_{\v,\f12\Psi}(H^{\f12+}_\h)}\|v_\Phi\|_{L^\infty_{\v,\f12 \Psi}(H^{\f{15}{4},0}_x)}\|\sh G\|_{H^{\f{17}{4},0}_\Psi}\\
\leq&C\e\t^{-\gamma_0-\f12+\f14}\|\sh u_\Phi\|_{H^{\f{19}{4},0}_{\f12\Psi}}\|\sh G\|_{H^{\f{17}{4},0}_\Psi}.
\end{align*}
Yet it follows from  interpolation inequality in the anisotropic Sobolev norm
and  Lemma \ref{lem2.2} that
\begin{align}\label{est: interpolation-1}
\|u\|_{H^{\f{19}{4},0}_{\f34\Psi}}\leq& C\|u\|_{H^{\f{17}{4},0}_{\f34\Psi}}^\f23\|u\|_{H^{\f{23}{4},0}_{\f34\Psi}}^{\f13}
\leq C\t^\f13\| G\|_{H^{\f{17}{4},0}_\Psi}^\f23\|\t^{-1}u\|_{H^{\f{23}{4},0}_\Psi}^\f13.
\end{align}
As a result, it comes out
\begin{align*}
I_2
\leq&C\e\t^{-\gamma_0+\f{1}{12}}\|\sh \t^{-1}u_\Phi\|_{H^{\f{23}{4},0}_\Psi}^\f13\|\sh G\|_{H^{\f{17}{4},0}_\Psi}^{\f 53}\\
\leq&C\e^\f12\t^{\beta-\gamma_0+\f{1}{12}}\dot\th(t)\Bigl(\|\sh \t^{-1}u_\Phi\|_{H^{\f{23}{4},0}_\Psi}^2+\|\sh G\|_{H^{\f{17}{4},0}_\Psi}^2\Bigr).
\end{align*}

\no\underline{$\bullet$ The estimate $I_3.$} Applying Lemmas \ref{lem: T_fg} and \ref{lem2.2} gives
\begin{align*}
I_3\leq&C\|\w{t}^{-1}\pa_y (y\varphi)\|_{L^2_{\v,\f12\Psi}(H^{\f12+}_x)}\|\sh v_\Phi\|_{L^\infty_{\v,\f12 \Psi}(H^{\f{15}{4},0}_x)}\|\sh G\|_{H^{\f{17}{4},0}_\Psi}\\
\leq&C\|\pa_y G\|_{L^2_{\v,\Psi}(H^{\f12+}_\h)}\|\sh v_\Phi\|_{L^\infty_{\v,\f12 \Psi}(H^{\f{15}{4},0}_\h)}\|\sh G\|_{H^{\f{17}{4},0}_\Psi},
\end{align*}
from which and the estimate of $I_2,$ we obtain
\begin{align*}
I_3
\leq C\e^\f12\t^{\beta-\gamma_0+\f{1}{12}}\dot\th(t)\Bigl(\|\sh \t^{-1}u_\Phi\|_{H^{\f{23}{4},0}_\Psi}^2+\|\sh G\|_{H^{\f{17}{4},0}_\Psi}^2\Bigr).
\end{align*}

\no\underline{$\bullet$ Estimate $I_4.$} It follows from Lemmas \ref{lem: T_fg} and \eqref{assume: 2} that for any $s\in\R$
\beq \label{est: int_0^y pa_y u v}
\begin{split}
\bigl\|\int_y^\infty T^\h_{\pa_y u}v_\Phi\,dy'\bigr\|_{L^\infty_\v(H^{s}_\h)}\leq& \bigl\|\int_y^\infty \|\pa_y u\|_{H^{\f12+}_\h}\|v_\Phi\|_{H^{s}_\h}\,dy'\bigr\|_{L^\infty_\v}\\
\leq&C\|\pa_y u\|_{L^2_{\v,\f34\Psi}(H^{\f12+}_\h)}\|e^{\f34\Psi}v_\Phi\|_{L^\infty_\v(H^{s}_\h)}
\Bigl(\int_y^\infty e^{-3\Psi}\,dy'\Bigr)^\f12\\
\leq& C\e\t^{-\ga_0} e^{-\f 32\Psi}\|u_\Phi\|_{H^{1+s,0}_{\f34\Psi}},
\end{split}\eeq from which
and \eqref{est: interpolation-1}, we infer
\begin{align*}
I_4\leq&\hbar \bigl\|\f{y}{\w{t}}\int_y^\infty T_{\pa_y u}v_\Phi\,dy'\|_{H^{\f{15}{4},0}_\Psi}\|G_\Phi\bigr\|_{H^{\f{17}{4},0}_\Psi}\\
\leq&C\e\t^{-\ga_0-\f12} \bigl\|\f{y}{\w{t}^\f12}e^{-\f12\Psi}\bigr\|_{L^2_\v}\|\sh u_\Phi\|_{H^{\f{19}{4},0}_{\f34\Psi}}\|\sh G_\Phi\|_{H^{\f{17}{4},0}_\Psi}\\
\leq&C\e^\f12\t^{\beta-\gamma_0+\f{1}{12}}\dot\th(t)\Bigl(\|\sh \t^{-1}u_\Phi\|_{H^{\f{23}{4},0}_\Psi}^2+\|\sh G\|_{H^{\f{17}{4},0}_\Psi}^2\Bigr).
\end{align*}

\no\underline{$\bullet$ The stimate $I_5.$} By applying Lemma \ref{lem: F}, we find
\begin{align*}
I_5\leq& \|\sh F\|_{H^{\f{15}{4},0}_\Psi}\|\sh G_\Phi\|_{H^{\f{17}{4},0}_{\Psi}}\\
\leq&C\e\Bigl(\t^{-\gamma_0-\f14}\|\sh G_\Phi\|_{H^{\f{17}{4},0}_{\Psi}}+\t^{-\gamma_0+\f14}\|\sh \pa_y G_\Phi\|_{H^{4,0}_\Psi}\Bigr)\|\sh G_\Phi\|_{H^{\f{17}{4},0}_{\Psi}}.
\end{align*}
Then for any $\eta>0,$ we get, by applying Young's inequality, that
\begin{align*}
I_5
\leq&C\bigl(\e^\f12\t^{\beta-\gamma_0-\f14}+\eta^{-1}\e^{\f32}\t^{\beta-2\gamma_0+\f12}\bigr)\dot\th(t)\|\sh G_\Phi\|_{H^{\f{17}{4},0}_{\Psi}}^2+\f\eta8\|\sh \pa_y G_\Phi\|_{H^{4,0}_\Psi}^2.
\end{align*}

By inserting the above estimates into \eqref{S6eq3} and then integrating the resulting inequality  over $[0,T]$, we conclude the proof of \eqref{est: G}
in case that $\b\leq \ga_0-\f1{12}.$
\end{proof}

\subsection{The Gevery estimate of $\pa_yG$}\label{Sect7.2}

Due to the loss of derivative in the $G$ equation \eqref{S1eq1}, we shall derive the estimate of $\|\pa_yG_\Phi(T)\|_{H^{3,0}_\Psi}$
instead of $\|\pa_yG_\Phi(T)\|_{H^{4,0}_\Psi}$ as the estimate of $G_\Phi$ in the previous section. In order
to do so, we first get, by applying
 $\pa_y $ to \eqref{eq: G_Phi}, that
 \beq \label{eq: pa_yG_Phi}
\begin{split}
\pa_t& \pa_yG_\Phi+\la\dot\th(t)\D^\f12 \pa_yG_\Phi-\pa_y^3G_\Phi+\w{t}^{-1}\pa_yG_\Phi\\
&+\pa_y\Bigl(T^\h_{u}\pa_x G_\Phi+T^\h_{\pa_y G}v_\Phi-\f1{2\w{t}}T^\h_{\pa_y(y\varphi)}v_\Phi+\f{y}{\w{t}}\int_y^\infty T^\h_{\pa_y u}v_\Phi\,dy'\Big)=\pa_yF,
\end{split}\eeq
with $F$ being given by \eqref{eq: F}.

Moreover, in view of the $G$ equation of \eqref{S1eq1} and the boundary conditions that $u|_{y=0}=v|_{y=0}=G|_{y=0}=0,$
we deduce that
\beq \label{S7eq1}
\pa_y^2G|_{y=0}=0.
\eeq

\begin{prop}\label{pro: pa_yG}
{\sl  Let $\hbar(t)$ be a non-negative and non-decreasing function on $[0,T]$ with $T\leq T^\star.$ Let $G$ be determined
by \eqref{S1eq3a}. Then  if
$\b\leq2\ga_0-\f1{2},$ for any $\eta>0,$ we have
\beq \label{S7eq2}
\begin{split}
\f12&\|\sqrt{\hbar}\pa_yG_\Phi(T)\|_{H^{3,0}_\Psi}^2
+\int_0^T\w{t}^{-1}\|\sh \pa_yG_\Phi(t)\|_{H^{3,0}_\Psi}^2\,dt-\f12\int_0^T\|\sqrt{\hbar'}\pa_yG_\Phi(t)\|_{H^{3,0}_\Psi}^2\,dt\\
&+\bigl(\la-C\eta^{-1}\e^\f32\bigr)\int_0^T\dot\th(t)
\|\sh \pa_yG_\Phi(t)\|_{H^{\f{13}{4},0}_\Psi}^2\,dt+\bigl(\f12-\f\eta 8\bigr)\int_0^T\|\sh\pa_y^2G_\Phi(t)\|_{H^{3,0}_\Psi}^2\,dt\\
\leq&\f12\|\sqrt{\hbar}(0)\pa_yG_\Phi(0)\|_{H^{3,0}_\Psi}^2+C\eta^{-1}\e^\f32\int_0^T\dot\th(t)\|\sh\t^{-\f12}G_\Phi\|_{H^{4,0}_\Psi}^2\,dt .
\end{split}
\eeq}
\end{prop}

\begin{proof}
Notice that
\begin{align*}
&\pa_y^2 G_\Phi|_{y=0},\quad F|_{y=0}=0,\\
\Bigl(T^\h_{u}\pa_x G_\Phi+T^\h_{\pa_y G}&v_\Phi-\f12\w{t}^{-1}T_{\pa_y(y\varphi)}v_\Phi+\f{y}{\w{t}}\int_y^\infty T_{\pa_y u}v_\Phi\,dy'\Bigr)\bigr|_{y=0}=0,
\end{align*}
we get, by taking $H^{3,0}_\Psi$ inner product of \eqref{eq: pa_yG_Phi} with $\hbar \pa_yG_\Phi$ and using integration by parts, that
\beq \label{S7eq3}
\begin{split}
\f12&\f{d}{dt}\|\sh \pa_yG_\Phi(t)\|_{H^{3,0}_\Psi}^2-\f12\|\sqrt{\hbar'} \pa_yG_\Phi(t)\|_{H^{3,0}_\Psi}^2+\w{t}^{-1}\|\sh \pa_yG_\Phi(t)\|_{H^{3,0}_\Psi}^2\\
&+\la\dot\th(t)\|\sh \pa_yG_\Phi\|_{
H^{\f{13}{4},0}_\Psi}^2+\f12\|\sh\pa_y^2 G_\Phi\|_{H^{3,0}_\Psi}^2\\
\leq& \hbar|\langle T^\h_{u}\pa_x G_\Phi, \pa_y^2G_\Phi \rangle_{H^{3,0}_\Psi}|+\hbar|\langle T^\h_{\pa_y G}v_\Phi, \pa_y^2G_\Phi \rangle_{H^{3,0}_\Psi}|+\f\hbar{2\w{t}}|\langle T^\h_{\pa_y(y\varphi)}v_\Phi, \pa_y^2G_\Phi \rangle_{H^{3,0}_\Psi}|\\
&+\hbar|\langle \f{y}{\w{t}}\int_y^\infty T^\h_{\pa_y u}v_\Phi\,dy', \pa_y^2G_\Phi \rangle_{H^{3,0}_\Psi}|+\hbar|\langle F, \pa_y^2G_\Phi \rangle_{H^{3,0}_\Psi}|
\eqdefa  II_1+\cdots+II_5.
\end{split}
\eeq

Next let us deal with the  estimates of  $II_i,~i=1,\cdots,5,$ term by term.\smallskip

\no\underline{$\bullet$ The estimate of $II_1.$} It follows from Lemma \ref{lem: T_fg} and  \eqref{S3eq8} that
\begin{align*}
\|T^\h_{u}\pa_x G_\Phi\|_{H^{3,0}_\Psi}\leq& C\|u\|_{L^\infty_\v(H^{\f12+}_\h)}\|G_\Phi\|_{H^{4,0}_\Psi}
\leq C\e\t^{-\gamma_0-\f14}\|G_\Phi\|_{H^{4,0}_\Psi},
\end{align*}
so that for any $\eta>0,$ one has
\begin{align*}
II_1\leq& \|\sh T^\h_{u}\pa_x G_\Phi\|_{H^{3,0}_\Psi}\|\sh \pa_y^2 G_\Phi\|_{H^{3,0}_\Psi}\\
\leq& C\e\t^{-\gamma_0+\f14}\|\sh\t^{-\f12}G_\Phi\|_{H^{4,0}_\Psi}\|\sh \pa_y^2 G_\Phi\|_{H^{3,0}_\Psi}\\
\leq& \f{\eta}{40}\|\sh \pa_y^2 G_\Phi\|_{H^{3,0}_\Psi}^2+C\eta^{-1}\e^\f32\t^{\beta-2\gamma_0+\f12}\dot\th(t)\|\sh\t^{-\f12}G_\Phi\|_{H^{4,0}_\Psi}^2.
\end{align*}

\no\underline{$\bullet$ The estimate of $II_2.$} Applying Lemma \ref{lem: T_fg}, \eqref{assume: 1} and  \eqref{S4eq3b} gives
\begin{align*}
\|T^\h_{\pa_y G}v_\Phi\|_{H^{3,0}_\Psi}\leq& C\|\pa_y G\|_{L^2_{\v,\f12\Psi}(H^{\f12+}_\h)}
\|v_\Phi\|_{L^\infty_{\v,\f12\Psi}(H^{3}_\h)}\\
\leq& C\e\t^{-\gamma_0-\f12}\t^\f14\|u_\Phi\|_{H^{4,0}_{\f12\Psi}}
\leq C\e\t^{-\gamma_0-\f14}\|G_\Phi\|_{H^{4,0}_\Psi}.
\end{align*}
Then we get, by a similar estimate of $II_1,$ that
\begin{align*}
II_2\leq& \|\sh T_{\pa_y G}v_\Phi\|_{H^{3,0}_\Psi}\|\sh \pa_y^2 G_\Phi\|_{H^{3,0}_\Psi}\\
\leq& \f{\eta}{40}\|\sh \pa_y^2 G_\Phi\|_{H^{3,0}_\Psi}^2+C\eta^{-1}\e^\f32\t^{\beta-2\gamma_0+\f12}\dot\th(t)\|\sh\t^{-\f12}G_\Phi\|_{H^{4,0}_\Psi}^2.
\end{align*}

\no\underline{$\bullet$  The estimate of $II_3.$}
Thanks to \eqref{S6eq2}, we deduce by a similar estimate of $II_2$ that
\begin{align*}
II_3
\leq& \f{\eta}{40}\|\sh \pa_y^2 G_\Phi\|_{H^{3,0}_\Psi}^2+C\eta^{-1}\e^\f32\t^{\beta-2\gamma_0+\f12}\dot\th(t)\|\sh\t^{-\f12}G_\Phi\|_{H^{4,0}_\Psi}^2.
\end{align*}

\no\underline{$\bullet$ The estimate of $II_4.$} Applying \eqref{est: int_0^y pa_y u v} with $s=3$ gives
\begin{align*}
II_4\leq&\hbar \bigl\|\f{y}{\w{t}}\int_y^\infty T_{\pa_y u}v_\Phi\,dy'\|_{H^{3,0}_\Psi}\|\pa_y^2G_\Phi\bigr\|_{H^{3,0}_\Psi}\\
\leq&C\e\t^{-\ga_0}\bigl\|\f{y}{\w{t}^{\f12}}e^{-\f12\Psi}\bigr\|_{L^2_\v}\|\sh\t^{-\f12} G_\Phi\|_{H^{4,0}_{\Psi}}\|\sh \pa_y^2G_\Phi\|_{H^{3,0}_\Psi}.
\end{align*}
Applying Young's inequality yields
\begin{align*}
II_4
\leq&\f{\eta}{40}\|\sh \pa_y^2 G_\Phi\|_{H^{3,0}_\Psi}^2+C\eta^{-1}\e^\f32\t^{\beta-2\gamma_0+\f12}\dot\th(t)\|\sh\t^{-\f12}G_\Phi\|_{H^{4,0}_\Psi}^2.
\end{align*}

\no\underline{$\bullet$ The estimate of $II_5.$} We get, by applying  Lemma \ref{lem: F}, that
\begin{align*}
II_5\leq& \|F\|_{H^{3,0}_\Psi}\|\sh \pa_y^2 G_\Phi\|_{H^{3,0}_\Psi}\\
\leq& C\e\Bigl(\t^{-\gamma_0+\f14}\|\sh\t^{-\f12} G_\Phi\|_{H^{4,0}_{\Psi}}+\t^{-\gamma_0+\f14}\|\sh \pa_y G_\Phi\|_{H^{3,0}_\Psi}\Bigr)\|\sh \pa_y^2 G_\Phi\|_{H^{3,0}_\Psi}.
\end{align*}
Applying Young's inequality gives rise to
\begin{align*}
II_5
\leq& \f{\eta}{40}\|\sh \pa_y^2 G_\Phi\|_{H^{3,0}_\Psi}^2+C\eta^{-1}\e^\f32\t^{\beta-2\gamma_0+\f12}\dot\th(t)\Bigl(\|\sh\t^{-\f12}G_\Phi\|_{H^{4,0}_\Psi}^2
+\|\sh\pa_yG_\Phi\|_{H^{3,0}_\Psi}^2\Bigr).
\end{align*}

By substituting the above estimates into \eqref{S7eq3} and  integrating the resulting inequality over $[0,T]$  leads to \eqref{S7eq2}
in case that $\b\leq2\ga_0-\f1{2}$.
This completes the proof of Proposition \ref{pro: pa_yG}.
\end{proof}

\subsection{The proof of Proposition \ref{S1prop3}}

\begin{proof}[Proof of Proposition \ref{S1prop3}] We first deduce from  Lemma \ref{lem2.1} that
\begin{align*}
\f{5-\eta}4\|\t^{\f{1-\eta}{4}}G_\Phi(t)\|_{H^{4,0}_\Psi}^2\leq \w{t}^{-1}\|\t^{\f{5-\eta}{4}} G_\Phi(t)\|_{H^{4,0}_\Psi}^2
+\bigl(\f12-\f\eta 2\bigr)\|\t^{\f{5-\eta}{4}}\pa_yG_\Phi(t)\|_{H^{4,0}_\Psi}^2.
\end{align*}
Then taking $\hbar=\t^{\f{5-\eta}{2}}$ in \eqref{est: G} leads to
\beq \label{S6eq4a}
\begin{split}
\|\t^{\f{5-\eta}{4}}&G_\Phi(T)\|_{H^{3,0}_\Psi}^2+
2\bigl(\la-C\e^\f12(1+\eta^{-1}\e)\bigr)\int_0^T\dot\th(t)
\|\t^{\f{5-\eta}{4}} G_\Phi(t)\|_{H^{\f{13}{4},0}_\Psi}^2\,dt\\
&+\f{3}{4}\eta\int_0^T\|\t^{\f{5-\eta}{4}}\pa_yG_\Phi(t)\|_{H^{3,0}_\Psi}^2\,dt\\
\leq &\|G_\Phi(0)\|_{H^{4,0}_\Psi}^2+C\e^\f12\int_0^T\dot\th(t)\|\t^{\f{1-\eta}{4}}u_\Phi\|_{H^{\f{23}{4},0}_\Psi}^2\,dt.
\end{split}\eeq
By taking $\la\geq 2C\e^\f12(1+\eta^{-1}\e)$ in \eqref{S6eq4a}, we achieve
 \eqref{S6eq4}.

While by
taking $\hbar(t)=\t^\f{7-\eta}{2}$ in \eqref{S7eq1}, we find
\beq\label{S7eq5}
\begin{split}
\|\langle &T\rangle^\f{7-\eta}{4}\pa_yG_\Phi(T)\|_{H^{3,0}_\Psi}^2
+2\bigl(\la-C\eta^{-1}\e^{\f32}\bigr)\int_0^T\dot\th(t)
\|\t^\f{7-\eta}{4} \pa_yG_\Phi(t)\|_{H^{\f{13}{4},0}_\Psi}^2\,dt\\
&\quad+\bigl(1-\f\eta4\bigr)\int_0^T\|\t^\f{7-\eta}{4}\pa_y^2G_\Phi(t)\|_{H^{3,0}_\Psi}^2\,dt
\leq\|\pa_yG_\Phi(0)\|_{H^{3,0}_\Psi}^2\\
&+C\eta^{-1}\e^\f32\int_0^T\dot\th(t)\|\t^\f{5-\eta}{4}G_\Phi\|_{H^{4,0}_\Psi}^2\,dt +\f{7-\eta}2\int_0^T\|\t^\f{5-\eta}{4}\pa_yG_\Phi(t)\|_{H^{3,0}_\Psi}^2\,dt.
\end{split}\eeq
Yet if $\la\geq 2C\e^{\f12}\bigl(1+\e\eta^{-1}\bigr),$ we deduce from  \eqref{S6eq4} that
\begin{align*}
&C\eta^{-1}\e^\f32\int_0^T\dot\th(t)\|\t^\f{5-\eta}{4}G_\Phi\|_{H^{4,0}_\Psi}^2\,dt +C\int_0^T\|\t^\f{5-\eta}{4}\pa_yG_\Phi(t')\|_{H^{3,0}_\Psi}^2\,dt\\
&\qquad\qquad\qquad\leq C\eta^{-1}\|G_\Phi(0)\|_{H^{4,0}_\Psi}^2+C\eta^{-1}\e^\f12\int_0^T\dot\th(t)\|\t^{\f{1-\eta}{4}}u_\Phi\|_{H^{\f{23}{4},0}_\Psi}^2\,dt,
\end{align*}
Substituting the above estimate into \eqref{S7eq5} leads to \eqref{S7eq4}. This completes the proof of Proposition
\ref{S1prop3}.
\end{proof}

\setcounter{equation}{0}
\section{The Sobolev norm estimates of $\pa_y^2 G$ and $\pa_y^3 G$ }\label{Sect8}

\subsection{The Sobolev norm estimates of $\pa_y^2 G$}
In this subsection, we shall deal with the Sobolev norm estimate of $\pa_y^2G.$ We first get,
by
taking $\pa_y^2$ on \eqref{S1eq1}, that
\beq \label{S8eq1}
\begin{split}
\pa_t\pa_y^2 G&-\pa_y^4 G+\w{t}^{-1}\pa_y^2G\\
&+\pa_y^2\Bigl(u\p_xG+v\p_yG-\f12\w{t}^{-1}v\p_y({y\vf})+\f{y}{\w{t}}\int_y^\infty\left(\p_y u\p_x\vf\right)\,dy'\Bigr)=0.
\end{split} \eeq
It is easy to observe that
\beq \label{def: F_1}
\begin{split}
\pa_y&\Bigl(u\p_xG+v\p_yG-\f1{2\w{t}}v\p_y({y\vf})+\f{y}{\w{t}}\int_y^\infty\left(\p_y u\p_x\vf\right)\,dy'\Bigr)\\
= &u\p_x\pa_yG+v\p_y^2G+\pa_y u\pa_x G+\pa_y v\bigl(\pa_y G-\f1{2\w{t}}\pa_y(y\varphi)\bigr)\\
&-\f1{2\w{t}}v\p_y^2({y\vf})-\f{y}{\w{t}}\left(\p_y u\p_x\vf\right)+\f{1}{\w{t}}\int_y^\infty\left(\p_y u\p_x\vf\right)\,dy'
\eqdefa \frak{F}.
\end{split}
\eeq

We first present the  estimate of $\frak{F}.$

 \begin{lem}\label{lem: F_1}
{\sl Let  $\frak{F}$ be given by  \eqref{def: F_1}. Then for $t\leq T^\star,$ one has
\beq \label{S8eq2}
\|\frak{F}\|_{H^{3,0}_\Psi}\leq C\e\t^{-\gamma_0+\f14}\bigl(\t^{-\f12}\|\pa_yG_\Phi\|_{H^{3,0}_\Psi}+\t^{-1}\|G_\Phi\|_{H^{4,0}_\Psi}\bigr).
\eeq
}
\end{lem}

\begin{proof} We shall frequently use the following classical law of product in Sobolev space:
\begin{align}\label{product: 1}
\|a~b\|_{H^s_\h}\leq C\|a\|_{H^{s}_\h}\|b\|_{H^s_\h}\quad \forall \ s>\f12.
\end{align}
Moreover, due to $\Phi(t,\xi)\geq\f\d2$ for any $k\in\N$ and $t\leq T^\ast,$ we have
\beq \label{S8eq2a}
\|a\|_{H^{k+s}_\h}^2=\int_{\R}\w{\xi}^{2(k+s)}|\widehat{a}(\xi)|^2\,d\xi\leq C\int_{\R}e^{2\Phi(t,\xi)}\w{\xi}^{2s}|\widehat{a}(\xi)|^2\,d\xi
=C\|a_\Phi\|_{H^s_\h}^2.
\eeq

It follows from a similar derivation of \eqref{S3eq8} that
\beq \label{S8eq3}
\|u(t)\|_{L_\v^\infty(H^{\sigma}_\h)}\leq
 C\t^\f14\|\pa_y G(t)\|_{H^{\sigma,0}_{\Psi}}\quad\forall \sigma\in\R,
\eeq
from which, \eqref{product: 1}, \eqref{S8eq2a} and \eqref{assume: 1},  we infer
\begin{align*}
\|u\p_x\pa_yG\|_{H^{3,0}_\Psi}\leq& C\|u\|_{L^\infty_{\v}(H^{3}_\h)}\|\pa_x\pa_y G\|_{H^{3,0}_\Psi}\\
\leq& C(\t^\f14\|\pa_yG\|_{H^{3,0}_{\Psi}}\|\pa_yG_\Phi\|_{H^{3,0}_\Psi}
\leq C\e\t^{-\gamma_0-\f14}\|\pa_yG_\Phi\|_{H^{3,0}_\Psi}.
\end{align*}

Similarly we deduce from \eqref{S6eq1} and \eqref{assume: 1} that
\begin{align*}
\|v\p_y^2G\|_{H^{3,0}_\Psi}\leq& C\|v\|_{L^\infty_{\v}(H^{3}_\h)}\|\pa_y^2 G\|_{H^{3,0}_\Psi}\\
\leq& C\t^\f14\|G\|_{H^{4,0}_{\Psi}}\|\pa_y^2 G\|_{H^{3,0}_\Psi}
\leq C\e\t^{-\gamma_0-\f34}\|G_\Phi\|_{H^{4,0}_\Psi}.
\end{align*}

Whereas it follows from \eqref{S7eq20d} and \eqref{S6eq1} that
\begin{align*}
\|\w{t}^{-1}v\p_y^2({y\vf})\|_{H^{3,0}_\Psi}\leq& C\|v\|_{L^\infty_{\v,\f12\Psi}(H^{3}_\h)}\|\w{t}^{-1}\p_y^2({y\vf})\|_{H^{3,0}_{\f12\Psi}}\\
\leq& C\t^{-\f14}\|G\|_{H^{4,0}_{\Psi}}\|\pa_y G\|_{H^{3,0}_\Psi}
\leq C\e\t^{-\gamma_0-\f34}\|G_\Phi\|_{H^{4,0}_\Psi}.
\end{align*}

Due to $\pa_y v=-\pa_xu,$ we get, by applying \eqref{assume: 1} and \eqref{S8eq3}, \eqref{S8eq2a}, \eqref{S2eq20b}, that
\begin{align*}
\|\pa_y v\pa_y G\|_{H^{3,0}_\Psi}\leq&  C\|\pa_xu\|_{L^\infty_{\v}(H^{3}_\h)}\|\pa_y G\|_{H^{3,0}_\Psi}
\leq  C\t^\f14\|\pa_yG\|_{H^{4,0}_{\Psi}}\|\pa_y G\|_{H^{3,0}_\Psi}\\
\leq&  C\t^\f14\|\pa_yG_\Phi\|_{H^{3,0}_{\Psi}}\|\pa_y G\|_{H^{3,0}_\Psi}\leq C\e\t^{-\gamma_0-\f14}\|\pa_yG_\Phi\|_{H^{3,0}_{\Psi}},\end{align*}
and
\begin{align*}
\w{t}^{-1}\|\pa_y v\pa_y(y\varphi)\|_{H^{3,0}_\Psi}\leq&  C\|\pa_xu\|_{L^\infty_{\v,\f12\Psi}(H^{3}_\h)}\|\w{t}^{-1}\pa_y(y\varphi)\|_{H^{3,0}_{\f12\Psi}}\\
\leq&C \t^{\f14}\|\pa_yG\|_{H^{4,0}_{\Psi}}\|\pa_y G\|_{H^{3,0}_\Psi}
\leq C\e \t^{-\gamma_0-\f14}\|\pa_yG_\Phi\|_{H^{3,0}_{\Psi}}.
\end{align*}

While it follows from \eqref{S3eq8b} that
\begin{align*}
\|\pa_y u\pa_x G\|_{H^{3,0}_\Psi}\leq& C\|\pa_y u\|_{L^\infty_\v(H^{3}_\h)}\|\pa_x G\|_{H^{3,0}_\Psi}
\leq C\e\t^{-\gamma_0-\f34}\|G_\Phi\|_{H^{4,0}_{\Psi}}.
\end{align*}

Finally notice that $v=\pa_x\vf,$ we deduce from \eqref{S6eq1} that
\begin{align*}
\|\f{y}{\w{t}}\left(\p_y u\p_x\vf\right)\|_{H^{3,0}_\Psi}\leq& C\|\f{y}{\w{t}}\pa_y u\|_{H^{3,0}_{\f12\Psi}}\|v\|_{L^\infty_{\v,{\f12\Psi}}(H^{3}_\h)}\\
\leq& C\t^{-\f14}\|\pa_y u\|_{H^{3,0}_{\f34\Psi}}\|G\|_{H^{4,0}_\Psi}
\leq C\e \t^{-\gamma_0-\f34}\|G_\Phi\|_{H^{4,0}_\Psi},
\end{align*}
and
\begin{align*}
\bigl\|\f{1}{\w{t}}\int_y^\infty\left(\p_y u\p_x\vf\right)\,dy'\bigr\|_{H^{3,0}_\Psi}
\leq&C\t^{-\f34}\|e^{-\f12\Psi}\|_{L^2_\v}\|\pa_y u\|_{H^{3,0}_{\f34\Psi}}\|v\|_{L^\infty_{\v,{\f12\Psi}}(H^{3}_\h)}\\
\leq& C\t^{-\gamma_0-\f34}\|G_\Phi\|_{H^{4,0}_\Psi}.
\end{align*}

In view of \eqref{def: F_1}, we conclude the proof of \eqref{S8eq2} by summarizing the above estimates.
\end{proof}

\begin{prop}\label{pro: pa_y^2 G}
{\sl Let $\hbar(t)$ be a non-negative and non-decreasing function on $[0,T]$ with $T\leq T^\star.$ Then   under the assumption  that
$\b\leq2\ga_0-\f1{2},$ for any $\eta>0,$ there exists $C$ so that
\beq \label{S8eq6}
\begin{split}
\|\sqrt{\hbar}&\pa_y^2G(T)\|_{H^{3,0}_\Psi}^2+2\int_0^T\w{t}^{-1}\|\sh \pa_y^2G(t)\|_{H^{3,0}_\Psi}^2\,dt
+\bigl(1-\f\eta4\bigr)\int_0^T\|\sh\pa_y^3G(t)\|_{H^{3,0}_\Psi}^2\,dt\\
\leq&\|\sqrt{\hbar}(0)\pa_y^2G(0)\|_{H^{3,0}_\Psi}^2+
\int_0^T\|\sqrt{\hbar'} \pa_y^2G(t)\|_{H^{3,0}_\Psi}^2\,dt\\
&+C\eta^{-1}\e^\f32\int_0^T\dot\th(t)\Bigl(\|\sh\t^{-\f12}\pa_yG_\Phi\|_{H^{3,0}_\Psi}^2+\|\sh \t^{-1}G_\Phi\|_{H^{4,0}_\Psi}^2\Bigr)\,dt.
\end{split}
\eeq
}
\end{prop}
\begin{proof}
Thanks to \eqref{S3eq7f} and \eqref{S7eq1},
we get, by taking $H^{3,0}_\Psi$ inner product of \eqref{S8eq1} with $\hbar \pa_y^2G$ and using integration by parts, that
\beq\label{S8eq7}
\begin{split}
\f12\f{d}{dt}\|\sh \pa_y^2G(t)\|_{H^{3,0}_\Psi}^2-\f12\|\sqrt{\hbar'} \pa_y^2G(t)\|_{H^{3,0}_\Psi}^2
&+\w{t}^{-1}\|\sh \pa_y^2G(t)\|_{H^{3,0}_\Psi}^2\\
&+\f12\|\sh\pa_y^3 G\|_{H^{3,0}_\Psi}^2\leq \hbar|\langle \frak{F}, \pa_y^3G\rangle_{H^{3,0}_\Psi}|.
\end{split} \eeq
Yet it follows from
 Lemma \ref{lem: F_1}, that
\begin{align*}
\hbar|\langle &\frak{F}, \pa_y^3G\rangle_{H^{3,0}_\Psi}|\leq  \|\sh\frak{F}\|_{H^{3,0}_\Psi}\|\sh\pa_y^3G\|_{H^{3,0}_\Psi}\\
\leq&\f{\eta}{8}\|\sh\pa_y^3G\|_{H^{3,0}_\Psi}^2
+C\eta^{-1}\e^\f32\t^{\beta-2\gamma_0+\f12}\dot\th(t)\Bigl(\|\sh\t^{-\f12}\pa_yG_\Phi\|_{H^{3,0}_\Psi}^2+\|\sh \t^{-1}G_\Phi\|_{H^{4,0}_\Psi}^2\Bigr).
\end{align*}
By substituting the above inequality into \eqref{S8eq7}, we conclude the proof of \eqref{S8eq6}.
\end{proof}

 \subsection{The estimates of $\pa_y^3 G$ }\label{Sect8.2}

To get the decay estimate of $\|\pa_y^2u\|_{L^\infty_\v(H^{\f12+}_\h)}$, which has been used in the proof
of Lemma \ref{S3lem2} (see \eqref{Saeq2a}), according to the proof of
Lemma \ref{lem: phi}, we need the estimates of $\|\pa^3_y G_\Phi\|_{H^{\f12+,0}_\Psi}.$
 To do it, we shall present the estimates of $\|\pa_y^3 G\|_{H^{2,0}_\Psi}$ in this subsection (compared with $\|\pa_y^2 G\|_{H^{3,0}_\Psi}$).

We first get, by taking $\pa_y$ to \eqref{S8eq1}, that
\beq \label{S10eq1}
\begin{split}
\pa_t\pa_y^3 G&-\pa_y^5 G+\w{t}^{-1}\pa_y^3G+\pa^2_y\frak{F}=0.
\end{split} \eeq
It is easy to observe from \eqref{def: F_1} that
\beq \label{def: F_2}
\begin{split}
\pa_y\frak{F}
=&u\pa_x\pa_y^2 G+v\pa_y^3 G+2\pa_y u\pa_x\pa_y G+\pa_y^2 u\pa_x G+\pa_y^2 v\bigl(\pa_y G\\
&-\f1{2\w{t}}\pa_y(y\varphi)\bigr)+\pa_y v\bigl(2\pa_y^2 G-\f1{\w{t}}\pa_y^2(y\varphi)\bigr)-\f1{2\w{t}}v\p_y^3({y\vf})\\
&-\f{2}{\w{t}}\left(\p_y u\p_x\vf\right)-\f{y}{\w{t}}(\pa_y^2u\pa_x\varphi+\pa_y u\pa_x\pa_y \varphi).
\end{split}
\eeq

We first present the  estimate of $\pa_y\frak{F}.$

 \begin{lem}\label{lem: F_2}
{\sl Let  $\pa_y\frak{F}$ be given by  \eqref{def: F_2}. Then for $t<T^\star,$ one has
\beq \label{S8eq2}
\|\pa_y\frak{F}\|_{H^{2,0}_\Psi}\leq C\e\t^{-\gamma_0+\f14}\Bigl(\|\t^{-\f12}\pa_y^2G\|_{H^{3,0}_\Psi}+\| \t^{-1}\pa_yG_\Phi\|_{H^{3,0}_\Psi}+\| \t^{-\f32}G_\Phi\|_{H^{4,0}_\Psi}\Bigr).
\eeq
}
\end{lem}

\begin{proof}
The proof of this lemma follows the same line as that of Lemma \ref{lem: F_1}. We leave the technical details to
the interested readers.
\end{proof}

\begin{prop}\label{pro: pa_y^2 G}
{\sl Let $\hbar(t)$ be a non-negative and non-decreasing function on $[0,T]$ with $T\leq T^\star.$ Then   under the assumption  that
$\b\leq2\ga_0-\f1{2},$ for any $\eta>0,$ there exists $C$ so that
\beq \label{S10eq6}
\begin{split}
\|\sqrt{\hbar}&\pa_y^3G(T)\|_{H^{2,0}_\Psi}^2+2\int_0^T\w{t}^{-1}\|\sh \pa_y^3G(t)\|_{H^{2,0}_\Psi}^2\,dt\\
&+\bigl(1-\f\eta4\bigr)\int_0^T\|\sh\pa_y^4G(t)\|_{H^{2,0}_\Psi}^2\,dt
\leq\|\sqrt{\hbar}(0)\pa_y^3G(0)\|_{H^{2,0}_\Psi}^2\\
&+
\int_0^T\|\sqrt{\hbar'} \pa_y^3G(t)\|_{H^{2,0}_\Psi}^2\,dt+C\eta^{-1}\e^2\int_0^T\|\sh \t^{-1} \pa_y^2G(t)\|_{H^{3,0}_\Psi}^2\,dt\\
&+C\eta^{-1}\e^\f32\int_0^T\dot\th(t)\Bigl(\|\sh\t^{-1}\pa_yG_\Phi\|_{H^{3,0}_\Psi}^2+\|\sh \t^{-\f32}G_\Phi\|_{H^{4,0}_\Psi}^2\Bigr)\,dt.
\end{split}
\eeq
}
\end{prop}
\begin{proof}
Thanks to equation \eqref{S8eq1} and $\pa_y^2 G|_{y=0}=0,$ we find
\begin{align}\label{BC: pa_y^3 G}
(-\pa_y^4G+\pa_y\frak{F})|_{y=0}=0.
\end{align}
Then thanks to \eqref{S3eq7f}, we get, by taking $H^{2,0}_\Psi$ inner product of \eqref{S10eq1} with $\hbar \pa_y^3G$ and using integration by parts, that
\beq\label{S10eq7}
\begin{split}
\f12\f{d}{dt}\|\sh \pa_y^3G(t)\|_{H^{2,0}_\Psi}^2-\f12\|\sqrt{\hbar'} \pa_y^3G(t)\|_{H^{2,0}_\Psi}^2
&+\w{t}^{-1}\|\sh \pa_y^3G(t)\|_{H^{2,0}_\Psi}^2\\
&+\f12\|\sh\pa_y^4 G\|_{H^{2,0}_\Psi}^2\leq \hbar|\langle \pa_y\frak{F}, \pa_y^4G\rangle_{H^{2,0}_\Psi}|.
\end{split} \eeq
Yet it follows from
 Lemma \ref{lem: F_2}, that
\begin{align*}
\hbar|\langle \pa_y\frak{F}, \pa_y^4G\rangle_{H^{2,0}_\Psi}|\leq& \|\sh \pa_y\frak{F}\|_{H^{2,0}_\Psi}\|\sh\pa_y^4G\|_{H^{2,0}_\Psi}\\
\leq&\f{\eta}{8}\|\sh\pa_y^4G\|_{H^{2,0}_\Psi}^2+C\eta^{-1}\e^2\t^{-2\gamma_0+\f32}\|\sh \t^{-1}\pa_y^2 G\|_{H^{3,0}_\Psi}^2\\
&+C\eta^{-1}\e^\f32\t^{\beta-2\gamma_0+\f12}\dot\th(t)\Bigl(\| \sh\t^{-1}\pa_yG_\Phi\|_{H^{3,0}_\Psi}^2+\| \sh\t^{-\f32}G_\Phi\|_{H^{4,0}_\Psi}^2\Bigr).
\end{align*}
By substituting the above inequality into \eqref{S10eq7} and integrating the resulting inequality over $[0,T],$ we conclude the proof of \eqref{S10eq6}.
\end{proof}

\subsection{The proof of Proposition \ref{S1prop4}}

\begin{proof}[Proof of Proposition \ref{S1prop4}]
Taking $\hbar(t)=\t^\f{9-\eta}{2}$ in \eqref{S8eq6} leads to
\begin{align*}
\|\t^{\f{9-\eta}{4}}&\pa_y^2G(T)\|_{H^{3,0}_\Psi}^2
+\bigl(1-\f\eta4\bigr)\int_0^T\|\t^{\f{9-\eta}{4}}\pa_y^3G(t)\|_{H^{3,0}_\Psi}^2\,dt\\
\leq&\|\pa_y^2G(0)\|_{H^{3,0}_\Psi}^2+ \f{9-\eta}2 \int_0^T\|\t^{\f{7-\eta}{4}} \pa_y^2G(t)\|_{H^{3,0}_\Psi}^2\,dt\\
&+C\eta^{-1}\e^\f32\int_0^T\dot\th(t)\Bigl(\|\t^{\f{7-\eta}{4}}\pa_yG_\Phi\|_{H^{3,0}_\Psi}^2+
\|\t^{\f{5-\eta}{4}}G_\Phi\|_{H^{4,0}_\Psi}^2\Bigr)\,dt.
\end{align*}
While it follows from Proposition \ref{S1prop3} that if $\la\geq C\e^\f12\bigl(1+\e\eta^{-1}\bigr)$
\begin{align*}
&C\eta^{-1}\e^\f32\int_0^T\dot\th(t)(\|\t^{\f{7-\eta}{4}}\pa_yG_\Phi\|_{H^{3,0}_\Psi}^2+\|\t^{\f{5-\eta}{4}}G_\Phi\|_{H^{4,0}_\Psi}^2)dt'+C\int_0^T\|\t^{\f{7-\eta}{4}} \pa_y^2G(t)\|_{H^{3,0}_\Psi}^2dt'\\
&\leq C\eta^{-1}(\|G_\Phi(0)\|_{H^{4,0}_\Psi}^2+\|\pa_yG_\Phi(0)\|_{H^{3,0}_\Psi}^2)+C\eta^{-1}\e^\f12\int_0^T\dot\th(t)\|\t^{\f{1-\eta}{4}}u_\Phi\|_{H^{\f{23}{4},0}_\Psi}^2dt',
\end{align*}
which leads to \eqref{S11eq7}.

While taking $\hbar(t)=\t^\f{11-\eta}{2}$ in \eqref{S10eq6} leads to
\beq \label{S10eq8}
\begin{split}
\|&\t^{\f{11-\eta}{4}}\pa_y^3G(T)\|_{H^{2,0}_\Psi}^2
+\bigl(1-\f\eta4\bigr)\int_0^T\|\t^{\f{11-\eta}{4}}\pa_y^4G(t)\|_{H^{2,0}_\Psi}^2\,dt\\
\leq&\|\pa_y^3G(0)\|_{H^{2,0}_\Psi}^2+ \f{11-\eta}2 \int_0^T\|\t^{\f{9-\eta}{4}} \pa_y^3G(t)\|_{H^{3,0}_\Psi}^2\,dt\\
&+C\eta^{-1}\e^2\int_0^T\|\t^{\f{7-\eta}{4}}\ \pa_y^2G_\Phi(t)\|_{H^{3,0}_\Psi}^2\,dt\\
&+C\eta^{-1}\e^\f32\int_0^T\dot\th(t)\Bigl(\|\t^{\f{7-\eta}{4}}\pa_yG_\Phi\|_{H^{3,0}_\Psi}^2+
\|\t^{\f{5-\eta}{4}}G_\Phi\|_{H^{4,0}_\Psi}^2\Bigr)\,dt.
\end{split} \eeq
While  if $\la\geq C\e^\f12\bigl(1+\e\eta^{-1}\bigr)$ and
$\eta\geq\e^2,$  it follows from  \eqref{S6eq4}, \eqref{S7eq4} and  \eqref{S11eq7} that
\begin{align*}
&C\eta^{-1}\e^\f32\int_0^T\dot\th(t)(\|\t^{\f{7-\eta}{4}}\pa_yG_\Phi\|_{H^{3,0}_\Psi}^2+\|\t^{\f{5-\eta}{4}}G_\Phi\|_{H^{4,0}_\Psi}^2)dt'\\
&\quad+C\int_0^T\|\t^{\f{9-\eta}{4}} \pa_y^3G(t)\|_{H^{3,0}_\Psi}^2dt'+C\eta^{-1}\e^2\int_0^T\|\t^{\f{7-\eta}{4}}\ \pa_y^2 G_\Phi(t)\|_{H^{3,0}_\Psi}^2\,dt\\
&\leq C\eta^{-1}(\|G_\Phi(0)\|_{H^{4,0}_\Psi}^2+\|\pa_yG_\Phi(0)\|_{H^{3,0}_\Psi}^2+\|\pa_y^2G(0)\|_{H^{3,0}_\Psi}^2)\\
&\quad+C\eta^{-1}\e^\f12\int_0^T\dot\th(t)\|\t^{\f{1-\eta}{4}}u_\Phi\|_{H^{\f{23}{4},0}_\Psi}^2dt'.
\end{align*}
 By inserting the above estimate into \eqref{S10eq8}, we achieve \eqref{S8eq7pq}. This completes the proof of
 Proposition \ref{S1prop4}.
\end{proof}

\appendix

\setcounter{equation}{0}
\section{The proof of Lemma \ref{S3lem2}}\label{appa}

The goal of this section is to present the proof of Lemma \ref{S3lem2}.

\begin{proof}[Proof of Lemma \ref{S3lem2}] We decompose the proof of this lemma into the following steps:

\no$\bullet$ \underline{The estimate of $E_4.$} We first observe from \eqref{S3eq-1} that
\begin{align*}
\pa_t\left(\f{1}{\dot\th}\right)=\beta \e^{-\f12}\t^{\beta-1}.
\end{align*}
Then in view of \eqref{S3eq5}, we get, by applying \eqref{est: int_y^infty H} and \eqref{assume: 2}, that
\begin{align*}
|E_4|\leq& \beta\e^{-\f12}\int_0^T \w{t}^{\beta-\f34 }\|\pa_y u\|_{L^2_\v(H^{\f12+}_\h)}\|\sh H\|_{H^{7,0}_\Psi}\|\sh \phi\|_{H^{\f{15}{2},0}_\Psi}\,dt\\
\leq&C\int_0^T
\w{t}^{2\beta-\left({\gamma_0}+\f54\right)}
\dot\th(t)\bigl(\|\sh  H(t)\|_{H^{7,0}_\Psi}^2+
\|\sh \phi(t)\|_{H^{\f{15}{2},0}_\Psi}^2\bigr)\,dt, \end{align*}
which together with the fact: $\beta\leq\f12\left({\gamma_0}+\f54\right),$ ensures that
\begin{align}\label{Saeq1}
|E_4|\leq C\int_0^T
\dot\th(t)\bigl(\|\sh H(t)\|_{H^{7,0}_\Psi}^2+
\|\sh \phi(t)\|_{H^{\f{15}{2},0}_\Psi}^2\bigr)\,dt.
\end{align}

\no$\bullet$ \underline{The estimate of $ E_5.$} In view of \eqref{S1eq1a}, we have
\begin{align*}
\pa_t\left(\f{\d(t)}{\dot\th(t)}\right)=\d(t) \pa_t\left(\f{1}{\dot\th(t)}\right)-\la.
\end{align*}
We get, by a similar derivation of \eqref{Saeq1}, that
\begin{align*}
\bigl|\int_0^T\hbar(t) \d(t)\pa_t\left(\f{1}{\dot\th}\right)&\bigl\langle T^\h_{\pa_yD_xu}\Lambda(D_x)\pa_x\int_y^\infty Hdz,\phi\bigr\rangle_{H^{\f{27}{4},0}_\Psi} \, dt\bigr|\\
&\qquad\leq C\int_0^T
\dot\th(t)\bigl(\|\sh H(t)\|_{H^{7,0}_\Psi}^2+
\|\sh \phi(t)\|_{H^{\f{15}{2},0}_\Psi}^2\bigr)\,dt.
\end{align*}
While it follows from  \eqref{assume: 2}, and \eqref{est: int_y^infty H} that
\begin{align*}
&\la\int_0^T\hbar\bigl|\langle T^\h_{\pa_yD_xu}\Lambda(D_x)\pa_x\int_y^\infty Hdz,\phi\bigr\rangle_{H^{\f{27}{4},0}_\Psi}\bigr|\,dt\\
&\leq C\la\int_0^T
\|\pa_yu\|_{L^2_y(H^{\f32+}_x)}\t^\f14\|\sh H(t)\|_{H^{7,0}_\Psi}
\|\sh \phi(t)\|_{H^{\f{15}{2},0}_\Psi}\,dt\\
&\leq C\e\la\int_0^T
\t^{-\gamma_0-\f14}\|\sh H(t)\|_{H^{7,0}_\Psi}
\|\sh \phi(t)\|_{H^{\f{15}{2},0}_\Psi}\,dt,
\end{align*}
which together with \eqref{S3eq-1} and $\beta\leq \gamma_0+\f14$ ensures that
\begin{align*}
&\la\int_0^T\hbar\bigl|\langle T^\h_{\pa_yD_xu}\Lambda(D_x)\pa_x\int_y^\infty Hdz,\phi\bigr\rangle_{H^{\f{27}{4},0}_\Psi}\bigr|\,dt\\
&\leq C\e^{\f12} \la\int_0^T\t^{\beta-\gamma_0-\f14}\dot\th(t)\bigl(\|\sh H(t)\|_{H^{7,0}_\Psi}^2+
\|\sh \phi(t)\|_{H^{\f{15}{2},0}_\Psi}^2\bigr)dt\\
&\leq C\e^{\f12}\la\int_0^T\dot\th(t)\bigl(\|\sh H(t')\|_{H^{7,0}_\Psi}^2+
\|\sh \phi(t')\|_{H^{\f{15}{2},0}_\Psi}^2\bigr)dt.
\end{align*}

Overall, if $\beta\leq \min\bigl\{\beta+\f14, \f12\left({\gamma_0}+\f54\right)\bigr\}$,
we deduce from that
\begin{align}\label{Saeq2}
|E_5|
&\leq C\bigl(1+\e^\f12\la\bigr)\int_0^T\dot\th(t)\bigl(\|\sh H(t)\|_{H^{7,0}_\Psi}^2+
\|\sh \phi(t)\|_{H^{\f{15}{2},0}_\Psi}^2\bigr)dt.
\end{align}

\no$\bullet$\underline{The estimate of $E_6.$}
In view of \eqref{defcL}, we get, by a  straight calculation, that
\begin{align*}
 [T_{\pa_y u}^\h,\mathcal{L}]f=&-T^\h_{\pa_t\pa_yu-\pa_y^3u}f+\la\dot\th[T^\h_{\pa_y u},\D^\f12]f+[T^\h_{\pa_y u},T^\h_u\pa_x]f\\
 &+[T^\h_{\pa_y u}; T^\h_v\pa_y]f+\f{\d(t)}2[T^\h_{\pa_y u}; T^\h_{D_xu}\Lambda(D)\pa_x]f-2T^\h_{\pa_y^2u}\pa_yf .
\end{align*}
Yet notice from \eqref{eq: tu} that
\begin{align*}
\pa_t\pa_yu-\pa_y^3u=-\pa_y(u\pa_xu+v\pa_y u)=-u\pa_x\pa_y u-v\pa_y^2u,
\end{align*}
from which and \eqref{est: int_y^infty H}, we infer
\begin{align*}
\bigl\|T^\h_{\pa_t\pa_yu-\pa_y^3u}&\pa_x\int_y^\infty H\,dz\bigr\|_{H^{6,0}_\Psi}\leq C\t^\f14\|(u\pa_x\pa_y u-v\pa_y^2u)\|_{L^2_\v(H^{\f12+}_\h)}\|H\|_{H^{7,0}_\Psi}\\
\leq& C\t^\f14\Bigl(\|u\|_{L^\infty_\v(H^{\f12+}_\h)}\|\pa_y u\|_{H^{\f32+,0}}+\|v\|_{L^\infty_\v(H^{\f12+}_\h)}\|\pa_y^2u\|_{H^{\f12+,0}}\Bigr)\|H\|_{H^{7,0}_\Psi},
\end{align*}
which together with \eqref{assume: 2},  \eqref{S3eq8} and \eqref{S3eq-1} ensures that
\begin{align*}
\bigl\|T^\h_{\pa_t\pa_yu-\pa_y^3u}\pa_x\int_y^\infty H\,dz\bigr\|_{H^{6,0}_\Psi}\leq&C\e^2\t^{-2\gamma_0-\f12}\|H\|_{H^{7,0}_\Psi} \\
\leq& C\e\t^{2\beta-2\gamma_0-\f12}\dot\th^2(t)\|H\|_{H^{7,0}_\Psi}.
\end{align*}

Whereas it follows from  Lemma \ref{lem: com1} and \eqref{est: int_y^infty H}, that
\begin{align*}
\bigl\|[T^\h_{\pa_y u}; T^\h_u\pa_x]\pa_x\int_y^\infty H\,dz\bigr\|_{H^{6,0}_\Psi}\leq &C\t^\f14\|\pa_yu\|_{L^2_\v(H^{\f32+}_\h)}\|u\|_{L^\infty_\v(H^{\f32+}_\h)}\|H\|_{H^{7,0}_\Psi}\\
\leq&C\e^2\t^{-2\gamma_0-\f12}\|H\|_{H^{7,0}_\Psi}\\
\leq& C\e\t^{2\beta-2\gamma_0-\f12}\dot\th^2(t)\|H\|_{H^{7,0}_\Psi}.
\end{align*}
Along the same line, due to $\pa_y v=-\pa_x u,$  Lemma \ref{lem: phi} and \eqref{assume: 2} , we have
 \begin{align*}
\bigl\|[T^\h_{\pa_y u},&T_v^\h \pa_y]\pa_x\int_y^\infty H\,dz\bigr\|_{H^{6,0}_\Psi}\leq
 \|[T^\h_{\pa_y u},T^\h_v]\pa_x H\|_{H^{6,0}_\Psi}+\|T^\h_{\pa_y u}T^\h_{\pa_yv}\pa_x\int_y^\infty H\,dz\|_{H^{6,0}_\Psi}\\
\leq&C\t^\f14\Big(\|\pa_yu\|_{L^2_\v(H^{\f32+}_\h)}\|v\|_{L^\infty_\v(H^{\f32+}_{\h})}+\|\pa_yu\|_{L^2_\v(H^{\f32+}_\h)}
\|u\|_{L^\infty_\v(H^{\f32+}_\h)}\Big)\|H\|_{H^{7,0}_\Psi},\\
\leq&C\e^2\t^{-2\gamma_0-\f12}\|H\|_{H^{7,0}_\Psi}\\
\leq& C\e\t^{2\beta-2\gamma_0-\f12}\dot\th^2(t)\|H\|_{H^{7,0}_\Psi},
\end{align*}
and
 \begin{align*}
\f{\d(t)}2\bigl\|[T_{\pa_y u}^\h; T^\h_{D_xu}\Lambda(D)\pa_x]\pa_x\int_y^\infty H\,dz\bigr\|_{H^{6,0}_\Psi}\leq&C\t^\f14\|\pa_yu\|_{L^2_\v(H^{\f32+}_\h)}\|u\|_{L^\infty_\v(H^{\f52+}_\h)}\|H\|_{H^{7,0}_\Psi},\\
\leq&C\e\t^{2\beta-2\gamma_0-\f12}\dot\th^2(t)\|H\|_{H^{7,0}_\Psi},
\end{align*}
 and
 \beq\label{Saeq2a}
 \begin{split}
\|T^\h_{\pa_y^2u}\pa_xH\|_{H^{6,0}_\Psi}\leq&C\|\pa_y^2u\|_{L^\infty_\v(H^{\f12+}_\h)}\|H\|_{H^{7,0}_\Psi}\leq C\e \t^{-\gamma_0-\f54}\|H\|_{H^{7,0}_\Psi} \\
\leq& C \t^{2\beta-\gamma_0-\f54}\dot\th^2(t)\|H\|_{H^{7,0}_\Psi} .
\end{split} \eeq

By summarizing  the above estimates, we get for $\beta\leq\min\{\gamma_0+\f14,\f12(\gamma_0+\f54)\}$ that
\begin{align}\label{Saeq3}
\bigl\| [T^\h_{\pa_y u}; \mathcal{L}-\la\dot\th\D^\f12]\pa_x\int_0^y H\,dz\bigr\|_{H^{6,0}_\Psi}
\leq C \dot\th^2(t)
\|H\|_{H^{7,0}_\Psi}.
\end{align}

On the other hand, we deduce from Lemma \ref{lem: com1},   \eqref{assume: 2} and \eqref{est: int_y^infty H} that
\beq\label{Saeq4}
\begin{split}
\la\dot\th\bigl\|[T^\h_{\pa_y u},\D^\f12]\pa_x\int_y^\infty H\,dz\bigr\|_{H^{\f{25}{4},0}_\Psi}\leq& C\la\dot\th\t^\f14\|\pa_yu\|_{L^2_\v(H^{\f32+}_\h)}\|H\|_{H^{\f{27}{4},0}_\Psi}\\
\leq& C\e \la\t^{-\gamma_0-\f14}\dot\th \|H\|_{H^{\f{27}{4},0}_\Psi}\\
\leq&C\e^\f12\la\t^{\beta-\gamma_0-\f14}\dot\th^2 \|H\|_{H^{\f{27}{4},0}_\Psi}.
\end{split} \eeq

Therefore, for $\beta\leq\min\bigl\{\gamma_0+\f14,\f12(\gamma_0+\f54)\bigr\},$ we get, by summing up
\eqref{Saeq3} and \eqref{Saeq4}, that
\beq \label{Saeq5}
\begin{split}
|E_6|\leq&\int_0^T\f{\hbar(t)}{\dot\th(t)}\bigl\|[T^\h_{\pa_y u}; \mathcal{L}-\la\dot\th\D^\f12]\pa_x\int_y^\infty H\,dz\bigr\|_{H^{6,0}_\Psi}\|\phi\|_{H^{\f{15}{2},0}_\Psi}\,dt\\
&+\la\int_0^T\hbar(t)\bigl\|[T^\h_{\pa_y u};\D^\f12]\pa_x\int_y^\infty H\,dz\bigr\|_{H^{\f{25}{4},0}_\Psi}\|\phi\|_{H^{\f{29}{4}}_\Psi}\,dt\\
\leq&  C\bigl(1+\e^\f12\la\bigr)\int_0^T \dot\th(t)\Big(\|\sh H(t)\|_{H^{7,0}_\Psi}^2+\|\sh \phi(t)\|_{H^{\f{15}{2},0}_\Psi}^2\Big)\,dt.
\end{split} \eeq
$E_{10}$ shares the above estimate.

\no$\bullet$\underline{The estimates of $E_7,E_8$ and $E_9.$}
By applying Lemma \ref{lem: T_fg} and   \eqref{est: int_y^infty H}, we obtain
\begin{align*}
|E_7| \leq& C\int_0^T\f{\hbar}{\dot\th}\t^\f14\|\pa_y u\|_{L^2_\v(H^{\f12+}_\h)}\|\pa_x u\|_{L^\infty_\v(H^{\f12+}_\h)}\|H\|_{H^{7,0}_\Psi}\|\phi\|_{H^{\f{15}{2},0}_\Psi}\,dt,
\end{align*}
which together with \eqref{assume: 2} and \eqref{S3eq8} ensures that
\begin{align*}
|E_7|
\leq&C\e \int_0^T \t^{2\beta-2\gamma_0-\f12}\dot\th(t)\Big(\|\sh H(t)\|_{H^{7,0}_\Psi}^2+
\|\sh \phi(t)\|_{H^{\f{15}{2},0}_\Psi}^2\Big)\,dt\\
\leq&C\e\int_0^T\dot\th(t)\Big(\|\sh H(t)\|_{H^{7,0}_\Psi}^2+\|\sh \phi(t)\|_{H^{\f{15}{2},0}_\Psi}^2\Big)\,dt,
\end{align*}
if $\b\leq \ga_0+\f14.$

Similarly,  if $\b\leq \ga_0+\f14,$  we have
\begin{align*}
|E_{8}|\leq&C\int_0^T\f{\hbar}{\dot\th} \t^{\f14}\|\pa_y u\|_{L^\infty_\v(H^{\f12+}_\h)} \|\pa_xv\|_{L^\infty_\v(H^{\f12+}_\h)}\|H\|_{H^{7,0}_\Psi}\|\phi\|_{H^{\f{15}{2},0}_\Psi}\,dt\\
\leq&C\e \int_0^T  \t^{2\beta-2\gamma_0-\f12} \dot\th(t)\Big(\|\sh H(t)\|_{H^{7,0}_\Psi}^2+\|\sh \phi(t)\|_{H^{\f{15}{2},0}_\Psi}^2\Big)\,dt\\
\leq&C\e\int_0^T\dot\th(t)\Big(\|\sh H(t)\|_{H^{7,0}_\Psi}^2+\|\sh \phi(t)\|_{H^{\f{15}{2},0}_\Psi}^2\Big)\,dt,
\end{align*}
and
\begin{align*}
|E_{9}|\leq& C\int_0^T\f{\d\hbar}{\dot\th}\t^\f14\|u\|_{L^\infty_\v(H^{\f52+}_\h)}\|\pa_y u\|_{L^2_\v(H^{\f12+}_\h)}\|H\|_{H^{7,0}_\Psi}\|\phi\|_{H^{\f{15}{2},0}_\Psi}\,dt\\
\leq&C\e\int_0^T\t^{2\beta-2\gamma_0-\f12}\dot\th\Bigl(\|\sh H\|_{H^{7,0}_\Psi}^2+ \|\sh \phi\|_{H^{\f{15}{2},0}_\Psi}^2 \Bigr)\,dt\\
\leq&C\e\int_0^T\dot\th(t)\Bigl(\|\sh H(t)\|_{H^{7,0}_\Psi}^2+\|\sh \phi(t)\|_{H^{\f{15}{2},0}_\Psi}^2\Bigr)\,dt.
\end{align*}

As a result, it comes out
\beq \sum_{i=7}^9|E_i|\leq C\e\int_0^T\dot\th(t)\Bigl(\|\sh H(t')\|_{H^{7,0}_\Psi}^2+\|\sh \phi(t)\|_{H^{\f{15}{2},0}_\Psi}^2\Bigr)\,dt.
\label{Saeq6}
\eeq
$\sum_{i=11}^{13}|E_i|$ shares the above estimate.

By summing up the estimates \eqref{Saeq1}, \eqref{Saeq2} and \eqref{Saeq5}, we obtain \eqref{S3eq7b},
and \eqref{Saeq6} implies \eqref{S3eq7c}. This concludes the proof of Lemma \ref{S3lem2}.
\end{proof}

\setcounter{equation}{0}
\section{The proof of Lemma \ref{S4lem2}}\label{appb}

\begin{proof}[Proof of Lemma \ref{S4lem2}]
 Let us  first compute $\mathcal{L}^*u_\Phi$.
Indeed we deduce from \eqref{defcL*} and the equation of $u_\Phi$ in \eqref{eq: tu_Phi} that
\begin{align*}
\mathcal{L}^*u_\Phi=&2\la\dot\th\D^\f12u_\Phi+\bigl(T^\h_u\pa_x-\pa_x(T^\h_u)^*\bigr)u_\Phi
+\bigl(T_v^\h\pa_y-\bigl(\pa_y+\f{y}{2\t}\bigr)(T^\h_v)^*\bigr)u_\Phi\\
&+\f{\d(t)}2\bigl(T^\h_{D_xu}\Lambda(D_x)\pa_x-\Lambda(D_x)\pa_x(T^\h_{D_x u})^*\bigr)u_\Phi+T^\h_{\pa_y u}v_\Phi\\
&+\f{\d(t)}2T^\h_{\pa_yD_xu}\Lambda(D_x)v_\Phi-\Bigl(\pa_y^2 +\bigl(\pa_y+\f{y}{2\t}\bigr)^2-\f{y^2}{4\t^2}\Bigr)u_\Phi-f,
\end{align*}
and
\begin{align*}
\Bigl(\pa_y^2 +\bigl(\pa_y+\f{y}{2\t}\bigr)^2-\f{y^2}{4\t^2}\Bigr)u_\Phi=2\pa_y^2u_\Phi+\f{y}{\t}\pa_yu_\Phi+\f{1}{2\t}u_\Phi.
\end{align*}
Then due to  $u|_{y=0}=u|_{y\to+\infty}=0,$ we get, by using
integration by parts, that
\beq\label{Sbeq1}
\begin{split}
B_1=&\e\int_0^T\f{\w{t}^{-\f14-\gamma_0}}{\dot\th(t)}\hbar\langle H,\mathcal{L}^*u_\Phi\rangle_{H^{6,0}_\Psi}\,dt
-\e\int_0^T \pa_{t}\Bigl(\f{\w{t}^{-\f14-\gamma_0}}{\dot\th(t)}\Bigr)\hbar\langle H,u_\Phi\rangle_{H^{6,0}_\Psi}\,dt\\
&+\e\f{\hbar\w{t}^{-\f14-\gamma_0}}{\dot\th(t)}\langle H(t),u_\Phi(t)\rangle_{H^{6,0}_\Psi}|_{t=0}^{t=T}\\
=&\e\f{\hbar(T)\langle T\rangle^{-\f14-\gamma_0}}{\dot\th(T)}\langle H(T),u_\Phi(T)\rangle_{H^{6,0}_\Psi}-\e\int_0^T \pa_{t}\Bigl(\f{\w{t}^{-\f14-\gamma_0}}{\dot\th(t)}\Bigr)\hbar\langle H,u_\Phi\rangle_{H^{6,0}_\Psi}\,dt\\
&+2\e\la\int_0^T\w{t}^{-\f14-\gamma_0}\hbar\langle H,\D^\f12u_\Phi\rangle_{H^{6,0}_\Psi}\,dt
+\e\int_0^T\f{\w{t}^{-\f14-\gamma_0}\hbar}{\dot\th(t)}\\
&\times\Bigl(\bigl\langle H,(T^\h_u\pa_x-\pa_x(T^\h_u)^*)u_\Phi\bigr\rangle_{H^{6,0}_\Psi}+\bigl\langle H,(T_v^\h\pa_y-\pa_y(T^\h_v)^*)u_\Phi\bigr\rangle_{H^{6,0}_\Psi}\\
&\quad-\bigl\langle H,\f{y}{2\w{t}}(T_v^\h)^*u_\Phi\bigr\rangle_{H^{6,0}_\Psi}+\f{\d(t)}2\bigl\langle H,(T^\h_{D_xu}\Lambda(D_x)\pa_x-\Lambda(D_x)\pa_x(T^\h_{D_x u})^*)u_\Phi\bigr\rangle_{H^{6,0}_\Psi}
\\
&\quad+\langle H,T_{\pa_y u}^\h v_\Phi\rangle_{H^{6,0}_\Psi}+\f{\d(t)}2\bigl\langle H,T^\h_{\pa_yD_xu}\Lambda(D_x)v_\Phi\bigr\rangle_{H^{6,0}_\Psi}
-2\langle H,\pa_y^2 u_\Phi\rangle_{H^{6,0}_\Psi}\\
&\quad-\bigl\langle H,\f{y}{\w{t}}\pa_yu_\Phi\bigr\rangle_{H^{6,0}_\Psi}-\f{1}{2\w{t}}\langle H, u_\Phi\rangle_{H^{6,0}_\Psi}-\langle H,f\rangle_{H^{6,0}_\Psi}\Bigr)\,dt
\eqdefa B_{1}^1+\cdots+B_{1}^{13}.
\end{split}\eeq

Next let us handle the  estimates of  $B_1^{i}, i=1,\cdots, 13,$ term by term.\smallskip

\no$\bullet$\underline{The estimate of $B_1^1.$} For any $\zeta>0,$
by applying Young's inequality, we get for $\beta\leq \gamma_0+\f14$ that
\begin{align*}
|B_{1}^1|\leq& \zeta\|\sh u_\Phi(T)\|_{H^{\f{11}{2},0}_\Psi}^2+C\zeta^{-1}\e\langle T\rangle^{2\beta-2\gamma_0-\f12} \|\sh H(T)\|_{H^{\f{27}{4},0}_\Psi}^2\\
\leq& \zeta\|\sh u_\Phi(T)\|_{H^{\f{11}{2},0}_\Psi}^2+C\zeta^{-1}\e\|\sh H(T)\|_{H^{\f{27}{4},0}_\Psi}^2.
\end{align*}

\no$\bullet$\underline{The estimate of $B_1^2.$}
In view of \eqref{S3eq-1}, we have
\begin{align*}
\e\pa_t\Bigl(\f{\t^{-\f14-\gamma_0}}{\dot\th(t)}\Bigr)=\bigl(\beta-\gamma_0-\f14\bigr)\e^{\f12}\t^{\beta-\gamma_0-\f54}
=\bigl(\beta-\gamma_0-\f14\bigr)\dot\th(t)\t^{2\beta-\gamma_0-\f54}.
\end{align*}
Applying H\"older inequality and using $\beta\leq \f{1}{2}(\gamma_0+\f54)$  yields
\begin{align*}
|B_1^2|\leq&C\int_0^T\t^{2\beta-\gamma_0-\f54}\dot\th(t)\Bigl(\|\sh u_\Phi\|_{H^{\f{23}{4},0}_\Psi}^2+\|\sh H\|_{H^{\f{25}{4},0}_\Psi}^2\Bigr) \,dt\\
\leq& C\int_0^T\dot\th(t)\Bigl(\|\sh u_\Phi\|_{H^{\f{23}{4},0}_\Psi}^2+\|\sh H\|_{H^{\f{25}{4},0}_\Psi}^2\Bigr)\, dt.
\end{align*}

\no$\bullet$\underline{The estimate of $B_1^3.$}
Due to $\beta\leq \gamma_0+\f14,$ we get, by using Young's inequality, that
 \begin{align*}
 |B_{1}^3|\leq& C\e^\f12\la\int_0^T \t^{\beta-\gamma_0-\f14}\dot\th(t)\|\sh u_\Phi\|_{H^{\f{23}{4},0}_\Psi}\|\sh H\|_{H^{7,0}_\Psi}\,dt\\
 \leq& C\e^\f12\la\int_0^T \dot\th(t)\Bigl(\|\sh u_\Phi\|_{H^{\f{23}{4},0}_\Psi}^2+\|\sh H\|_{H^{7,0}_\Psi}^2\Bigr)\,dt.
 \end{align*}

\no$\bullet$\underline{The estimate of $B_1^4.$}
As $\beta\leq \gamma_0+\f14,$ we get, by applying Lemma \ref{lem: T_fg}, that
\begin{align*}
|B_{1}^4|\leq&C\e^{\f12}\int_0^T\|\sh H\|_{H^{7,0}_\Psi}(\|\sh T_{u}^\h\pa_xu_\Phi\|_{H^{5,0}_\Psi}
+\|\sh \pa_x(T^\h_{u})^*u_\Phi\|_{H^{5,0}_\Psi})\,dt\\
\leq&C\e^{\f12}\int_0^T\|\sh H\|_{H^{7,0}_\Psi}\|u\|_{L^\infty_\v(H^{\f12+}_\h)}\|\sh u_\Phi\|_{H^{6,0}_\Psi}\,dt,
\end{align*}
which together with \eqref{S3eq8} implies that
\begin{align*}
|B_{1}^4|
\leq&C\e^{\f32}\int_0^T\t^{-\gamma_0-\f14}\|\sh H\|_{H^{7,0}_\Psi}\|\sh u_\Phi\|_{H^{6,0}_\Psi}\,dt\\
\leq&C\e^{\f32}\zeta^{-1}\int_0^T\dot\th(t)\|\sh H\|_{H^{7,0}_\Psi}^2\,dt+ \f\zeta6\e\int_0^T\t^{-\gamma_0-\f14}\|u_\Phi\|_{H^{6,0}_\Psi}^2\,dt.
\end{align*}

\no$\bullet$\underline{The estimate of $B_1^5.$}
Applying Lebnitz formula gives $$T^\h_v\pa_yu-\pa_y (T^\h_v)^*u=(T^\h_v-(T^\h_v)^*)\pa_yu-(T^\h_{\pa_y v})^*u,$$
from which and Lemmas \ref{lem: T_fg} and \ref{lem: com1}, we infer
\begin{align*}
|B_{1}^5|\leq&C\e^{\f12}\int_0^T\t^{\beta-\gamma_0-\f14}\|\sh H\|_{H^{7,0}_\Psi}\Big(\|\sh (T^\h_v-(T^\h_v)^*)\pa_yu_\Phi\|_{H^{5,0}_\Psi}+\|\sh (T^\h_{\pa_y v})^*u_\Phi\|_{H^{5,0}_\Psi}\Big)\,dt\\
\leq&C\e^{\f12}\int_0^T\t^{\beta-\gamma_0-\f14}\|\sh H\|_{H^{7,0}_\Psi}\\
&\qquad\qquad\times \Big(\|v\|_{L^\infty_\v(H^{\f12+}_\h)}\|\sh \pa_yu_\Phi\|_{H^{5,0}_\Psi}
+\|\pa_y v\|_{L^\infty_\v(H^{\f12+}_\h)}\|\sh u_\Phi\|_{H^{5,0}_\Psi}\Big)\,dt,
\end{align*}
which together with \eqref{S3eq8} ensures that
\begin{align*}
|B_{1}^5|
\leq&C\e^{\f32}\int_0^T\t^{\beta-2\gamma_0}\Bigl(\t^{-\f12}\|\sh H\|_{H^{7,0}_\Psi}\|\sh u_\Phi\|_{H^{6,0}_\Psi}+\|\sh H\|_{H^{7,0}_\Psi}\|\sh\pa_yu_\Phi\|_{H^{\f{11}{2},0}_\Psi}\Bigr)\,dt\\
\leq&C\e^{\f54}\zeta^{-1}\int_0^T\dot\th(t)\|\sh H\|_{H^{7,0}_\Psi}^2dt+
 \f\zeta6\e\int_0^T\t^{-\gamma_0-\f14}\|\sh u_\Phi\|_{H^{6,0}_\Psi}^2\,dt\\
&\qquad\qquad\qquad\qquad\qquad\qquad\qquad\qquad+\f\zeta2\e^\f 54\int_0^T\|\sh\pa_yu_\Phi\|_{H^{\f{11}{2},0}_\Psi}^2\,dt,
\end{align*}
if $\beta\leq \min\bigl\{\f43\gamma_0,\gamma_0+\f14\bigr\}.$

\no$\bullet$\underline{The estimate of $B_1^6.$}
We first observe from Lemma \ref{lem2.2} and \eqref{assume: 1} that
 \begin{align*}
\bigl\|\f{y}{2\t}v\bigr\|_{L^\infty_\v(H^{\f12+}_\h)}=&\bigl\|\f{y}{2\t}\int_y^\infty\pa_xu\, dz\bigr\|_{L^\infty_\v(H^{\f12+}_\h)}\\
\leq & C\|\f{y}{2\t^{\f34}}e^{-\f34\Psi}\bigr\|_{L^\infty_\v}\|u\|_{H^{\f32+,0}_{\f34\Psi}}\leq C\t^{-\f14}\|G\|_{H^{\f32+,0}_\Psi}
\leq C\e\t^{-\ga_0-\f14},
\end{align*}
from which, Lemma \ref{lem: T_fg} and $\beta\leq \gamma_0+\f14,$ we infer
\begin{align*}
|B_{1}^6|\leq&C\e^{\f12}\int_0^T \|\sh H(t)\|_{H^{7,0}_\Psi}\bigl\|\f{y}{2\t}v\bigr\|_{L^\infty_\v(H^{\f12+}_\h)}\|\sh u_\Phi\|_{H^{5,0}_\Psi}\,dt\\
\leq&C\e^{\f32}\zeta^{-1}\int_0^T\dot\th(t)\|\sh H\|_{H^{7,0}_\Psi}^2dt+ \f\zeta6\e\int_0^T\t^{-\gamma_0-\f14}\|\sh u_\Phi\|_{H^{6,0}_\Psi}^2\,dt.
\end{align*}

\no$\bullet$\underline{The estimate of $B_1^7.$}
Similar to the estimate of $B_1^4$, we have
 \begin{align*}
 |B_{1}^7|\leq&C\e^\f32\zeta^{-1}\int_0^T\dot\th(t)\|\sh H\|_{H^{7,0}_\Psi}^2\,dt+ \f\zeta6\e\int_0^T\t^{-\gamma_0-\f14}\|\sh u_\Phi\|_{H^{6,0}_\Psi}^2\,dt.
 \end{align*}

\no$\bullet$\underline{The estimates of $B_1^8,~ B_1^9.$}
It follows from  Lemma \ref{lem: T_fg}, \eqref{assume: 2}, \eqref{S4eq3b} and $\beta\leq \gamma_0+\f14$ that
\begin{align*}
|B_1^8|\leq& C\e^{\f12}\int_0^T\|\sh H\|_{H^{7,0}_\Psi}\|\pa_y u\|_{H^{\f12+,0}}\|\sh v_\Phi\|_{L^\infty_{y,\Psi}(H^{5}_\h)} \,dt\\
\leq&C\e^{\f32}\int_0^T\t^{-\gamma_0-\f14}\|\sh H\|_{H^{7,0}_\Psi}\|\sh u_\Phi\|_{H^{6,0}_\Psi} dt\\
\leq& C\e^\f32\zeta^{-1}\int_0^T\dot\th(t)\|\sh H\|_{H^{7,0}_\Psi}^2\,dt+ \f\zeta6\e\int_0^T\t^{-\gamma_0-\f14}\|\sh u_\Phi\|_{H^{6,0}_\Psi}^2\,dt.
\end{align*}

Along the same line, one has
 \begin{align*}
 |B_{1}^9|
 \leq&C\e^\f32\zeta^{-1}\int_0^T\dot\th(t)\|\sh H\|_{H^{7,0}_\Psi}^2dt+ \f\zeta6\e\int_0^T\t^{-\gamma_0-\f14}\|\sh u_\Phi\|_{H^{6,0}_\Psi}^2dt.
 \end{align*}

\no$\bullet$\underline{The estimate of $B_1^{10}.$}
Applying H\"older inequality gives
\begin{align*}
|B_{1}^{10}|\leq&\e^{\f12} \int_0^T\t^{\beta-\gamma_0-\f34}\|\sh H\|_{H^{7,0}_\Psi}\|\sh\t^{\f12}\pa_y ^2u_\Phi\|_{H^{5,0}_\Psi} \,dt\\
\leq & C\zeta^{-1}\e^\f14\int_0^T\t^{3\beta-2\gamma_0-\f32}\dot\th(t)\|\sh H\|_{H^{7,0}_\Psi}^2\,dt+ \zeta\e^\f14\int_0^T\|\sh\t^{\f12}\pa_y^2 u_\Phi\|_{H^{5,0}_\Psi}^2\,dt\\
\leq&C\e^\f14\zeta^{-1}\int_0^T\dot\th(t)\|\sh H\|_{H^{7,0}_\Psi}^2\,dt+ \zeta \e^\f14\int_0^T\|\sh\t^{\f12}\pa_y^2 u_\Phi\|_{H^{5,0}_\Psi}^2\,dt,
\end{align*}
if $\beta\leq \f23 \gamma_0+\f12.$

\no$\bullet$\underline{The estimate of $B_1^{11}.$}
Similar to the estimate of $B_{1}^{10},$ we have
\begin{align*}
|B_{1}^{11}|\leq&\e^{\f12}\int_0^T\t^{\beta-\gamma_0-\f34}\bigl\|\f{y}{\t}\sh H\bigr\|_{H^{\f{27}{4},0}_\Psi}\| \sh \t^\f12\pa_yu_\Phi\|_{H^{\f{21}{4},0}_\Psi} \,dt\\
\leq& \zeta\e^\f14\int_0^T\bigl\|\f{y}{\t}\sh H\bigr\|_{H^{\f{27}{4},0}_\Psi}^2dt+C\e^\f14\zeta^{-1}\int_0^T\dot\th(t)\| \sh \t^\f12\pa_yu_\Phi\|_{H^{\f{21}{4},0}_\Psi}^2 \,dt,
\end{align*}
if $\beta\leq \f23 \gamma_0+\f12.$

\no$\bullet$\underline{Estimate of $B_1^{12}.$} It is easy to observe that if $\beta\leq \f12(\gamma_0+\f54),$
\begin{align*}
|B_{1}^{12}|\leq& \e^{\f12}\int_0^T\t^{\beta-\gamma_0-\f54}\|\sh H\|_{H^{7,0}_\Psi}\|\sh u_\Phi\|_{H^{5,0}_\Psi}\,dt\\
\leq&\int_0^T\dot\th(t)(\|\sh H\|_{H^{7,0}_\Psi}^2+\|\sh u_\Phi\|_{H^{5,0}_\Psi}^2)\,dt.
\end{align*}

\no$\bullet$\underline{The estimate of $B_1^{13}.$}
Finally applying Lemma \ref{lem: f} for $s=5$ gives
\begin{align*}
|B_{1}^{13}|\leq &C\e^\f12\int_0^T\t^{\beta-\gamma_0-\f14}\|\sh H\|_{H^{7,0}_\Psi}\|\sh f\|_{H^{5,0}_\Psi}\,dt\\
\leq&C\e^\f12\int_0^T\t^{\beta-\gamma_0}\|\sh H\|_{H^{7,0}_\Psi}\Bigl(\|\pa_y G_\Phi\|_{H^{\f52+,0}_{\Psi}}\|\sh u_\Phi\|_{H^{5,0}_\Psi}+\|G_\Phi\|_{H^{\f52+,0}_{\Psi}}\|\sh\pa_y u_\Phi\|_{H^{\f92,0}_\Psi}\Bigr)\,dt.
\end{align*}
So that by virtue of \eqref{assume: 1}, for any $\zeta>0,$ we get, by applying Young's inequality, that
\begin{align*}
|B_{1}^{13}|
\leq& C\e\int_0^T\t^{2\beta-2\gamma_0-\f12}\dot\th(t)\|\sh H\|_{H^{7,0}_\Psi}\|\sh u_\Phi\|_{H^{5,0}_\Psi}dt+\zeta\e^{\f54}\int_0^T\|\sh\pa_y u_\Phi\|_{H^{\f{11}2,0}_\Psi}^2 dt\\
&+C\zeta^{-1}\e^{\f54}\int_0^T\t^{3\beta-4\gamma_0}\dot\th(t)\|\sh H\|_{H^{7,0}_\Psi}^2 dt\\
\leq&C(\e+\zeta^{-1}\e^{\f54})\int_0^T\dot\th(t)(\|\sh H\|_{H^{7,0}_\Psi}^2+\|\sh u_\Phi\|_{H^{5,0}_\Psi})dt+\f\zeta2\e^{\f54}\int_0^T\|\sh\pa_y u_\Phi\|_{H^{\f{11}2,0}_\Psi}^2 \,dt,
\end{align*}
if $\beta\leq\min\{\gamma_0+\f14,\f{4}{3}\gamma_0\}. $

By substituting the above estimates into \eqref{Sbeq1}, we conclude the proof \eqref{S4eq3e}.
\end{proof}

\section*{Acknowledgments}

The authors would like to thank  Marius Paicu and Zhifei Zhang for profitable discussions.
 P. Zhang is partially supported
by NSF of China under Grants   11731007 and 11688101,  and  K.C.Wong Education Foundation.

\end{document}